\documentclass[11pt]{amsart} 

\usepackage{amssymb,cite,hyperref}
\usepackage[margin=1in]{geometry}
\usepackage{tensor,enumitem}
\usepackage[all,cmtip]{xy}

\newcommand{\sm}[1]{( \begin{smallmatrix} #1 \end{smallmatrix} )}
\newcommand{\subgrp}[1]{\langle #1 \rangle}
\newcommand{\set}[1]{\left\{ #1 \right\}}
\newcommand{\abs}[1]{\left| #1 \right|}
\newcommand{\bs}[1]{\boldsymbol{#1}}
\newcommand{\wt}[1]{\widetilde{ #1}}
\newcommand{\ul}[1]{\underline{#1}}
\newcommand{\ol}[1]{\overline{#1}}
\newcommand{\wh}[1]{\widehat{ #1}}

\newcommand{\gotimes}{\tensor[^g]{\otimes}{}}
\newcommand{\da}[1]{\!\!\downarrow_{#1}}

\newcommand{\ev}{\textup{ev}}
\newcommand{\odd}{\textup{odd}}
\newcommand{\ve}{\varepsilon}

\newcommand{\whphi}{\wh{\phi}}
\newcommand{\wtalpha}{\wt{\alpha}}
\newcommand{\wtrho}{\wt{\rho}}
\newcommand{\wtsigma}{\wt{\sigma}}

\DeclareMathOperator{\ann}{ann}
\DeclareMathOperator{\coker}{coker}

\DeclareMathOperator{\End}{End}
\DeclareMathOperator{\Ext}{Ext}
\DeclareMathOperator{\fd}{fd}
\DeclareMathOperator{\opH}{H}
\newcommand{\Hbul}{\opH^\bullet}
\DeclareMathOperator{\Hom}{Hom}
\DeclareMathOperator{\id}{id}
\DeclareMathOperator{\im}{im}
\DeclareMathOperator{\ind}{ind}
\DeclareMathOperator{\Max}{Max}

\DeclareMathOperator{\pd}{pd}
\DeclareMathOperator{\res}{res}

\DeclareMathOperator{\Spec}{Spec}
\DeclareMathOperator{\Tor}{Tor}

\newcommand{\G}{\mathbb{G}}
\newcommand{\M}{\mathbb{M}}
\newcommand{\N}{\mathbb{N}}
\renewcommand{\P}{\mathbb{P}}
\newcommand{\Z}{\mathbb{Z}}

\newcommand{\fm}{\mathfrak{m}}
\newcommand{\fn}{\mathfrak{n}}
\newcommand{\fp}{\mathfrak{p}}
\newcommand{\fs}{\mathfrak{s}}

\newcommand{\alg}{\mathfrak{alg}}
\newcommand{\fmod}{\mathfrak{mod}}
\newcommand{\calg}{\mathfrak{calg}}
\newcommand{\salg}{\mathfrak{salg}}
\newcommand{\csalg}{\mathfrak{csalg}}
\newcommand{\sgrp}{\mathfrak{sgrp}}
\newcommand{\fsmod}{\mathfrak{smod}}
\newcommand{\sets}{\mathfrak{sets}}
\newcommand{\svec}{\mathfrak{svec}}

\newcommand{\calI}{\mathcal{I}}
\newcommand{\calJ}{\mathcal{J}}
\newcommand{\calN}{\mathcal{N}}

\newcommand{\Ga}{\G_a}
\newcommand{\Gam}{\Ga^-}
\newcommand{\Gar}{\G_{a(r)}}
\newcommand{\Gas}{\G_{a(s)}}
\newcommand{\Gaone}{\G_{a(1)}}

\newcommand{\zero}{\ol{0}}
\newcommand{\one}{\ol{1}}

\newcommand{\Azero}{A_{\zero}}
\newcommand{\Aone}{A_{\one}}

\newcommand{\Vone}{V_{\one}}
\newcommand{\Vzero}{V_{\zero}}

\newcommand{\GLmnr}{GL_{m|n(r)}}

\newcommand{\Grs}{G_{r;s}}
\newcommand{\Grmsp}{G_{r-1;s+1}}

\newcommand{\Mr}{\M_r}
\newcommand{\Mone}{\M_1}
\newcommand{\Moneone}{\M_{1;1}}
\newcommand{\Mones}{\M_{1;s}}
\newcommand{\Mrf}{\M_{r;f}}
\newcommand{\Mrfeta}{\M_{r;f,\eta}}
\newcommand{\Mrseta}{\M_{r;s,\eta}}
\newcommand{\Mrs}{\M_{r;s}}

\newcommand{\Nr}{\calN_r}

\newcommand{\NrG}{\calN_r(G)}
\newcommand{\None}{\calN_1}
\newcommand{\NoneG}{\None(G)}
\renewcommand{\Pr}{\P_r}
\newcommand{\Pone}{\P_1}

\newcommand{\ulExt}{\ul{\Ext}}
\newcommand{\bfHom}{\mathbf{Hom}}
\newcommand{\ulHom}{\ul{\Hom}}
\newcommand{\bsPi}{\bs{\Pi}}

\newcommand{\Vrg}{V_r(G)}
\newcommand{\Vrs}{V_{r;s}}
\newcommand{\bsV}{\bs{V}}
\newcommand{\bsvr}{\bsV\!_r}
\newcommand{\bsvrg}{\bsvr(G)}

\numberwithin{equation}{subsection}

\newtheorem{theorem}{Theorem}[subsection]
\newtheorem*{theorem*}{Theorem}
\newtheorem{proposition}[theorem]{Proposition}
\newtheorem{corollary}[theorem]{Corollary}
\newtheorem{lemma}[theorem]{Lemma}

\theoremstyle{definition}
\newtheorem{definition}[theorem]{Definition}
\newtheorem*{definition*}{Definition}
\newtheorem{notation}[theorem]{Notation}
\newtheorem{example}[theorem]{Example}
\newtheorem{remark}[theorem]{Remark}

\title{Support schemes for infinitesimal unipotent supergroups}

\author{Christopher M.\ Drupieski}
\address{Department of Mathematical Sciences,
		DePaul University,
		Chicago, IL 60614, USA}
\email{c.drupieski@depaul.edu}

\author{Jonathan R. Kujawa}
\address{Department of Mathematics \\
		University of Oklahoma \\
		Norman, OK 73019, USA}
\email{kujawa@math.ou.edu}

\thanks{The first author was supported in part by a Simons Collaboration Grant for Mathematicians, by a paid faculty leave from DePaul University in Winter and Spring  quarters 2018, and by NSF Grant No.\ DMS-1440140 while he was in residence at the Mathematical Sciences Research Institute in Berkeley, CA, during the Spring 2018 semester. The second author was supported in part by NSA grant H98230-16-0055 and in part by a Simons Collaboration Grant for Mathematicians.}


\allowdisplaybreaks

\begin{document}

\begin{abstract}
We investigate support schemes for infinitesimal unipotent supergroups and their representations. Our main results provide a non-cohomological description of these schemes that generalizes the classical work of Suslin, Friedlander, and Bendel. As a consequence, support schemes in this setting have the desired features of such a theory, including naturality with respect to group homomorphisms, the tensor product property, and realizability. As an application of the theory developed here, we investigate support varieties for certain finite-dimensional Hopf subalgebras of the Steenrod algebra.
\end{abstract}

\maketitle

\tableofcontents

\section{Introduction}

Since the pioneering work of Quillen \cite{Quillen:1971}, geometric techniques have played a central role in non-semi\-simple representation theory. Of particular relevance to this paper is the seminal work of Suslin, Friedlander, and Bendel \cite{Suslin:1997, Suslin:1997a}, in which they develop a theory of support varieties for infinitesimal group schemes over fields of positive characteristic. Their main results give a non-cohomological description of the spectrum of the cohomology ring and of the support varieties of finite-dimensional modules. Their work demonstrates that unipotent group schemes, and one-parameter subgroups in particular, play a fundamental role. 

The main goal of this paper is to generalize the results and methods of Suslin, Friedlander, and Bendel to encompass representations of graded objects over fields of odd characteristic. Specifically, we develop the theory of infinitesimal uni\-potent group schemes and one-parameter subgroups, but in the super setting. Throughout, the prefix ``super''  denotes the existence of a $\Z_{2}$-grading and the use of graded analogues of the classical definitions. This includes $\Z$-graded objects as a special case, since one can reduce the $\Z$-gradings modulo two to obtain $\Z_{2}$-graded objects and then apply the theory developed here. In particular, this includes $\Z$-graded Hopf algebras, which play an important role in algebraic topology. As an application of this philosophy, at the end of the paper we explain how our results extend and correct the existing literature on cohomological support varieties for finite-dimensional graded Hopf subalgebras of the Steenrod algebra.

\subsection{Overview}

As mentioned above, infinitesimal one-parameter subgroups (i.e., subgroups isomorphic to a Frobenius kernel of the additive group scheme) play a fundamental role in the classical setting. Previous work by the authors \cite{Drupieski:2017a} and forthcoming work by Benson, Iyengar, Krause, and Pevtsova \cite{Benson:2018} suggests that the correct graded analogues of one-parameter subgroups are the multi\-parameter supergroups, whose definitions we recall in Section \ref{subsec:somesupergroups}. In contrast to their classical counterparts, the multiparameter supergroups are not all unipotent. By definition, the group algebra of each multi\-parameter supergroup is a finite-dimensional Hopf superalgebra quotient (for some $r \geq 1$) of the Hopf superalgebra $\Pr$ defined in \eqref{eq:kMrpol}. In fact, as we show in Proposition \ref{prop:Prquotients}, this property characterizes the group algebras of the multiparameter supergroups.

Motivated by the preceding observation, in Section \ref{subsec:Vrg} we define for each finite $k$-supergroup scheme $G$ the $k$-superfunctor $\bsvrg$, whose set of $A$-points (for each commutative $k$-superalgebra $A$) is given by
	\begin{align*}
	\bsvrg(A) &= \Hom_{Hopf/A}(\Pr \otimes_k A, kG \otimes_k A),
	\end{align*}
the set of Hopf $A$-superalgebra homomorphisms $\rho: \Pr \otimes_k A \rightarrow kG \otimes_k A$. Here $kG = k[G]^\#$ denotes the group algebra of $G$ (the Hopf superalgebra dual to the coordinate algebra of $G$). As observed in Lemma \ref{lemma:bsvrg}, $\bsvrg$ admits the structure of an affine $k$-superscheme of finite type. Then the underlying purely even subfunctor $\Vrg=\bsvrg_{\ev}$ of $\bsvrg$ is an affine $k$-scheme of finite type. This enables us (in parallel to the approach of Suslin, Friedlander, and Bendel) to define, for each finite-dimensional $kG$-super\-module $M$, the closed subscheme
	\[
	\Vrg_M := \set{ \fs \in \Vrg : \pd_{\Pone \otimes_k k(\fs)}(M \otimes_k k(\fs)) = \infty }.
	\]
For an explanation of notation we refer the reader to Section \ref{subsec:definesupportscheme}. Our definition of the support set $\Vrg_M$ is inspired by similar definitions appearing in the literature in the context of commutative local rings (cf.\ \cite{Avramov:1989,Avramov:2000,Jorgensen:2002}), and which were brought to our attention by way of a talk by Srikanth Iyengar at the Conference on Groups, Representations, and Cohomology, held at Sabal M\`or Ostaig, Isle of Skye, Scotland, in June 2015. We note that, while the main results of this paper are proved only for \emph{infinitesimal unipotent} supergroup schemes, the definition of the support scheme $\Vrg_M$ makes sense for \emph{any} finite $k$-supergroup scheme. Our proof that $\Vrg_M$ is a Zariski closed conical subset of $\Vrg$ relies on the fact that, as an ungraded algebra, the Hopf algebra $\Pone = k[u,v]/\subgrp{u^p+v^2}$ is a hypersurface ring. In particular, using Eisenbud's theory of matrix factorizations \cite{Eisenbud:1980}, we show in Proposition \ref{prop:Poneprojdim} that a $\Pone$-supermodule has infinite projective dimension if and only if a certain cup product in cohomology is nonzero.

Now given a finite $k$-supergroup scheme $G$, write $H(G,k)$ for the subalgebra
	\[
	\bigoplus_{n \geq 0} \opH^n(G,k)_{\ol{n}} = \opH^{\ev}(G,k)_{\zero} \oplus \opH^{\odd}(G,k)_{\one}
	\]
of the full cohomology ring $\Hbul(G,k)$.  This is a finitely-generated, commutative (in the ungraded sense) $k$-algebra, and its spectrum coincides with that of $\Hbul(G,k)$. In Section \ref{subsec:universalhom} we show that, for any finite $k$-supergroup scheme $G$, there exists a natural homomorphism of $\Z[\frac{p^r}{2}]$-graded algebras 
	\[
	\psi_r: H(G,k) \rightarrow k[\Vrg]
	\]
that multiplies degrees by $\frac{p^r}{2}$. The first main result of this paper is that, for $G$ infinitesimal unipotent, $\psi_r$ induces a universal homeomorphism between the associated schemes.

\begin{theorem*}[Theorem \ref{thm:psiuniversalhomeo}]
Let $G$ be an infinitesimal  unipotent $k$-super\-group scheme of height $\leq r$. Then the kernel of the homomorphism
	\[
	\psi_r: H(G,k) \rightarrow k[\Vrg]
	\]
is a locally nilpotent ideal, and the image of $\psi_r$ contains the $p^r$-th power of each element of $k[\Vrg]$. Consequently, the associated morphism of schemes $\Psi_r: \Vrg \rightarrow \abs{G}$ is a universal homeomorphism.
\end{theorem*}  

Next let $M$ be a finite-dimensional rational $G$-supermodule. The second main result of the paper asserts that the morphism of schemes $\Psi_r: \Vrg \rightarrow \abs{G}$ restricts to a homeomorphism between the non-cohomological support scheme $\Vrg_M$ and the cohomological support scheme $\abs{G}_{M}$.

\begin{theorem*}[Theorem \ref{thm:Psirinv}]
Let $G$ be an infinitesimal unipotent $k$-supergroup scheme of height $\leq r$, and let $M$ be a finite-dimensional rational $G$-supermodule. Then the morphism $\Psi_r: \Vrg \rightarrow \abs{G}$ satisfies $\Psi_r^{-1}(\abs{G}_M) = \Vrg_M$. Thus, $\Psi_r$ restricts to a finite universal homeomorphism
	\[
	\Psi_r: \Vrg_M \stackrel{\sim}{\rightarrow} \abs{G}_M.
	\]
\end{theorem*}

The proofs of Theorems \ref{thm:psiuniversalhomeo} and \ref{thm:Psirinv} build on the authors' previous work investigating the cohomology of multiparameter supergroups \cite{Drupieski:2017a,Drupieski:2017b}. The proofs also rely, critically, on the detection theorems of Benson, Iyengar, Krause, and Pevtsova \cite{Benson:2018}. While the overall strategy of the arguments parallels, in broad strokes, the methods of Suslin, Friedlander, and Bendel from the classical setting, fully implementing that strategy requires non-obvious generalizations, intricate calculations, and the use of deep results from commutative algebra and elsewhere.

It is worth emphasizing that when $G$ is an infinitesimal unipotent $k$-\emph{group} scheme (i.e., when $G$ is purely even), the schemes $\Vrg$ and $\Vrg_M$ as defined in this paper reduce to the schemes of the same names as defined by Suslin, Friedlander, and Bendel \cite{Suslin:1997a, Suslin:1997}; see Remark \ref{remark:Vrgpurelyeven} and Lemma \ref{lemma:finiteifffree}. We are thus justified in adopting their notation to our new context, and the main results of this paper are true generalizations of their classical counterparts.

Theorem \ref{thm:Psirinv} provides a non-cohomological description for the scheme $\abs{G}_{M}$. Applying this description, in Section \ref{SS:applications} we show that the cohomological support schemes of infinitesimal uni\-potent supergroups have the main desirable properties of such a theory: naturality with respect to group homomorphisms, the tensor product property, and realization. In particular, in Theorem \ref{theorem:supportvarietyunion} we show that $\abs{G}_M$ is a union of pieces coming from the multiparameter subsupergroup schemes of $G$. A similar stratification theorem, albeit one not directly comparable with ours, has been previously stated in the context of finite-dimensional graded connected cocommutative Hopf algebras by Nakano and Palmieri \cite[Theorem 3.2]{Nakano:1998}. However, their proof implicitly relies on an $F$-surjectivity theorem stated by Palmieri \cite[Theorem 4.1]{Palmieri:1997}, and as we discuss at the end of Section \ref{SS:applications}, the $F$-surjectivity theorem depends on the graded Hopf algebra in question having only finitely many graded Hopf subalgebras.

Finally, in Section \ref{subsec:examples} we analyze in detail the support varieties of an interesting family of $2p$-dimensional supermodules over the supergroup $\Moneone = \Gaone \times \Gam$, showing that the support varieties of these modules correspond to the affine lines in the two-dimensional affine space $\abs{\Moneone}$. Then in Section \ref{subsec:steenrod} we explain how the group algebra $k\Moneone$ occurs as a graded Hopf subalgebra of the Steenrod algebra, and we apply our calculations from Section \ref{subsec:examples} to show that, in general, and in contrast to many support variety theories appearing in the literature, support varieties in the graded setting need not be described in terms of projectivity over cyclic subalgebras. For a discussion of additional cautionary examples in the graded setting, we refer to the reader to \cite[\S1.4]{Drupieski:2017a}.

\subsection{Future and related work}

As mentioned previously, while the main results of this paper apply only to infinitesimal \emph{unipotent} supergroups, the definition of the support scheme $\Vrg_M$ makes sense for \emph{any} finite $k$-supergroup scheme. One could thus hope to extend the results of this paper to arbitrary infinitesimal $k$-supergroup schemes. Indeed, the decomposition described in Lemma \ref{lemma:bsvrg} already allows us to interpret one of the main calculations of our earlier work \cite[Corollary 6.2.4]{Drupieski:2017a} as saying that, modulo a finite morphism of varieties, the cohomological variety $\abs{\GLmnr}$ of the $r$-th Frobenius kernel of the general linear supergroup identifies with the affine variety $V_r(\GLmnr)(k)$.

The main obstacle to extending the results of this paper from unipotent to non-unipotent super\-groups is showing---for non-unipotent infinitesimal supergroup schemes---that projectivity of modules and nilpotence of cohomology classes can be detected by restriction to an appropriate family of finite supergroup schemes (e.g., the multiparameter supergroup schemes). In the classical ungraded setting, the extension from a detection theorem for unipotent infinitesimal groups to non-unipotent infinitesimal groups is accomplished by an argument that exploits algebro-geometric relationships between the general linear group $GL_n$ and (any) one of its Borel subgroups $B$. For example, Kempf vanishing ensures that the restriction map in cohomology $\Hbul(GL_n,M) \rightarrow \Hbul(B,M)$ is an isomorphism for any rational $GL_n$-module $M$, and ensures that the algebraic group induction functor $\ind_B^{GL_n}(-)$ maps the trivial module $k$ to itself. One could naively hope that some kind of bootstrap argument like this could be made in the graded setting as well, but the fact that algebraic supergroups have non-conjugate Borels, and the fact that the super analogue of the induction functor $\ind_B^{GL_{m|n}}(-)$ behaves differently depending on which Borel subgroup $B \subset GL_{m|n}$ is chosen, immediately presents serious difficulties. Once again, substantial new ideas will be needed in the graded setting, including perhaps a better understanding in positive characteristic of the algebro-geometric relationship between the general linear supergroup $GL_{m|n}$ and its Borel subgroups.  Recently, Grantcharov, Grantcharov, Nakano, and Wu \cite{Grantcharov:2018} introduced certain parabolic subalgebras for complex Lie superalgebras which have good homological properties. The positive characteristic analogue of these subalgebras may also shed light on support schemes for general infinitesimal $k$-supergroup schemes.

\subsection{Acknowledgements} \label{subsection:Acknowledgements}

It would be difficult for the authors to overstate their gratitude to David Benson, Srikanth Iyengar, Henning Krause, and Julia Pevtsova for helpful conversations and for their willingness to share early versions of their detection theorem manuscript \cite{Benson:2018}. In particular, the first author thanks David Benson for explaining how the theory of matrix factorizations applies to the algebra $\Pone$. The authors thank Luchezar Avramov and Srikanth Iyengar for sharing drafts of their manuscript \cite{Avramov:2018} and for other conversations that helped lead to the proofs of Lemma \ref{lemma:injprojdimtorsionmodule} and Proposition \ref{prop:secondradical}.  We also thank Daniel Nakano and John Palmieri for their comments on an earlier version of the paper. Finally, the first author thanks the Mathematical Sciences Research Institute in Berkeley, CA:  key results were obtained while he enjoyed their hospitality during the program on Group Representation Theory and Applications in Spring 2018.

\subsection{Conventions} \label{subsection:conventions}

We generally follow the conventions of our previous work \cite{Drupieski:2017a,Drupieski:2017b}, to which we refer the reader for any unexplained terminology or notation. For additional standard terminology and notation, the reader may consult Jantzen's book \cite{Jantzen:2003}. Except when indicated otherwise, $k$ will denote a field of characteristic $p \geq 3$ and $r$ will denote a positive integer. All vector spaces will be $k$-vector spaces, and all unadorned tensor products will denote tensor products over $k$. Given a $k$-vector space $V$, let $V^\# = \Hom_k(V,k)$ be its $k$-linear dual. Let $\N = \set{0,1,2,3,\ldots}$ denote the set of non-negative integers.

Set $\Z_2 = \Z/2\Z = \set{\zero,\one}$. Following the literature, we use the prefix `super' to indicate that an object is $\Z_2$-graded. We denote the decomposition of a vector superspace into its $\Z_2$-homogeneous components by $V = \Vzero \oplus \Vone$, calling $\Vzero$ and $\Vone$ the even and odd subspaces of $V$, respectively, and writing $\ol{v} \in \Z_2$ to denote the superdegree of a homogeneous element $v \in V$. Whenever we state a formula in which homogeneous degrees of elements are specified, we mean that the formula is true as written for homogeneous elements and that it extends linearly to non-homogeneous elements. For example, the parity map $\pi: V \rightarrow V$ is defined by $\pi(v) = (-1)^{\ol{v}} v$. We use the symbol $\cong$ to denote even (i.e., degree-preserving) isomorphisms of superspaces, and reserve the symbol $\simeq$ for odd (i.e., degree-reversing) isomorphisms.

\section{Preliminaries}\label{section:preliminaries}

\subsection{Multiparameter supergroups} \label{subsec:somesupergroups}

In this section we recall the definitions of some of the affine supergroup schemes introduced in \cite{Drupieski:2017a}. Given an affine $k$-supergroup scheme $G$ with coordinate Hopf superalgebra $k[G]$, set $kG = k[G]^\#$. The supercoalgebra structure on $k[G]$ induces by duality a $k$-superalgebra structure on $kG$; with this structure, we call $kG$ the \emph{group algebra} of $G$. If $G$ is a finite $k$-supergroup scheme, then $kG$ inherits the structure of a Hopf $k$-superalgebra.

First, $\Mr$ is the affine $k$-super\-group scheme whose coordinate algebra $k[\Mr]$ is the commutative $k$-super\-algebra generated by the odd element $\tau$ and the even elements $\theta$ and $\sigma_i$ for $i \in \N$, such that $\tau^2 = 0$, $\sigma_0 = 1$, $\theta^{p^{r-1}} = \sigma_1$, and $\sigma_i \sigma_j = \binom{i+j}{i} \sigma_{i+j}$, i.e.,
\[ \textstyle
k[\Mr] = k[\tau,\theta,\sigma_1,\sigma_2,\ldots]/\subgrp{ \tau^2=0, \theta^{p^{r-1}} = \sigma_1, \text{ and } \sigma_i\sigma_j = \binom{i+j}{i}\sigma_{i+j} \text{ for $i,j \in \N$}}.
\]
Then the set of monomials $\{ \theta^i \sigma_j , \tau \theta^i \sigma_j : 0 \leq i < p^{r-1}, j \in \N \}$ is a homogeneous basis for $k[\Mr]$, which we call the distinguished homogeneous basis for $k[\Mr]$. The coproduct $\Delta$ and the antipode $S$ on $k[\Mr]$ are defined on generators by the formulas
	\begin{align*}
	\Delta(\tau) &= \tau \otimes 1 + 1 \otimes \tau, & S(\tau) &= -\tau, \\
	\Delta(\theta) &= \theta \otimes 1 + 1 \otimes \theta, & S(\theta) &= -\theta, \\
	\Delta(\sigma_i) &= \textstyle \sum_{u+v=i} \sigma_u \otimes \sigma_v + \sum_{u+v+p=i} \sigma_u \tau \otimes \sigma_v \tau, & S(\sigma_i) &= (-1)^i \sigma_i.
	\end{align*}
For $0 \leq i \leq r-1$, let $u_i \in k\Mr = k[\Mr]^\#$ be the even linear functional that is dual to the distinguished basis element $\theta^{p^i} \in k[\Mr]$ (so in particular, $u_{r-1}$ is dual to $\sigma_1 = \theta^{p^{r-1}}$), and let $v \in k\Mr$ be the odd linear functional that is dual to the distinguished basis vector $\tau \in k[\Mr]$. Then by \cite[Proposition 3.1.4]{Drupieski:2017a}, the group algebra $k\Mr$ is given by
	\begin{equation} \label{eq:kMr}
	k\Mr = k[[u_0,\ldots,u_{r-1},v]]/\subgrp{u_0^p,\ldots,u_{r-2}^p,u_{r-1}^p + v^2}.
	\end{equation}

Let $\Pr$ be the `polynomial subalgebra' of $k\Mr$,
	\begin{equation} \label{eq:kMrpol}
	\Pr = k[u_0,\ldots,u_{r-1},v]/\subgrp{u_0^p,\ldots,u_{r-2}^p,u_{r-1}^p + v^2}.
	\end{equation}
By \cite[Remark 3.1.3(3)]{Drupieski:2017a}, the $\Z_2$-grading on $k[\Mr]$ lifts to a $\Z$-grading such that $\deg(\tau) = p^r$, $\deg(\theta) = 2$, and $\deg(\sigma_i) = 2ip^{r-1}$, which makes $k[\Mr]$ into a graded Hopf algebra of finite type in the sense of Milnor and Moore \cite{Milnor:1965}. Then $\Pr$ is the graded dual of $k[\Mr]$. In particular, $\Pr$ inherits by duality the structure of a graded Hopf algebra of finite type \cite[Proposition 4.8]{Milnor:1965}. To describe this structure, first define $u_i$ for $i \geq r$ by $u_i = (u_{r-1})^{p^{i-r+1}}$, so that $u_r = u_{r-1}^p$, $u_{r+1} = u_{r-1}^{p^2}$, etc. Next, given an integer $\ell \geq 0$ with base-$p$ decomposition $\ell = \sum_{i \geq 0} \ell_i p^i$, set
	\begin{equation} \label{eq:gammaell}
	\gamma_\ell = \prod_{i \geq 0} \frac{u_i^{\ell_i}}{\ell_i!} = \frac{(u_0)^{\ell_0}(u_1)^{\ell_1}(u_2)^{\ell_2}\cdots}{(\ell_0!)(\ell_1!)(\ell_2!)\cdots}.
	\end{equation}
This product is a well-defined monomial in $\Pr$ by the fact that $0 \leq \ell_i < p$ for each $i$ and the fact that the $\ell_i$ are eventually all equal to $0$; for $\ell \geq p^r$, this definition for $\gamma_\ell$ differs by a scalar factor from the definition in \cite[Proposition 3.1.4(2)]{Drupieski:2017b}. Then the set of monomials $\set{\gamma_\ell, v \cdot \gamma_\ell: \ell \in \N}$ is a homogeneous basis for $\Pr$, which we call the distinguished homogeneous basis for $\Pr$. Now the coproduct $\Delta$ and antipode $S$ on $\Pr$ are determined by the formulas
	\begin{align*}
	\Delta(\gamma_\ell) &= \textstyle \sum_{i+j=\ell} \gamma_i \otimes \gamma_j, & S(\gamma_\ell) &= (-1)^\ell \gamma_\ell & \text{for $0 \leq \ell < p^r$,} \\
	\Delta(u_{r-1}^p) &= u_{r-1}^p \otimes 1 + 1 \otimes u_{r-1}^p, & S(u_{r-1}^p) &= - u_{r-1}^p, \\
	\Delta(v) &= v \otimes 1 + 1 \otimes v, & S(v) &= -v;
	\end{align*}
cf.\ \cite[Proposition 3.1.4(6)]{Drupieski:2017a}. More generally, let $\ell \in \N$ be arbitrary, and write $\ell = a+bp^r$ for integers $a$ and $b$ with $0 \leq a < p^r$ and $b \geq 0$. Then $\gamma_\ell = \gamma_a \cdot \gamma_{bp^r}$, and the preceding formulas imply that
	\begin{equation} \label{eq:coproductgammaell}
	\Delta(\gamma_\ell) = \sum_{\substack{i+j=a \\ s+t \stackrel{N.C.}{=} b}} \gamma_{i+sp^r} \otimes \gamma_{j+tp^r},
	\end{equation}
where $s+t \stackrel{N.C.}{=} b$ means that no carries are required when $s$ and $t$ are added in base $p$. (The `no carries' condition is a consequence of using the Binomial Theorem to compute $\Delta((u_{r-1}^p)^b)$, and then applying Lucas' theorem for binomial coefficients modulo $p$.)

Let $0 \neq f = \sum_{i=1}^t c_i T^{p^i} \in k[T]$ be an inseparable $p$-polynomial (i.e., a $p$-polynomial without a linear term), and let $\eta \in k$. Since $u_{r-1}^p$ and $u_0$ are each primitive in $\Pr$, the sum $f(u_{r-1}) + \eta \cdot u_0$ is also primitive in $\Pr$. Then by the assumption that $f \neq 0$, the quotient
	\[
	k\Mrfeta := \Pr/\subgrp{f(u_{r-1})+\eta \cdot u_0}
	\]
is a finite-dimensional cocommutative Hopf superalgebra. The \emph{multiparameter supergroup} $\Mrfeta$ is the affine $k$-supergroup scheme such that $k[\Mrfeta]^\# = k\Mrfeta$, i.e., such that the group algebra of $\Mrfeta$ is precisely $k\Mrfeta$. Set $\Mrf = \M_{r;f,0}$, and given an integer $s \geq 1$, set $\Mrseta = \M_{r;T^{p^s},\eta}$ and $\Mrs = \M_{r;T^{p^s},0}$. Then
	\begin{align*}
	k\Mrf &= k[u_0,\ldots,u_{r-1},v]/\subgrp{u_0^p,\ldots,u_{r-2}^p,u_{r-1}^p+v^2,f(u_{r-1})}, \\
	k\Mrseta &= k[u_0,\ldots,u_{r-1},v]/\subgrp{u_0^p,\ldots,u_{r-2}^p,u_{r-1}^p+v^2,u_{r-1}^{p^s}+\eta \cdot u_0}, \quad \text{and} \\
	k\Mrs &= k[u_0,\ldots,u_{r-1},v]/\subgrp{u_0^p,\ldots,u_{r-2}^p,u_{r-1}^p+v^2,u_{r-1}^{p^s}}.
	\end{align*}
The coordinate algebra $k[\Mrs]$ identifies with the Hopf subsuperalgebra of $k[\Mr]$ generated by $\tau$, $\theta$, and $\sigma_i$ for $1 \leq i < p^s$. (For $r=1$, the generator $\theta$ is irrelevant, since then $\sigma_1 = \theta^{p^{r-1}} = \theta$.) More generally, the coalgebra structure of $k[\Mrfeta]$ is given for $r = 1$ in \cite[Lemma 3.1.9]{Drupieski:2017a} and for $r \geq 2$ and $f = T^{p^s}$ in \cite[Lemma 2.2.1]{Drupieski:2017b}.

For $r \geq 2$, the assumption $f \neq 0$ is necessary in order for the quotient $\Pr/\subgrp{f(u_{r-1})+\eta \cdot u_0}$ to be finite-dimensional. For $r = 1$, the quotient $\Pone/\subgrp{f(u_0)+\eta \cdot u_0}$ is finite-dimensional so long as either $f \neq 0$ or $\eta \neq 0$. In particular, if $\eta \neq 0$, then
	\[
	k[u,v]/\subgrp{u^p+v^2, \eta \cdot u} = k[u,v]/\subgrp{u^p+v^2, u} = k[v]/\subgrp{v^2} = k\Gam,
	\]
the group algebra of the purely odd additive supergroup scheme $\Gam$.

\begin{definition}[Multiparameter supergroups] \label{def:multiparameter}
An affine $k$-supergroup scheme is a \emph{multiparameter $k$-supergroup scheme} if it is isomorphic to one of the following $k$-supergroup schemes:
	\begin{itemize}
	\item $\Gar$ for some $r \in \N$,
	\item $\Gar \times \Gam$ for some $r \in \N$, or
	\item $\Mrfeta$ for some $r \geq 1$, some inseparable $p$-polynomial $0 \neq f \in k[T]$, and some $\eta \in k$.
	\end{itemize}
By convention, $\G_{a(0)}$ is the trivial group scheme, with $k\G_{a(0)} = k = k[\G_{a(0)}]$.
\end{definition}

There are some repetitions in the list in the preceding definition. For example, $\M_{r;1} = \Gar \times \Gam$, and $\Mrseta \cong \M_{r;s,\eta'}$ if $\eta/\eta' = a^{p^{r+s-1}-1}$ for some $a \in k$; cf.\ \cite[Theorem 3.8]{Benson:2018}. We do not attempt to classify the isomorphisms among the multiparameter supergroups.

The multiparameter $k$-supergroup schemes are all infinitesimal: $\Gar$ and $\Mrfeta$ are infinitesimal of height $r$ (cf.\ \cite[Lemma 3.1.7]{Drupieski:2017a}), while $\Gam$ is infinitesimal of height $1$. Furthermore, there are canonical Hopf superalgebra identifications
	\begin{align*}
	\Pr/\subgrp{v} &\cong k\Gar, &	\Pr/\subgrp{u_0,\ldots,u_{r-1}} &\cong k\Gam, & \Pr/\subgrp{u_0} &\cong \P_{r-1}.
	\end{align*}
Thus if $E$ is a multiparameter $k$-supergroup scheme of height $r' \leq r$, then there is a canonical Hopf superalgebra quotient map $\Pr \twoheadrightarrow \P_{r'} \twoheadrightarrow kE$.

The group algebra $k\Mrseta$ is evidently a local algebra, so $\Mrseta$ is a unipotent supergroup scheme. More generally, it follows from \cite[Remark 3.1.3(4)]{Drupieski:2017a} that any finite-dimensional rational representation of $\Mr$ factors for $s \gg 0$ through the canonical quotient map $\Mr \twoheadrightarrow \Mrs$ corresponding to the inclusion of coordinate algebras $k[\Mrs] \hookrightarrow k[\Mr]$. Combined with the local finiteness of rational representations, this implies that $\Mr$ is also unipotent. If $f$ is not a scalar multiple of a single monomial, then $\Mrfeta$ is not unipotent.

\subsection{Hopf superalgebra quotients of \texorpdfstring{$\Pr$}{Pr}}

Benson, Iyengar, Krause, and Pevtsova (BIKP) \cite{Benson:2018} define a finite $k$-super\-group scheme to be \emph{elementary} if it is isomorphic for some positive integers $r,s,t$ to a quotient of $\Mrs \times (\Z/p\Z)^t$.\footnote{See \cite[Definition 1.1]{Benson:2018}. Their $E_{m,n}^-$ is our $\M_{n;m}$, and their $E_{m,n,\mu}^-$ is our $\M_{n+1;m,-\mu}$.} Here $\Z/p\Z$ denotes the finite constant group scheme corresponding to the finite cyclic group $\Z/p\Z$. Then an infinitesimal $k$-supergroup scheme $G$ is elementary if and only if it is a quotient of $\Mrs$ for some positive integers $r$ and $s$. Equivalently, $G$ is elementary if and only if its group algebra $kG$ is a Hopf superalgebra quotient of $k\Mrs$. More generally, we can classify all finite-dimensional Hopf superalgebra quotients of $\Pr$.

\begin{proposition} \label{prop:Prquotients}
Let $A$ be a finite-dimensional Hopf superalgebra quotient of $\Pr$. Then $A$ is isomorphic to the group algebra of one of the following:
	\begin{itemize}
	\item $\Gas$ for some integer $0 \leq s \leq r$,
	\item $\Gas \times \Gam$ for some integer $0 \leq s \leq r$, or
	\item $\M_{s;f,\eta}$ for some $1 \leq s \leq r$, some inseparable $p$-polynomial $0 \neq f \in k[T]$, and some $\eta \in k$.
	\end{itemize}
In other words, $A$ is isomorphic to the group algebra of a multiparameter $k$-supergroup scheme.
\end{proposition}

\begin{proof}
Let $\phi: \Pr \rightarrow A$ be a quotient homomorphism of Hopf superalgebras, with $A$ finite-dimen\-sional. If $\phi(\gamma_1) = \phi(u_0) = 0$, then $\phi$ factors through the quotient $\Pr/ \subgrp{u_0}$. If $r = 1$, then $\Pone/\subgrp{u_0} \cong k[v]/\subgrp{v^2} = k\Gam$, and hence $A \cong k$ or $A \cong k\Gam$, depending on whether or not $\phi(v) = 0$. If $r > 1$, then $A$ is a quotient of $\Pr/\subgrp{u_0} \cong \P_{r-1}$, and we may assume by induction on $r$ that $A$ is isomorphic to one of the group algebras listed in the proposition. So assume that $\phi(\gamma_1) \neq 0$, and let $m \geq 1$ be the minimal integer such that $\phi(\gamma_1),\ldots,\phi(\gamma_m)$ are $k$-linearly independent in $A$ but $\phi(\gamma_1),\ldots,\phi(\gamma_m),\phi(\gamma_{m+1})$ are not. Then $\phi(\gamma_{m+1}) = \sum_{\ell=1}^m a_\ell \cdot \phi(\gamma_\ell)$ for some scalars $a_1\ldots,a_m \in k$. Since $\phi$ is a homomorphism of Hopf superalgebras, it satisfies the compatibility condition $\Delta_A \circ \phi = (\phi \otimes \phi) \circ \Delta$, where $\Delta_A$ denotes the coproduct on $A$. Suppose $m+1 = a+bp^r$ for integers $a$ and $b$ with $0 \leq a < p^r$ and $b \geq 0$. Then applying $\Delta_A$ to both sides of the dependence relation $\phi(\gamma_{m+1}) = \sum_{\ell=1}^m a_\ell \cdot \phi(\gamma_\ell)$, and applying the coproduct formula \eqref{eq:coproductgammaell}, one gets
	\begin{equation} \label{eq:coproductdeprelation}
	\sum_{\substack{i+j=a \\ s+t \stackrel{N.C.}{=} b}} \phi(\gamma_{i+sp^r}) \otimes \phi(\gamma_{j+tp^r}) = \sum_{\ell=1}^m \sum_{\substack{c+dp^r=\ell \\ 0 \leq c < p^r \\ d \geq 0}} \sum_{\substack{c_1+c_2 = c \\ d_1+d_2 \stackrel{N.C.}{=} d}} a_\ell \cdot \phi(\gamma_{c_1+d_1p^r}) \otimes \phi(\gamma_{c_2+d_2p^r}).
	\end{equation}
Note that for each fixed value of $\ell$, the integers $c$ and $d$ such that $c+dp^r = \ell$ are unique. The summands on the left-hand side of \eqref{eq:coproductdeprelation} corresponding to the tuples $(i,j,s,t) = (a,0,b,0)$ and $(i,j,s,t) = (0,a,0,b)$ are $\phi(\gamma_{m+1}) \otimes 1$ and $1 \otimes \phi(\gamma_{m+1})$. Using the dependence relation, these terms can be rewritten as $\sum_{\ell=1}^m a_\ell \cdot \phi(\gamma_\ell) \otimes 1$ and $\sum_{\ell=1}^m a_\ell \cdot 1 \otimes \phi(\gamma_\ell)$. These expressions also appear on the right-hand side of \eqref{eq:coproductdeprelation}---namely, as the summands corresponding for each $\ell$ to the tuples of the form $(c_1,c_2,d_1,d_2) = (c,0,d,0)$ and $(c_1,c_2,d_1,d_2) = (0,c,0,d)$---so they can be subtracted from both sides of \eqref{eq:coproductdeprelation}. If $m+1$ is not of the form $p^{r+e}$ for some integer $e \geq 0$, then there will remain additional terms on the left-hand side of \eqref{eq:coproductdeprelation}, and hence a sum of terms of the form $\phi(\gamma_i) \otimes \phi(\gamma_j)$ with $i+j = m+1$ and $1 \leq i,j \leq m$ would be left equal to a combination of terms of the form $\phi(\gamma_i) \otimes \phi(\gamma_j)$ with $1 \leq i+j \leq m$. This would be a contradiction, because the fact that $\phi(\gamma_1),\ldots,\phi(\gamma_m)$ are linearly independent in $A$ implies that the set $\set{ \phi(\gamma_i) \otimes \phi(\gamma_j) : 1 \leq i,j \leq m}$ is linearly independent in $A \otimes A$. So it must be the case that $m+1 = p^{r+e}$ for some integer $e \geq 0$. Now the simplified version of \eqref{eq:coproductdeprelation} takes the form $0 = \sum a_\ell \cdot \phi(\gamma_{c_1+d_1p^r}) \otimes \phi(\gamma_{c_2+d_2p^r})$, where the sum is over all remaining tuples $(c_1,c_2,d_1,d_2)$ not of the form $(c,0,d,0)$ or $(0,c,0,d)$. Let us call any such remaining tuple $(c_1,c_2,d_1,d_2)$ a `nontrivial decomposition' of the integer $(c_1+d_1p^r)+(c_2+d_2p^r)$. By the linear independence of the set $\set{\phi(\gamma_i) \otimes \phi(\gamma_j): 1 \leq i,j \leq m}$, this implies that $a_\ell = 0$ for all $1 \leq \ell \leq m$ such that $\ell$ admits a nontrivial decomposition. Then the only possible nonzero coefficients that may occur in the dependence relation $\phi(\gamma_{m+1}) = \sum_{\ell=1}^m a_\ell \cdot \phi(\gamma_\ell)$ are $a_1,a_{p^r},a_{p^{r+1}},\ldots,a_{p^{r+e-1}}$.

We have shown under the assumption $\phi(\gamma_1) \neq 0$ that there exists an integer $e \geq 0$ such that $\phi(\gamma_1),\ldots,\phi(\gamma_{p^{r+e}-1})$ are linearly independent in $A$. Since these elements must be mapped into the augmentation ideal of $A$, but $\phi(\gamma_0) = \phi(1_{\Pr}) = 1_A$ is not, we deduce that $\phi(\gamma_0),\phi(\gamma_1),\ldots,\phi(\gamma_{p^{r+e}-1})$ are linearly independent in $A$. Suppose $\phi(v) = 0$. Then $\phi$ factors through the quotient $\Pr / \subgrp{v} \cong k\Gar$. Since the images of $\gamma_0,\gamma_1,\ldots,\gamma_{p^r-1}$ under the quotient map $\Pr \twoheadrightarrow \Pr/\subgrp{v} \cong k\Gar$ already form a basis for $k\Gar$, while the elements $\gamma_\ell$ for $\ell \geq p^r$ map to $0$, it follows in this case that $A \cong k\Gar$. Now suppose that $\phi(v) \neq 0$. We claim that the set $\set{ \phi(v \cdot \gamma_\ell) : 0 \leq \ell \leq p^{r+e}-1}$ is linearly independent in $\Aone$. If not, there exists a minimal integer $0 \leq m < p^{r+e}-1$ such that the set $\set{\phi(v \cdot \gamma_\ell) : 0 \leq \ell \leq m}$ is linearly independent in $\Aone$, but $\set{\phi(v \cdot \gamma_\ell): 0 \leq \ell \leq m+1}$ is not. Then $\phi(v \cdot \gamma_{m+1}) = \sum_{\ell=0}^m c_\ell \cdot \phi(v \cdot \gamma_\ell)$ for some scalars $c_0,\ldots,c_m \in k$. As before, suppose $m+1 = a+bp^r$ with $0 \leq a < p^r$ and $b \geq 0$. Then by the compatibility of $\phi$ and $\Delta$, one gets as above
	\begin{multline} \label{eq:vcoproductrelation}
	\sum_{\substack{i+j=a \\ s+t \stackrel{N.C.}{=} b}} \left\{ \phi(v \cdot \gamma_{i+sp^r}) \otimes \phi(\gamma_{j+tp^r}) + \phi(\gamma_{i+sp^r}) \otimes \phi(v \cdot \gamma_{j+tp^r})\right\} \\
	= \sum_{\ell=0}^m \sum_{\substack{c+dp^r=\ell \\ 0 \leq c < p^r \\ d \geq 0}} \sum_{\substack{c_1+c_2 = c \\ d_1+d_2 \stackrel{N.C.}{=} d}} c_\ell \cdot \left\{ \phi(v \cdot \gamma_{c_1+d_1p^r}) \otimes \phi(\gamma_{c_2+d_2p^r}) + \phi(\gamma_{c_1+d_1p^r}) \otimes \phi(v \cdot \gamma_{c_2+d_2p^r}) \right\}.
	\end{multline}
As before, the summands $\phi(v \cdot \gamma_{m+1}) \otimes 1$ and $1 \otimes \phi(v \cdot \gamma_{m+1})$ on the left-hand side of \eqref{eq:vcoproductrelation} can be rewritten using the dependence relation and then subtracted from both sides of the equation. The resulting new equation still includes the terms $\phi(v) \otimes \phi(\gamma_{m+1})$ and $\phi(\gamma_{m+1}) \otimes \phi(v)$ on the left-hand side, both with coefficient $1$, while the right-hand side is a sum of terms of the form $\phi(v \cdot \gamma_i) \otimes \phi(\gamma_j)$ and $\phi(\gamma_i) \otimes \phi(v \cdot \gamma_j)$ with $0 \leq i+j \leq m$. This is a contradiction, because our hypothesis and the results of the previous paragraph imply that the set $\set{\phi(\gamma_i), \phi(v \cdot \gamma_j) : 0 \leq i < p^{r+e}, 0 \leq j \leq m}$ is linearly independent in $A$ (there can be no nontrivial dependence relations between elements in $\Azero$ and $\Aone$), and hence tensor products of pairs of these elements are linearly independent in $A \otimes A$.

We have now shown, under the assumptions $\phi(\gamma_1) \neq 0$ and $\phi(v) \neq 0$, that there exists an integer $e \geq 0$ such that the set $\set{\phi(\gamma_\ell), \phi(v \cdot \gamma_\ell): 0 \leq \ell < p^{r+e}}$ is linearly independent in $A$, but that a dependence relation of the form $\phi(\gamma_{p^{r+e}}) = a_1 \cdot \phi(\gamma_1) + \sum_{i=1}^e a_{p^{r-1+i}} \cdot \phi(\gamma_{p^{r-1+i}})$ holds in $A$. Set $f = T^{p^{e+1}} - \sum_{i=1}^e a_{p^{r-1+i}} T^{p^i}$, and set $\eta = -a_1$. Then the dependence relation implies that the quotient homomorphism $\phi: \Pr \rightarrow A$ factors through the canonical quotient map
	\[
	\Pr \twoheadrightarrow k\Mrfeta = \Pr/\subgrp{f(u_{r-1}) + \eta \cdot u_0}.
	\]
Since the set $\set{\phi(\gamma_\ell), \phi(v \cdot \gamma_\ell): 0 \leq \ell < p^{r+e}}$ is already a homogeneous basis for $k\Mrfeta$, we deduce that $A$ must be isomorphic to $k\Mrfeta$.
\end{proof}

\begin{corollary} \label{cor:infinitesimalelementary}
Every infinitesimal elementary $k$-supergroup scheme is isomorphic to one of:
	\begin{enumerate}
	\item $\Gar$ for some integer $r \geq 0$,
	\item $\Gar \times \Ga^-$ for some integer $r \geq 0$,
	\item $\Mrs$ for some integers $r,s \geq 1$, or
	\item $\Mrseta$ for some integers $r \geq 2$, $s \geq 1$, and some scalar $0 \neq \eta \in k$.
	\end{enumerate}
\end{corollary}

\begin{proof}
Let $G$ be an infinitesimal elementary $k$-supergroup scheme and let $A = kG$ be its group algebra. Then $A$ is a finite-dimensional Hopf superalgebra quotient of $k\Mrs$ for some integers $r,s \geq 1$, and hence is also a Hopf superalgebra quotient $\Pr$. We want to show that $A$ is isomorphic to the group algebra of one of the supergroups listed in the statement of the corollary. Suppose by way of contradiction that $A$ is isomorphic to the group algebra of a supergroup listed in Proposition \ref{prop:Prquotients} but not listed in the corollary. Then either
	\begin{enumerate}
	\item $A \cong k\M_{1;f}$ for some inseparable $p$-polynomial $0 \neq f \in k[T]$ that is not a multiple of a single monomial, or \label{item:case1}
	\item $A \cong k\M_{1;f,\eta}$ for some inseparable $p$-polynomial $0 \neq f \in k[T]$ and some scalar $0 \neq \eta \in k$, or \label{item:case2}
	\item $A \cong k\M_{r';f,\eta}$ for some integer $2 \leq r' \leq r$, some inseparable $p$-polynomial $0 \neq f \in k[T]$ that is not a scalar multiple of a single monomial, and some scalar $\eta \in k$. \label{item:case3}
	\end{enumerate}
We will show in each of these cases that some nonzero element of the augmentation ideal of $A$ generates a separable subalgebra. This will produce a contradiction, because all nonzero elements in the augmentation ideal of $k\Mrs$, and hence also in any quotient of $k\Mrs$, are nilpotent and thus cannot generate a separable subalgebra.

In cases \eqref{item:case1} and \eqref{item:case2} we can write $A \cong k[u,v]/\subgrp{u^p+v^2,\sum_{i=m}^n a_i u^{p^i}}$ for some scalars $a_m,\ldots,a_n \in k$ with $0 \leq m < n$ and $a_m,a_n \neq 0$. Then the subalgebra $B$ generated by $x = u^{p^m}$ has the form $B \cong k[x]/\subgrp{\sum_{i=0}^{n-m} a_{m+i} x^{p^i}}$. The polynomial $\sum_{i=0}^{n-m} a_{m+i} x^{p^i} = a_m x + a_{m+1} x^p + \cdots + a_n x^{p^{n-m}}$ is separable, so $B$ is a separable subalgebra of $A$. In case \eqref{item:case3}, if $\eta \neq 0$ then $A \cong k\M_{r',f,\eta} \cong k\M_{r'-1,f^p}$ as $k$-superalgebras by \cite[Remark 3.1.8(4)]{Drupieski:2017b}, so we may assume that $A \cong k\M_{r',f}$ for some $r' \geq 1$ and some inseparable $p$-polynomial $0 \neq f \in k[T]$ that is not a scalar multiple of a single monomial. Suppose $f = \sum_{i=m}^n a_i T^{p^i}$ for some $a_m,\ldots,a_n \in k$ with $1 \leq m < n$ and $a_m,a_n \neq 0$. Then as in cases \eqref{item:case1} and \eqref{item:case2}, the subalgebra $B$ generated by $u_{r'-1}^{p^m}$ has the form $B \cong k[x]/\subgrp{\sum_{i=0}^{n-m} a_{m+i} x^{p^i}}$, and hence is a separable algebra.
\end{proof}

\begin{remark}
In the preceding corollary we made no assumption on the field $k$ other than the standing assumption that its characteristic is odd. BIKP \cite{Benson:2018} classify the elementary finite supergroup schemes under the assumption that the field $k$ is perfect, by applying Koch's classification of Dieudonn\'{e} modules killed by $p$ \cite{Koch:2001}. Under their stronger hypothesis, BIKP deduce that the only isomorphisms among the groups listed in Corollary \ref{cor:infinitesimalelementary} are
	\begin{enumerate}
	\item $\M_{r;1} \cong \Gar \times \Gam$, and
	\item $\Mrseta \cong \M_{r;s,\eta'}$ if and only if $\eta/\eta' = a^{p^{r+s-1}-1}$ for some $a \in k$.
	\end{enumerate}
In particular, if $k$ is algebraically closed, then $\Mrseta \cong \M_{r;s,\eta'}$ for all nonzero $\eta,\eta' \in k^\times$.
\end{remark}

\subsection{Parity change functors} \label{subsec:paritychange}

Recall from \cite[\S2.3.2]{Drupieski:2016} the two parity change functors $\Pi = - \otimes k^{0|1}$ and $\bsPi = k^{0|1} \otimes -$ on the category of $k$-superspaces. On objects, $\Pi$ and $\bsPi$ both act by reversing the $\Z_2$-grading of the underlying superspace. On morphisms, the functors act by $\Pi(\phi) = \phi$ and $\bsPi(\phi) = (-1)^{\ol{\phi}} \phi$, i.e., if $\phi: V \rightarrow W$ is a linear map, then $\Pi(\phi): \Pi(V) \rightarrow \Pi(W)$ is equal to $\phi$ as a function between the underlying sets, while $\bsPi(\phi): \bsPi(V) \rightarrow \bsPi(W)$ is equal to $(-1)^{\ol{\phi}} \phi$. Given a $k$-super\-space $V$ and an element $v \in V$, write $v^\pi$ and ${}^\pi v$ to denote the vector $v$ considered as an element of $\Pi(V)$ and $\bsPi(V)$, respectively. So
	\[
	\Pi(V) = \set{ v^\pi : v \in V} \quad \text{and} \quad \bsPi(V) = \set{ {}^\pi v: v \in V}.
	\]
There are canonical identifications $\Pi(V^\#) \cong \Pi(V)^\#$ and $\bsPi(V^\#) \cong \bsPi(V)^\#$. Specifically, if $\psi \in V^\#$, then we consider $\psi^\pi \in \Pi(V^\#)$ and ${}^\pi \psi \in \bsPi(V^\#)$ as functions $\psi^\pi: \Pi(V) \rightarrow k$ and ${}^\pi \psi: \bsPi(V) \rightarrow k$ via the formulas $(\psi^\pi)(v^\pi) = (-1)^{\ol{v}} \psi(v)$ and $({}^\pi \psi) ({}^\pi v) = (-1)^{\ol{\psi}} \psi(v)$. (Note that $\psi(v) \neq 0$ only if $\ol{\psi} = \ol{v}$.) The signs in these formulas arise from the convention that a symbol $x$ commutes with the superscript $\pi$ up to the sign $(-1)^{\ol{x}}$, and from the convention that ${}^{\pi\pi}x = x = x^{\pi\pi}$.

Now let $A$ be a $k$-superalgebra, and suppose $V$ and $W$ are (left) $A$-supermodules. To extend the functors $\Pi$ and $\bsPi$ to the category $\fsmod_A$ of (left) $A$-supermodules, define the action of $A$ on $\Pi(V)$ and $\bsPi(V)$ by the formulas
	\begin{align} 
	a. (v^\pi) &= (a.v)^\pi, \quad \text{and} \label{eq:PiVaction} \\
	a.({}^\pi v) &= (-1)^{\ol{a}} \cdot {}^\pi(a.v). \label{eq:bsPiVaction}
	\end{align}
Then the maps $(-)^\pi: v \mapsto (-1)^{\ol{v}} v^\pi$ and ${}^\pi(-): v \mapsto {}^\pi v$ define odd $A$-supermodule isomorphisms $V \simeq \Pi(V)$ and $V \simeq \bsPi(V)$.\footnote{Recall that $\phi: V \rightarrow W$ is a left $A$-supermodule homomorphism if $\phi(a.v) = (-1)^{\ol{a} \cdot \ol{\phi}} a.\phi(v)$ for all $a \in A$, $v \in V$.} There are canonical identifications
	\begin{equation} \label{eq:oddhom}
	\Hom_A(\bsPi(V),W)_{\zero} = \Hom_A(V,W)_{\one} = \Hom_A(V,\bsPi(W))_{\zero}
	\end{equation}
defined by pre- and post-composition with ${}^\pi(-)$, respectively.

The parity change functor $\bsPi$ also extends for each affine $k$-supergroup scheme $G$ to the category of rational $G$-supermodules. Let $V$ be a rational $G$-supermodule, and let $\Delta_V: V \rightarrow V \otimes k[G]$ be its (even) comodule structure map. Given $v \in V$, write $\Delta_V(v) = \sum v_0 \otimes v_1$ in the usual Sweedler notation. Then the comodule structure map $\Delta_{\bsPi(V)}: \bsPi(V) \rightarrow \bsPi(V) \otimes k[G]$ is defined by
	\[ \textstyle
	\Delta_{\bsPi(V)}({}^\pi v) = \sum {}^\pi(v_0) \otimes v_1.
	\]
Recall that the induced action of the group algebra $kG$ on $V$ is defined for $\phi \in kG = k[G]^\#$ and $v \in V$ by $\phi.v = (1 \otimes \phi) \circ \Delta_V(v)$. Then the induced action of $kG$ on $\bsPi(V)$ is related to the action of $kG$ on $V$ by the formula \eqref{eq:bsPiVaction}.

\begin{remark}
For consistency with the sign conventions described above, the extension of $\Pi$ to the category of rational $G$-supermodules ought to be defined so that $\Delta_{\Pi(V)}: \Pi(V) \rightarrow \Pi(V) \otimes k[G]$ is defined by $\Delta_{\Pi(V)}(v^\pi) = \sum (-1)^{\ol{v_1}} (v_0)^\pi \otimes v_1$. However, with this formula the induced action of $kG$ on $\Pi(V)$ does not obviously satisfy the sign convention of \eqref{eq:PiVaction}. So we choose only to consider the extension of $\bsPi$ to rational $G$-supermodules.
\end{remark}

\section{Homological dimensions}

\subsection{Projective dimension for \texorpdfstring{$\Pone$}{P1}} \label{subsec:projectiveresolutions}

In this section we describe projective resolutions of the trivial module for the group algebra $k\Mone$ and its polynomial subalgebra $\Pone$. These resolutions arise via Eisenbud's theory of matrix factorizations \cite{Eisenbud:1980}. (Note that $k\Mone$ and $\Pone$ are commutative in the non-super sense, so it makes sense to apply the results of \cite{Eisenbud:1980}.) We then apply the Hopf super\-algebra structure of $\Pone$ to characterize, in terms of the vanishing of a cup product in cohomology, when a $\Pone$-supermodule has finite projective dimension. First we make some general observations concerning the homological dimensions of supermodules.

Given a $k$-superalgebra $A$, let $\fsmod_A$ be the category of left $A$-supermodules, and let $\fmod_A$ be the ordinary category of (arbitrary, not necessarily graded) left $A$-modules. For $V,W \in \fsmod_A$, let $\Hom_A(V,W) = \Hom_{\fsmod_A}(V,W)$ be the $\Z_2$-graded set of left $A$-supermodule homomorphisms from $V$ to $W$, and let $\ulHom_A(V,W) = \Hom_{\fmod_A}(V,W)$ be the set of ordinary $A$-module homomorphisms from $V$ to $W$. Then
	\[
	\ulHom_A(V,W) = \{ \phi \in \Hom_k(V,W) : \phi(a.v) = a.\phi(v) \text{ for all $a \in A$, $v \in V$} \},
	\]
and $\Hom_A(V,W) = \Hom_A(V,W)_{\zero} \oplus \Hom_A(V,W)_{\one}$, where
	\begin{align*}
	\Hom_A(V,W)_{\zero} &= \{ \phi \in \Hom_k(V,W)_{\zero} : \phi(a.v) = a.\phi(v) \text{ for all $a \in A$, $v \in V$} \}, \quad \text{and} \\
	\Hom_A(V,W)_{\one} &= \{ \phi \in \Hom_k(V,W)_{\one} : \phi(a.v) = (-1)^{\ol{a} \cdot \ol{\phi}} a.\phi(v) \text{ for all $a \in A$, $v \in V$} \}.
	\end{align*}
The set $\ulHom_A(V,W)$ inherits a superspace structure from $\Hom_k(V,W)$, and then it immediately follows that $\Hom_A(V,W)_{\zero} = \ulHom_A(V,W)_{\zero}$. On the other hand, let $\pi: W \rightarrow W$ be the parity map defined by $\pi(w) = (-1)^{\ol{w}}w$. Then it is straightforward to check the assignment $\phi \mapsto \pi \circ \phi$ defines an isomorphism $\Hom_A(V,W)_{\one} \cong \ulHom_A(V,W)_{\one}$. Thus for each $V,W \in \fsmod_A$, one gets the superspace isomorphism
	\begin{equation} \label{eq:Homiso}
	\Hom_A(V,W) \cong \ulHom_A(V,W), \quad \phi \mapsto \begin{cases} \phi & \text{if $\phi$ is even,} \\ \pi \circ \phi & \text{if $\phi$ is odd,} \end{cases}
	\end{equation}
which is natural with respect to even homomorphisms in either variable.

Next recall from \cite[\S2.3]{Drupieski:2016a} that the category $\fsmod_A$ is not an abelian category, but its underlying even subcategory $(\fsmod_A)_{\ev}$, consisting of all of the objects of $\fsmod_A$ but only the even $A$-super\-module homomorphisms between them, is an abelian category. Specifically, $(\fsmod_A)_{\ev}$ identifies with the left module category for the smash product algebra $A \# k\Z_2$, where the action of $\Z_2$ on $A$ is defined by having the nontrivial element $\one \in \Z_2$ act on $A$ via the parity automorphism $\pi: A \rightarrow A$. Then $(\fsmod_A)_{\ev}$ contains both enough projectives and enough injectives. Now given $V,W \in \fsmod_A$, the extension groups $\Ext_A^n(V,W)$ are defined as the derived functors of either
	\[
	\Hom_A(V,-): (\fsmod_A)_{\ev} \rightarrow \svec, \quad \text{or equivalently} \quad \Hom_A(-,W): (\fsmod_A)_{\ev} \rightarrow \svec.
	\]
Note that the usual extension groups computed purely within the abelian category $(\fsmod_A)_{\ev}$ are given by just the even subspace $\Ext_A^n(-,-)_{\zero}$ of $\Ext_A^n(-,-)$.

As a left $A$-module, $A \# k\Z_2 = A \otimes k\Z_2$. In particular, $A \# k\Z_2$ is free as a left $A$-module, so any projective resolution in $(\fsmod_A)_{\ev}$ restricts to a projective resolution in $\fmod_A$. Then it follows for each $V,W \in \fsmod_A$ and $n \in \N$ that \eqref{eq:Homiso} extends to an isomorphism of extension groups
	\begin{equation} \label{eq:Extiso}
	\Ext_A^n(V,W) \cong \ulExt_A^n(V,W)
	\end{equation}
where $\ulExt_A^n(V,W)$ denotes the usual extension group in the category $\fmod_A$. Now for $V \in \fsmod_A$, let $\pd_A(V)$ denote the projective dimension of $V$ in the abelian category $(\fsmod_A)_{\ev}$, and let $\ul{\pd}_A(V)$ denote the projective dimension of $V$ in the abelian category $\fmod_A$. Similarly, write $\id_A(V)$ and $\ul{\id}_A(V)$ for the injective dimensions of $V$ in the categories $(\fsmod_A)_{\ev}$ and $\fmod_A$, respectively. Then the isomorphism \eqref{eq:Extiso} implies for $V,W \in \fsmod_A$ that
	\begin{equation} \label{eq:pdidweak}
	\pd_A(V) \leq \ul{\pd}_A(V) \quad \text{and} \quad \id_A(W) \leq \ul{\id}_A(W).
	\end{equation}
Conversely, since projective resolutions in $(\fsmod_A)_{\ev}$ restrict to projective resolutions in $\fmod_A$, it follows that $\ul{\pd}_A(V) \leq \pd_A(V)$, and hence
	\begin{equation} \label{eq:pdequal}
	\pd_A(V) = \ul{\pd}_A(V).
	\end{equation}

\begin{lemma} \label{lemma:injprojdimtorsionmodule}
Let $A$ be a $k$-superalgebra, and let $\fm \subset A$ be a superideal such that $A/\fm \cong k$. Assume that, when the $\Z_2$-grading on $A$ is ignored, $A$ is a commutative noetherian ring in the usual non-super sense. Let $V$ be a finitely-generated $A$-supermodule, and suppose that $V$ is $\fm$-torsion, i.e., suppose for each $v \in V$ that there exists an integer $\ell \geq 1$ such that $\fm^\ell.v = 0$. Then
	\begin{align}
	\pd_A(V) &= \ul{\pd}_A(V) = \sup \{i \in \N: \ul{\Ext}_A^i(V,k) \neq 0 \} = \sup \{i \in \N: \Ext_A^i(V,k) \neq 0 \}, \label{eq:pdequalities} \\
	\id_A(V) &= \ul{\id}_A(V) = \sup \{i \in \N: \ul{\Ext}_A^i(k,V) \neq 0 \} = \sup \{i \in \N: \Ext_A^i(k,V) \neq 0 \}. \label{eq:idequalities}
	\end{align}
\end{lemma}

\begin{proof}
First consider $V$ as an object in the ordinary module category $\fmod_A$, and consider $A$ as an ordinary commutative noetherian ring. Then $\fm$ is a maximal ideal in $A$ (i.e., it is maximal among all, not necessarily $\Z_2$-graded ideals) by the fact that $A/\fm \cong k$ is a field. For each prime ideal $\fm \neq \fp \subset A$, one has $(A-\fp) \cap \fm \neq \emptyset$ by the maximality of $\fm$. Since $V$ is $\fm$-torsion, this implies for each prime ideal $\fm \neq \fp \subset A$ that the localization $V_{\fp}$ is zero. Then the second equality in each of \eqref{eq:pdequalities} and \eqref{eq:idequalities} follows from \cite[\S5.3]{Avramov:1991}, and the third equality in each line is by \eqref{eq:Extiso}. Finally, suppose $\Ext_A^i(k,V) \neq 0$ for some $i \geq 0$. Then $i \leq \id_A(V)$, and hence $\ul{\id}_A(V) \leq \id_A(V)$. Now \eqref{eq:pdidweak} implies that $\id_A(V) = \ul{\id}_A(V)$, and we already know that $\pd_A(V) = \ul{\pd}_A(V)$ by \eqref{eq:pdequal}.
\end{proof}

For the rest of this subsection, let $u$ (resp.\ $v$) be an indeterminate of even (resp.\ odd) superdegree, and let $A$ be either the power series algebra $k[[u,v]]$ or its polynomial subalgebra $k[u,v]$. Set $x = u^p+v^2$, and let $B = A/\subgrp{x}$. Then $B$ is isomorphic to either the group algebra $k\Mone$ or its polynomial subalgebra $\Pone$. Set $F = G = A^{1|1} := A \oplus \Pi(A)$, and consider the elements of $A^{1|1}$ as column vectors whose first (upper) coordinate comes from $A$ and whose second (lower) coordinate comes from $\Pi(A)$.\footnote{Here we use $\Pi$ rather than $\bsPi$ so that no signs are involved in the left action of $A$ on $\Pi(A)$.} Let $\varphi: F \rightarrow G$ and $\psi: G \rightarrow F$ be the (even) $A$-supermodule homomorphisms defined by the matrices
	\[
	\varphi = \begin{pmatrix} u^{p-1} & v \\ v & -u \end{pmatrix} \quad \text{and} \quad \psi = \begin{pmatrix} u & v \\ v & -u^{p-1} \end{pmatrix}.
	\]
Then $\varphi \circ \psi = x \cdot 1_G$ and $\psi \circ \varphi = x \cdot 1_F$, so the pair $(\varphi,\psi)$ is a matrix factorization of $x$ in the sense of \cite[\S5]{Eisenbud:1980}. Next we want to verify that the ideal $\subgrp{x}/\subgrp{x^2} \subset A/\subgrp{x^2}$ is free as a module over $B = A/\subgrp{x}$. Since $A$ is a unique factorization domain, it suffices to show that $x$ is irreducible in $A$. First, $u^p$ is an element of the maximal ideal of $k[[u]]$, so $u^p+v^2$ is a Weierstrass polynomial in $k[[u]][v] \subset k[[u,v]]$. Moreover, a Weierstrass polynomial in $k[[u]][v]$ is irreducible in $k[[u,v]]$ if and only if it is irreducible in $k[[u]][v]$. Since $p$ is odd by our standing assumption on the characteristic of the field $k$, the element $-u^p$ is not a square in either $k[u]$ or $k[[u]]$. Then $v^2+u^p$ has no roots in either $k[u]$ or $k[[u]]$, and hence is irreducible both as an element of $k[u,v] = k[u][v]$ and as an element of $k[[u]][v]$. Then $x$ is irreducible in $A$. Applying \cite[Proposition 5.1]{Eisenbud:1980}, we now get:

\begin{proposition}
Retain the notation of the preceding paragraph. Let $\ol{\varphi}: B^{1|1} \rightarrow B^{1|1}$ and $\ol{\psi}: B^{1|1} \rightarrow B^{1|1}$ denote the (even) $B$-supermodule homomorphisms induced by $\varphi$ and $\psi$, respectively, and let $\ve: B \rightarrow k$ be the augmentation map. Then
\begin{equation} \label{eq:Bfreeresolution}
k \stackrel{\ve}{\longleftarrow} B \stackrel{(u,v)}\longleftarrow B^{1|1} \stackrel{\ol{\varphi}}{\longleftarrow} B^{1|1} \stackrel{\ol{\psi}}{\longleftarrow} B^{1|1} \stackrel{\ol{\varphi}}{\longleftarrow} B^{1|1} \stackrel{\ol{\psi}}{\longleftarrow} \cdots
\end{equation}
is a $B$-free resolution of the trivial module $k$. 
\end{proposition}

\begin{proof}
By \cite[Proposition 5.1]{Eisenbud:1980} and the observations preceding the proposition, the complex
\[
B^{1|1} \stackrel{\ol{\varphi}}{\longleftarrow} B^{1|1} \stackrel{\ol{\psi}}{\longleftarrow} B^{1|1} \stackrel{\ol{\varphi}}{\longleftarrow} B^{1|1} \stackrel{\ol{\psi}}{\longleftarrow} \cdots
\]
is a $B$-free resolution of the $B$-supermodule $\coker(\varphi) = \coker(\ol{\varphi})$. The complex \eqref{eq:Bfreeresolution} is evidently exact at $B$, and $\im(\ol{\varphi}) \subseteq \ker((u,v))$, so it suffices to verify that $\ker((u,v)) \subseteq \im(\ol{\varphi}) = \ker(\ol{\psi})$. So let $\alpha,\beta \in B$, and suppose $u\alpha+v\beta = 0$, i.e., suppose $\binom{\alpha}{\beta^\pi} \in \ker((u,v))$. Then
	\[
	\ol{\psi}\begin{pmatrix} \alpha \\ \beta^\pi \end{pmatrix} = \begin{pmatrix} u & v \\ v & -u^{p-1} \end{pmatrix} \begin{pmatrix} \alpha \\ \beta^\pi \end{pmatrix} = \begin{pmatrix} u\alpha+v\beta \\ (v\alpha-u^{p-1}\beta)^\pi \end{pmatrix} = \begin{pmatrix} 0 \\ (v\alpha-u^{p-1}\beta)^\pi \end{pmatrix}.
	\]
As verified in the paragraph preceding the proposition, $x = u^p+v^2$ is irreducible and hence prime in $A$ (because $A$ is a unique factorization domain), so $B = A/\subgrp{x}$ is a domain. In particular, $v\alpha-u^{p-1}\beta = 0$ in $B$ if and only if $u(v\alpha-u^{p-1}\beta) = 0$. But
	\[
	u(v\alpha-u^{p-1}\beta) = uv\alpha - u^p\beta = uv\alpha + v^2\beta = v(u\alpha+v\beta) = v \cdot 0 = 0.
	\]
Then $\binom{\alpha}{\beta^\pi} \in \ker(\ol{\psi})$, and hence $\ker((u,v)) \subseteq \im(\ol{\varphi})$.
\end{proof}

\begin{corollary} \label{cor:H(B,k)}
Let $B$ be either the group algebra $k\Mone$ or its polynomial subalgebra $\Pone$. Then
	\[
	\opH^i(B,k) = \Ext_B^i(k,k) \cong \begin{cases} k^{1|0} & \text{if $i = 0$,} \\ k^{1|1} & \text{if $i \geq 1$.} \end{cases}
	\]
\end{corollary}

\begin{proof}
Let $P_\bullet$ be the $B$-free resolution of $k$ described in \eqref{eq:Bfreeresolution}. So $P_0 = B$, and $P_i = B^{1|1}$ for $i \geq 1$. Then $\Hbul(B,k)$ can be computed as the cohomology of the cochain complex $\Hom_B(P_\bullet,k)$. The differentials in this complex are all trivial, so the calculation of $\Hbul(B,k)$ follows.
\end{proof}

In the next proposition we use the fact that $\Pone$ is a Hopf superalgebra in order to talk about cup products in cohomology. Given a $\Pone$-supermodule $M$, let $1_M: M \rightarrow M$ be its identity map.

\begin{proposition} \label{prop:Poneprojdim}
Let $M$ and $N$ be $\Pone$-supermodules, and let $0 \neq y \in \opH^1(\Pone,k)_{\one}$. Then the right cup product action of $y$ defines for all $i \geq 2$ an odd isomorphism
	\[
	\Ext_{\Pone}^i(M,N) \simeq \Ext_{\Pone}^{i+1}(M,N).
	\]
In particular, $\pd_{\Pone}(M) = \infty$ if and only if the cup product $1_M \cup y^2 \in \Ext_{\Pone}^2(M,M)$ is nonzero.
\end{proposition}

\begin{proof}
Set $B = \Pone$, and again let $P_\bullet$ be the $B$-free resolution of $k$ described in \eqref{eq:Bfreeresolution}. Then the K\"{u}nneth Theorem and the proof of \cite[Proposition 3.1.5]{Benson:1998} imply that $P_\bullet \otimes M$ is a resolution of $M$ by free $B$-supermodules. Because $B$ is (super)cocommutative, $P_\bullet \otimes M \cong M \otimes P_\bullet$ as a complex of $B$-supermodules, and hence $\Ext_B^\bullet(M,N)$ is the cohomology of the complex $\Hom_B(M \otimes P_\bullet,N)$.

Since $\opH^1(B,k)_{\one}$ is one-dimensional by Corollary \ref{cor:H(B,k)}, we may assume that $y$ is represented by the cochain $\wt{y}: B^{1|1} \rightarrow k$ defined by $\wt{y}\binom{\alpha}{\beta^\pi} = \ve(\beta)$. Here $\ve: B \rightarrow k$ denotes the augmentation map of $B$. Define $\phi: B^{1|1} \rightarrow B^{1|1}$ by the formula
	\[
	\phi \begin{pmatrix} \alpha \\ \beta^\pi \end{pmatrix} = \begin{pmatrix} (-1)^{\ol{\beta}} \beta \\ (-1)^{\ol{\alpha}} \alpha^{\pi} \end{pmatrix}.
	\]
Then $\phi$ is an odd $B$-supermodule isomorphism with $\phi^{-1} = \phi$. A straightforward calculation checks the commutativity of the diagram
	\begin{equation} \label{eq:oddisodiagram}
	\vcenter{\xymatrix{
B^{1|1} \ar@{<-}[r]^{\ol{\psi}} \ar@{->}[d]^{-\phi} & B^{1|1} \ar@{<-}[r]^{\ol{\varphi}} \ar@{->}[d]^{\phi} & B^{1|1} \ar@{<-}[r]^{\ol{\psi}} \ar@{->}[d]^{-\phi} & B^{1|1} \ar@{<-}[r]^{\ol{\varphi}} \ar@{->}[d]^{\phi} & \cdots \\
B^{1|1} \ar@{<-}[r]^{\ol{\varphi}} & B^{1|1} \ar@{<-}[r]^{\ol{\psi}} & B^{1|1} \ar@{<-}[r]^{\ol{\varphi}} & B^{1|1} \ar@{<-}[r]^{\ol{\psi}} & \cdots.
}}
	\end{equation}
Then considering the diagram obtained by applying $M \otimes -$ to \eqref{eq:oddisodiagram}, it follows for $i \geq 2$ that sending a cochain $f: M \otimes P_i \rightarrow N$ to the map $f \circ (1_M \otimes (-1)^i \phi): M \otimes P_{i+1} \rightarrow M \otimes P_i \rightarrow N$ induces an odd isomorphism $\Ext_B^i(M,N) \simeq \Ext_B^{i+1}(M,N)$.

To see that this isomorphism can be realized via the right cup product action of $y$, first note that the cup product action of $y$ on an element $z \in \Ext_B^i(M,N)$ factors as
	\[
	z \cup y = (z \circ 1_M) \cup y = z \circ (1_M \cup y),
	\]
where $\circ$ denotes the Yoneda composition of extensions. The cup product $1_M \cup y$ is represented by the cochain $1_M \otimes \wt{y}: M \otimes B^{1|1} \rightarrow M \otimes k = M$. Next observe that the following diagram commutes:
	\begin{equation} \label{eq:precomposediagram}
	\vcenter{\xymatrix{
\coker(\ol{\varphi}) \ar@{<<-}[r] \ar@{->}[d]^{\wt{y}} & B^{1|1} \ar@{<-}[r]^{\ol{\varphi}} \ar@{->}[d]^{\pi_B \circ \phi} & B^{1|1} \ar@{<-}[r]^{\ol{\psi}} \ar@{->}[d]^{-\phi} & B^{1|1} \ar@{<-}[r]^{\ol{\varphi}} \ar@{->}[d]^{\phi} & B^{1|1} \ar@{<-}[r]^{\ol{\psi}} \ar@{->}[d]^{-\phi} & \cdots \\
k \ar@{<-}[r]^{\ve} & B \ar@{<-}[r]^{(u,v)} & B^{1|1} \ar@{<-}[r]^{\ol{\varphi}} & B^{1|1} \ar@{<-}[r]^{\ol{\psi}} & B^{1|1} \ar@{<-}[r]^{\ol{\varphi}} & \cdots.
}}
	\end{equation}
Here $\pi_B$ denotes the canonical projection map $B^{1|1} = B \oplus \Pi(B) \twoheadrightarrow B$, so that $\pi_B \circ \phi: B^{1|1} \rightarrow B$ is given by $(\pi_B \circ \phi) \binom{\alpha}{\beta^\pi} = (-1)^{\ol{\beta}} \beta$. Also by abuse of notation we have written $\wt{y} : \coker(\ol{\varphi}) \rightarrow k$ for the map canonically induced by $\wt{y}: B^{1|1} \rightarrow k$. Now considering the commutative diagram obtained by applying $M \otimes -$ to \eqref{eq:precomposediagram}, it follows from \cite[Exercise III.6.2]{Mac-Lane:1995} that if $z \in \Ext_B^i(M,N)$ is represented by the cocycle $f: M \otimes P_i \rightarrow N$, then $z \circ (1_M \cup y)$ is represented by the cocycle
	\[
	\begin{cases} f \circ (1_M \otimes (\pi_B \circ \phi)): M \otimes P_1 \rightarrow M \otimes P_0 \rightarrow N, & \text{if $i = 0$,} \\ f \circ (1_M \otimes (-1)^i \phi): M \otimes P_{i+1} \rightarrow M \otimes P_i \rightarrow N, & \text{if $i \geq 1$;} \end{cases}
	\]
Thus the odd isomorphism $\Ext_B^i(M,N) \simeq \Ext_B^{i+1}(M,N)$ of the preceding paragraph can be realized via the right cup product action of $y$.

For the last assertion of the proposition, if $1_M \cup y^2 \neq 0$, then $\Ext_B^2(M,M) \neq 0$, and hence $\Ext_B^i(M,M) \neq 0$ for all $i \geq 2$, which implies that $\pd_{\Pone}(M) = \infty$. On the other hand, if $1_M \cup y^2 = 0$, then for all $B$-supermodules $N$ and all $i \geq 2$, the isomorphism
	\[
	-\cup y^2 = -\circ (1_M \cup y^2): \Ext_B^i(M,N) \cong \Ext_B^{i+2}(M,N)
	\]
is the zero map. Then $\Ext_B^i(M,-) \equiv 0$ for $i \geq 2$, and hence $\pd_B(M) \leq 1$.
\end{proof}

\begin{remark} \label{remark:flatbasechange}
Let $P_\bullet$ be the $\Pone$-free resolution of $k$ described in \eqref{eq:Bfreeresolution}, and let $A \in \calg_k$ be a purely even commutative $k$-algebra. Set $A\Pone = \Pone \otimes_k A$. Then $F := P_\bullet \otimes_k A$ is a resolution of $A$ by free $A\Pone$-supermodules of finite rank. More generally, let $V$ be an $A\Pone$-super\-module, and suppose that $V$ is free of finite rank over $A$, say, $V = M \otimes_k A$ for some finite-dimensional $k$-superspace $M$. Then it follows as in the proof that $V \otimes_A F = M \otimes_k F$ is a resolution of $V$ by free $A\Pone$-super\-modules of finite rank. Now let $A'$ be a purely even commutative $A$-algebra that is flat over $A$. Set $V_{A'} = V \otimes_A A'$, and set $F_{A'} = F \otimes_A A'$. By the exactness of the functor $- \otimes_A A'$, one gets
	\begin{equation} \label{eq:flatbasechange}
\Ext_{A\Pone}^\bullet(V,V) \otimes_A A' = \Hbul(\Hom_{A\Pone}(V \otimes_A F,V) \otimes_A A')
	\end{equation}
Next, since each $V \otimes_A F_i$ is a free $A\Pone$-supermodule of finite rank, and since $V$ is a free over $A$ of finite rank, it follows that the natural map
	\[
	\Phi: \Hom_{A\Pone}(V \otimes_A F,V) \otimes_A A' \rightarrow \Hom_{A\Pone}(V \otimes_A F,V \otimes_A A'),
	\]
defined for $\phi \in \Hom_{A\Pone}(V \otimes_A F,V)$, $a' \in A'$, $v \in V$, and $f \in F$ by
	\[
	\Phi(\phi \otimes_A a')(v \otimes_A f) = \phi(v \otimes_A f) \otimes_A a',
	\]
is an isomorphism. Then
	\begin{align*}
	\Hom_{A\Pone}(V \otimes_A F,V) \otimes_A A' &\cong \Hom_{A\Pone}(V \otimes_A F,V_{A'}) \cong \Hom_{A'\Pone}(V_{A'} \otimes_{A'} F_{A'},V_{A'})
	\end{align*}
as complexes. From this and \eqref{eq:flatbasechange} we deduce that
	\[
	\Ext_{A\Pone}^\bullet(V,V) \otimes_A A' \cong \Ext_{A' \Pone}^\bullet(V_{A'},V_{A'}),
	\]
i.e., flat base change induces an isomorphism in cohomology; cf.\ \cite[I.4.13]{Jantzen:2003}.
\end{remark}

\subsection{Injective and projective dimensions for supergroups} \label{subsec:injprojdim}

Let $G$ be an affine $k$-supergroup scheme. Given a rational $G$-supermodule $V$, define the injective dimension of $V$ in the category of rational $G$-supermodules by
	\[
	\id_G(V) = \sup \set{ i \in \N : \Ext_G^i(-,V) \neq 0 }.
	\]
Similarly, define the projective dimension of $V$ in the category of rational $G$-supermodules by
	\[
	\pd_G(V) = \sup \set{ i \in \N : \Ext_G^i(V,-) \neq 0 }.
	\]
The next two lemmas were stated and proved in \cite[\S3.3]{Drupieski:2017b}.

\begin{lemma} \label{lemma:idequivalent}
Let $G$ be an affine $k$-supergroup scheme, and let $V$ be a rational $G$-supermodule. Then the following are equivalent:
	\begin{enumerate}
	\item $\id_G(V) \leq n < \infty$
	\item \label{item:extvanishing} $\Ext_G^i(L,V) = 0$ for all $i > n$ and all irreducible rational $G$-supermodules $L$.
	\item There exists a resolution of $V$ by rational injective $G$-supermodules,
		\[
		0 \rightarrow V \rightarrow Q_0 \rightarrow Q_1 \rightarrow Q_2 \rightarrow \cdots,
		\]
	such that $Q_i = 0$ for all $i > n$.
	\end{enumerate}
\end{lemma}

\begin{lemma} \label{lemma:idEndV}
Let $G$ be an affine $k$-supergroup scheme, and let $V$ be a finite-dimensional rational $G$-supermodule. Then $\id_G(V) = \id_G(V \otimes V^\#) = \id_G(\Hom_k(V,V))$.
\end{lemma}

In general, an affine supergroup scheme need not admit any nonzero projective rational supermodules (cf.\ \cite[I.3.18]{Jantzen:2003}), so we have the following weaker projective analogue of Lemma \ref{lemma:idequivalent}.

\begin{lemma} \label{lemma:pdequivalent}
Let $G$ be an affine $k$-supergroup scheme, and let $V$ be a finite-dimensional rational $G$-supermodule. Then the following are equivalent:
	\begin{enumerate}
	\item $\pd_G(V) \leq n < \infty$
	\item $\Ext_G^i(V,L) = 0$ for all $i > n$ and all irreducible rational $G$-supermodules $L$.
	\end{enumerate}
\end{lemma}

\begin{proof}
The first statement clearly implies the second. Conversely, suppose (2) holds. By the local finiteness of rational representations, every irreducible rational $G$-supermodule is finite-dimensional. Then arguing by induction on the composition length of $W$, and considering the long exact sequence in cohomology, it follows for all finite-dimensional rational $G$-supermodules $W$ and all $i > n$ that $\Ext_G^i(V,W) = 0$. Now let $W$ be an arbitrary rational $G$-supermodule. Then $W$ is the direct limit of its finite-dimensional $G$-submodules, say, $W = \varinjlim W_j$. Then $V^\# \otimes W = \varinjlim V^\# \otimes W_j$, so arguing as in \cite[I.4.17]{Jantzen:2003}, one gets for all $i > n$ that
	\[
	\Ext_G^i(V,W) \cong \Ext_G^i(k,V^\# \otimes W) = \varinjlim \Ext_G^i(k,V^\# \otimes W_j) \cong \varinjlim \Ext_G^i(V,W_j) = 0,
	\]
and hence $\pd_G(V) \leq n$. (The finite-dimensionality assumption on $V$ is used for the first and last isomorphisms in the preceding equation; cf.\ \cite[I.4.4]{Jantzen:2003} or \cite[Lemma 2.3.4]{Drupieski:2016a}.)
\end{proof}

\begin{lemma} \label{lemma:idpdequal}
Let $G$ be an affine $k$-supergroup scheme, and let $V$ be a finite-dimensional rational $G$-supermodule. Then
	\[
	\id_G(V) = \id_G(V^\#) = \pd_G(V) = \pd_G(V^\#).
	\]
\end{lemma}

\begin{proof}
First, $\id_G(V) = \id_G(V \otimes V^\#)$ by Lemma \ref{lemma:idEndV}. But $V \otimes V^\# \cong V^\# \otimes V \cong V^\# \otimes (V^\#)^\#$ as $G$-supermodules, so applying Lemma \ref{lemma:idEndV} again we get $\id_G(V) = \id_G(V^\# \otimes (V^\#)^\#) = \id_G(V^\#)$. Next, as observed in the proof of Lemma \ref{lemma:pdequivalent}, the irreducible rational $G$-supermodules are all finite-dimensional. This implies for $L$ irreducible that $\Ext_G^i(L,V) \cong \Ext_G^i(V^\#,L^\#)$, and that up to isomorphism the set of irreducible rational $G$-super\-modules is closed under the operation $L \mapsto L^\#$ of taking linear duals. Then the following statements are equivalent:
	\begin{itemize}
	\item $\Ext_G^i(L,V) = 0$ for all $i > n$ and all irreducible rational $G$-supermodules $L$.
	\item $\Ext_G^i(V^\#,L) = 0$ for all $i > n$ and all irreducible rational $G$-supermodules $L$.
	\end{itemize}
By Lemmas \ref{lemma:idequivalent} and \ref{lemma:pdequivalent}, this implies that $\id_G(V) = \pd_G(V^\#)$. Then replacing $V$ with $V^\#$, and using the $G$-supermodule isomorphism $V \cong (V^\#)^\#$, one gets $\id_G(V^\#) = \pd_G(V)$.
\end{proof}

If $G$ is a finite $k$-supergroup scheme, then the previous lemma is a direct consequence of fact that $kG$ is a Hopf superalgebra, and hence self-injective; cf.\ \cite[Lemma 2.3.2]{Drupieski:2016a}.

\begin{lemma} \label{lemma:idGunipotent}
Let $G$ be a unipotent affine $k$-supergroup scheme, and let $V$ be a rational $G$-super\-module. Then
	\[
	\id_G(V) = \sup \{ i \in \N : \Ext_G^i(k,V) \neq 0 \}.
	\]
If $V$ is finite-dimensional, then
	\[
	\id_G(V) = \sup \{ i \in \N : \Ext_G^i(V,V) \neq 0 \} \quad
	\text{and} \quad \pd_G(V) = \sup \{ i \in \N : \Ext_G^i(V,k) \neq 0 \}.
	\]
\end{lemma}

\begin{proof}
The first characterization of $\id_G(V)$ follows from Lemma \ref{lemma:idequivalent} and the fact that, up to isomorphism and parity change, the trivial module $k$ is the unique irreducible rational $G$-super\-module. For the second characterization of $\id_G(V)$, observe now by Lemma \ref{lemma:idEndV} that
	\[
	\id_G(V) = \id_G(V \otimes V^\#) = \sup \{ n \in \N : \Ext_G^n(k,V \otimes V^\#) \neq 0 \}.
	\]
Since $\Ext_G^\bullet(k,V \otimes V^\#) \cong \Ext_G^\bullet(V,V)$, the result for $\id_G(V)$ follows. The argument for $\pd_G(V)$ is entirely similar to the argument for $\id_G(V)$, using instead Lemma \ref{lemma:pdequivalent}.
\end{proof}

\subsection{Homological dimensions for \texorpdfstring{$\Mone$, $\Pone$, and $k\Mone$}{M1, P1, and kM1}} \label{subsec:injprojdimM1}

In this section we investigate how the injective and projective dimensions of a finite-dimensional rational $\Mone$-supermodule $V$ are related to its injective and projective dimensions in the categories of $\Pone$- and $k\Mone$-supermodules.

For the rest of this section write $k\Mone = k[[u,v]]/\subgrp{u^p+v^2}$. Recall that $k[\Mone] = \bigcup_{s \geq 1} k[\Mones]$, where $k[\Mones]$ is the finite-dimensional Hopf subsuperalgebra of $k[\Mone]$ generated by $\tau$ and $\sigma_i$ for $0 \leq i < p^s$. Then if $V$ is a finite-dimensional rational $\Mone$-supermodule, it follows that the comodule structure map $V \rightarrow V \otimes k[\Mone]$ has image in $V \otimes k[\Mones]$ for some $s \geq 1$. This implies that the action of $\Mone$ on $V$ factors through the quotient $\Mone \twoheadrightarrow \Mones$, and the induced action of $k\Mone$ on $V$ factors through the quotient $k\Mone \twoheadrightarrow k\Mone/\subgrp{u^{p^s}} \cong k\Mones$. Thus if $M$ and $N$ are finite-dimensional rational $\Mone$-supermodules, we can consider them both as $\Mones$-super\-modules (equivalently, as $k\Mones$-super\-modules) for some integer $s \geq 1$, with the $\Mone$- and $k\Mone$-supermodule structures coming from the quotient maps $\Mone \twoheadrightarrow \Mones$ and $k\Mone \twoheadrightarrow k\Mones$. Conversely, any $\Mones$-super\-module can be lifted to a rational $\Mone$-supermodule via the quotient $\Mone \twoheadrightarrow \Mones$.

\begin{lemma} \label{lemma:Extrestrictioniso}
Let $s \geq 1$. Let $M$ be a $k\Mones$-supermodule, viewed as a $k\Mone$-super\-module via the canonical quotient map $k\Mone \twoheadrightarrow k\Mones$, and let $N$ be a $k\Mone$-supermodule. Then restriction to $\Pone$ defines an isomorphism
	\[
	\Ext_{k\Mone}^\bullet(M,N) \cong \Ext_{\Pone}^\bullet(M,N).
	\]
In particular, if $M$ and $N$ are finite-dimensional rational $\Mone$-supermodules, then
	\[
	\Ext_{k\Mone}^\bullet(M,N) \cong \Ext_{\Pone}^\bullet(M,N) \cong \Ext_{\Mone}^\bullet(M,N).
	\]
\end{lemma}

\begin{proof}
The algebra $k\Mone = k[[u,v]]/\subgrp{u^p+v^2}$ is the completion of $\Pone = k[u,v]/\subgrp{u^p+v^2}$ at the ideal $\subgrp{u}$ generated by $u$. Since $\Pone$ is a noetherian ring, this implies that $k\Mone$ is a flat $\Pone$-algebra, and hence that the functor $k\Mone \otimes_{\Pone} -$ is exact \cite[Proposition 10.14]{Atiyah:1969}. Let $P_\bullet \rightarrow M$ be a resolution of $M$ by projective $\Pone$-supermodules. By the exactness of the functor $k\Mone \otimes_{\Pone} -$, the complex $k\Mone \otimes_{\Pone} P_\bullet$ is a resolution of $k\Mone \otimes_{\Pone} M$ by projective $k\Mone$-supermodules. Now since $k\Mone$ acts on $M$ via the quotient map $k\Mone \rightarrow k\Mones$, it follows that $k\Mone \otimes_{\Pone} M = k\Mones \otimes_{k\Mones} M = M$ as a $k\Mone$-supermodule. Then $k\Mone \otimes_{\Pone} P_\bullet$ is a resolution of $M$ by projective $k\Mone$-supermodules, and the map $\Phi: P_\bullet \rightarrow k\Mone \otimes_{\Pone} P_\bullet$ defined by $\Phi(z) = 1 \otimes_{\Pone} z$ is a $\Pone$-supermodule chain map that lifts the identity on $M$. Since $\Phi$ induces the evident isomorphism of cochain complexes
	\[
	\Hom_{k\Mone}(k\Mone \otimes_{\Pone} P_\bullet,N) \cong \Hom_{\Pone}(P_\bullet,N),
	\]
it follows that restriction to $\Pone$ defines an isomorphism $\Ext_{k\Mone}^\bullet(M,N) \cong \Ext_{\Pone}^\bullet(M,N)$. This proves the first assertion of the lemma, and then the second is by \cite[Proposition 3.3.6]{Drupieski:2017b}.
\end{proof}

\begin{proposition} \label{prop:injprojdimequalities}
Let $V$ be a finite-dimensional rational $\Mone$-supermodule. Then
	\[
	\id_{k\Mone}(V) = \id_{\Pone}(V) = \id_{\Mone}(V) = \pd_{\Mone}(V) = \pd_{\Pone}(V) = \pd_{k\Mone}(V).
	\]
\end{proposition}

\begin{proof}
By Lemma \ref{lemma:Extrestrictioniso} there are isomorphisms of graded superspaces
	\begin{align*}
	\Ext_{k\Mone}^\bullet(k,V) &\cong \Ext_{\Pone}^\bullet(k,V) \cong \Ext_{\Mone}^\bullet(k,V), \quad \text{and} \\
	\Ext_{k\Mone}^\bullet(V,k) &\cong \Ext_{\Pone}^\bullet(V,k) \cong \Ext_{\Mone}^\bullet(V,k).
	\end{align*}
Furthermore, it follows from the discussion preceding Lemma \ref{lemma:Extrestrictioniso} that $V$ is $\fm$-torsion, where $\fm$ is the maximal ideal of $k\Mone$ (resp.\ of $\Pone$) generated by $u$ and $v$. Then the asserted equalities of injective and projective dimensions follow from Lemmas \ref{lemma:injprojdimtorsionmodule}, \ref{lemma:idGunipotent}, and \ref{lemma:idpdequal}.
\end{proof}

The next lemma is an analogue of \cite[Proposition 2.2]{Friedlander:2005}.

\begin{proposition} \label{prop:secondradical}
Let $\alpha,\beta,\gamma,\delta$ be indeterminates with $\alpha,\beta,\gamma$ of even superdegree and $\delta$ of odd super\-degree. Let $s$ and $t$ be positive integers, let $R = k[\alpha,\beta,\gamma,\delta]/\subgrp{\alpha^p+\delta^2,\beta^p,\alpha^{p^s},\gamma^{p^t}}$, and let $M$ be a finitely-generated $R$-supermodule. Let $\sigma_{\alpha},\sigma_{\alpha+\beta\gamma}: \Pone \rightarrow R$ be the $k$-super\-algebra homomorphisms defined by $\sigma_{\alpha}(u) = \alpha$ and $\sigma_{\alpha}(v) = \delta$, and $\sigma_{\alpha+\beta\gamma}(u) = \alpha+\beta\gamma$ and $\sigma_{\alpha+\beta\gamma}(v) = \delta$, respectively, and let $M\da{\alpha}$ and $M\da{\alpha+\beta\gamma}$ denote the pullbacks of $M$ along $\sigma_\alpha$ and $\sigma_{\alpha+\beta\gamma}$. Then
	\[
	\pd_{\Pone}(M\da{\alpha}) < \infty \quad \text{if and only if} \quad \pd_{\Pone}(M\da{\alpha+\beta\gamma}) < \infty.
	\]
\end{proposition}

\begin{proof}
By \eqref{eq:pdequal} we may ignore the $\Z_2$-gradings on $\Pone$ and $M$, so for the duration of the proof we operate purely in the context of ordinary commutative algebra.

Set $P = k[\alpha,\beta,\gamma,\delta]$, let $\fn = \subgrp{\alpha,\beta,\gamma,\delta} \subset P$, and let $I = \subgrp{\alpha^p+\delta^2,\beta^p} \subset P$. Then $M$ is a $P/I$-module via the evident quotient map $P/I \rightarrow R$. Set $\wtalpha = \alpha + \beta \gamma$, let $f = \alpha^p+\delta^2$, and let $g = \wtalpha^p + \delta^2$. Then $f-g = \beta^p \gamma^p \in \fn I$, so by \cite[Theorem 2.1]{Avramov:2018} there exists for each $i \in \N$ an isomorphism $\Tor_i^{P/\subgrp{f}}(k,M) \cong \Tor_i^{P/\subgrp{g}}(k,M)$. Since $M$ is $\fn$-torsion and $\fn$ is a maximal ideal in $P$, this implies by \cite[\S5.3]{Avramov:1991} that $\fd_{P/\subgrp{f}}(M) = \fd_{P/\subgrp{g}}(M)$, where $\fd_A(M)$ denotes the flat dimension of $M$ as an $A$-module. But flat dimension and projective dimension coincide for finitely-generated modules over noetherian rings \cite[Corollary 2.10F]{Avramov:1991}, so
	\[
	\pd_{P/\subgrp{f}}(M) = \pd_{P/\subgrp{g}}(M).
	\]
Next set $Q = k[\alpha,\delta]/\subgrp{\alpha^p+\delta^2}$, and set $Q' = k[\wtalpha,\delta]/\subgrp{\wtalpha^p+\delta^2}$. Then $P/\subgrp{f} = Q[\beta,\gamma]$ is a polynomial extension of $Q$, and $P/\subgrp{g} = Q'[\beta,\gamma]$ is a polynomial extension of $Q'$, so
	\begin{align*}
	\pd_Q(M) &< \infty \quad \text{if and only if} \quad \pd_{P/\subgrp{f}}(M) < \infty, \quad \text{and} \\
	\pd_{Q'}(M) &< \infty \quad \text{if and only if} \quad \pd_{P/\subgrp{g}}(M) < \infty;
	\end{align*}
cf.\ \cite[\S5.6]{Avramov:2018}. Now $\pd_Q(M) = \pd_{\Pone}(M\da{\alpha})$ and $\pd_{Q'}(M) = \pd_{\Pone}(M\da{\alpha+\beta\gamma})$, so the result follows.
\end{proof}

\section{Support schemes}\label{section:four}

\subsection{The functor of multiparameter supergroups} \label{subsec:Vrg}

Given an affine $k$-supergroup scheme $G \in \sgrp_k$ and a commutative $k$-superalgebra $A \in \csalg_k$, let $G_A = G \otimes_k A \in \sgrp_A$ denote the affine $A$-super\-group scheme obtained from $G$ via base change to $A$, i.e., the affine $A$-supergroup scheme with coordinate Hopf $A$-superalgebra $A[G_A] = k[G] \otimes_k A$. If $G$ is finite, then $A[G_A]$ is a free $A$-super\-module of finite rank, and there are canonical identifications
	\[
	\Hom_A(A[G_A],A) \cong \Hom_k(k[G],A) \cong A \otimes_k k[G]^\# \cong A \otimes_k kG.
	\]
Thus if $G$ is finite, the group algebra $AG_A$ of $G_A$ identifies as a Hopf $A$-superalgebra with $kG \otimes_k A$.

Next recall that given affine $k$-supergroup schemes $G$ and $H$, the $k$-superfunctor
	\[
	\bfHom(G,H): \csalg_k \rightarrow \sets
	\]
is defined by
	\[
	\bfHom(G,H)(A) = \Hom_{Grp/A}(G_A,H_A),
	\]
the set of $A$-supergroup scheme homomorphisms $\rho: G_A \rightarrow H_A$. In \cite[Theorem 3.3.6]{Drupieski:2017a}, we showed that if $H$ is algebraic and if $G$ is a multiparameter $k$-supergroup scheme, then $\bfHom(G,H)$ admits the structure of an affine $k$-superscheme of finite type over $k$. If $G$ and $H$ are both finite, then the set of Hopf $A$-superalgebra homomorphisms $\rho: A[H_A] \rightarrow A[G_A]$ identifies by duality with the set of Hopf $A$-superalgebra homomorphisms $\rho: AG_A \rightarrow AH_A$, and hence
	\[
	\bfHom(G,H)(A) \cong \Hom_{Hopf/A}(AG_A,AH_A) \quad \text{for $G,H$ finite.}
	\]
	
\begin{lemma} \label{lemma:HomRS}
Let $R$ and $S$ be Hopf $k$-superalgebras such that $R$ is finitely-generated as a $k$-algebra and $S$ is finite-dimensional over $k$. Define the $k$-superfunctor $\bfHom(R,S): \csalg_k \rightarrow \sets$ by
	\[
	\bfHom(R,S)(A) = \Hom_{Hopf/A}(R \otimes_k A,S \otimes_k A).
	\]
Then $\bfHom(R,S)$ admits the structure of an affine $k$-super\-scheme of finite type. With this structure, the assignment $S \mapsto \bfHom(R,S)$ is a covariant functor from the category of finite-dimensional Hopf $k$-super\-algebras to the category of affine $k$-superschemes that takes injections to closed embeddings. Similarly, the assignment $R \mapsto \bfHom(R,S)$ takes surjections to closed embeddings.
\end{lemma}

\begin{proof}
Let $r_0,r_1,\ldots,r_m$ be a homogeneous generating set for $R$, and let $s_0,s_1,\ldots,s_n$ be a homo\-geneous basis for $S$. Assume that $r_0$ (resp.\ $s_0$) is the identity element of $R$ (resp.\ $S$), and that the remaining elements generate the augmentation ideals of their respective rings. For $1 \leq i \leq m$ and $1 \leq j \leq n$, let $x_{ij}$ be an indeterminant of superdegree $\ol{r_i} + \ol{s_j}$, and let $T$ be the (super)commutative polynomial $k$-super\-algebra generated by the $x_{ij}$. If $A \in \csalg_k$ and if $\rho : R \otimes_k A \rightarrow S \otimes_k A$ is a homomorphism of Hopf $A$-superalgebras, then there exists a unique $k$-superalgebra homomorphism $\wtrho: T \rightarrow A$ such that $\rho(r_i) = \sum_{j=1}^n s_j \otimes \wtrho(x_{ij}) \in S \otimes_k A$ for each $1 \leq i \leq m$, and $\rho$ is completely determined by $\wtrho$. (Automatically, $\rho(r_0) = s_0$.) We will argue that there are certain polynomials in $T$, depending only on the Hopf superalgebra structures of $R$ and $S$, on which $\wtrho$ must vanish, but that if a $k$-super\-algebra map $\wtsigma: T \rightarrow A$ vanishes on the ideal $J \subset T$ generated by those polynomials, then the assignments $r_i \mapsto \sum_{j=1}^n s_j \otimes \wtsigma(x_{ij})$ for $1 \leq i \leq m$ extend uniquely to a homomorphism of Hopf $A$-superalgebras $\sigma: R \otimes_k A \rightarrow S \otimes_k A$.

Let $\rho$ and $\wtrho$ be as above, and let $f = f(t_0,\ldots,t_m)$ be a polynomial over $k$ in the non-commuting variables $t_0,\ldots,t_m$. Then applying the algebra relations in $S$ and the (super)commutativity of $A$, it follows that there exist polynomials $f_1,\ldots,f_n \in T$, depending only on $f$ and the algebra relations in $S$, such that
	\begin{align*}
	\rho(f(r_0,\ldots,r_m)) = f(\rho(r_0),\ldots,\rho(r_m)) &= \sum_{j=1}^n s_j \otimes f_j(\wtrho(x_{11}),\ldots,\wtrho(x_{mn})) \\
	&= \sum_{j=1}^n s_j \otimes \wtrho(f_j(x_{11},\ldots,x_{mn})).
	\end{align*}
In particular, if $f(r_0,\ldots,r_m)$ is an algebra relation in $R$, then it must be the case that $\wt{\rho}(f_j) = 0$ for each $j$. Conversely, if $\wtsigma: T \rightarrow A$ is a $k$-superalgebra homomorphism such that $\wtsigma(f_j) = 0$ for each $j$ and each algebra relation $f$ of $R$, then it follows that the assignments $r_i \mapsto \sum_{j=1}^n s_j \otimes \wtsigma(x_{ij})$ uniquely extend to an $A$-super\-algebra homomorphism $\sigma: R \otimes_k A \rightarrow S \otimes_k A$.

Next, write
	\[
	\Delta(r_i) = \sum_{(r_i)} r_{i(1)}(r_0,\ldots,r_m) \otimes r_{i(2)}(r_0,\ldots,r_m)
	\]
for some polynomials $r_{i(1)}(t_0,\ldots,t_m)$ and $r_{i(2)}(t_0,\ldots,t_m)$ in the non-commuting variables $t_0,\ldots,t_m$. Then making the canonical identification $(S \otimes_k A) \otimes_A (S \otimes_k A) \cong (S \otimes_k S) \otimes_k A$, and applying the algebra relations in $S$ and the commutativity of $A$, it follows that
	\begin{align*}
	(\rho \otimes \rho) \circ \Delta(r_i) &= \sum_{(r_i)} r_{i(1)}(\rho(r_0),\ldots,\rho(r_m)) \otimes r_{i(2)}(\rho(r_0),\ldots,\rho(r_m)) \\
	&= \sum_{c,d=0}^n (s_c \otimes s_d) \otimes g_{cd}^i(\wtrho(x_{11}),\ldots,\wtrho(x_{mn})) \\
	&= \sum_{c,d=0}^n (s_c \otimes s_d) \otimes \wtrho(g_{cd}^i(x_{11},\ldots,x_{mn}))
	\end{align*}
for some polynomials $g_{cd}^i \in T$ that depend only on the algebra structure of $S$. On the other hand, write the coproduct in $S$ as $\Delta (s_j) = \sum_{c,d=0}^n (s_c \otimes s_d) \cdot b_{cd}^j$ with $b_{cd}^j \in k$. Then
	\[
	\Delta \circ \rho(r_i) = \Delta\left(\sum_{j=1}^n s_j \otimes \wtrho(x_{ij}) \right) = \sum_{c,d=0}^n \left((s_c \otimes s_d) \otimes \left( \sum_{j=1}^n b_{cd}^j \cdot \wtrho(x_{ij})\right)\right).
	\]
Since $(\rho \circ \rho) \circ \Delta = \Delta \circ \rho$, it must be the case that $\wtrho(g_{cd}^i-\sum_{j=1}^n b_{cd}^j \cdot x_{ij}) = 0$ for each $1 \leq i \leq m$ and each $0 \leq c,d \leq n$. Conversely, if $\wtsigma: T \rightarrow A$ vanishes on each $g_{cd}^i-\sum_{j=1}^n b_{cd}^j \cdot x_{ij}$ and on each $f_j$ from the previous paragraph, then $\wtsigma$ lifts to an $A$-superalgebra homomorphism $\sigma: R \otimes_k A \rightarrow S \otimes_k A$ such that $(\sigma \otimes \sigma) \circ \Delta(r_i) = \Delta \circ \sigma(r_i)$ for each $1 \leq i \leq m$, and hence such that  $(\sigma \otimes \sigma) \circ \Delta = \Delta \circ \sigma$ in general because the $r_i$ generate $R$ and because $\Delta$ and $\sigma$ are algebra maps.

In a similar fashion to the previous two paragraphs, one can show that $\sigma$ is compatible with the antipodes on $R \otimes_k A$ and $S \otimes_k A$ if and only if $\wtsigma$ vanishes on additional polynomials in $T$ that depend only on the Hopf superalgebra structures of $R$ and $S$; we leave the details of this verification to the reader. Taking $J$ to be the ideal in $T$ generated by these additional polynomial relations and the polynomial relations from the previous two paragraphs, this shows that $\bfHom(R,S)$ is an affine $k$-superscheme represented by $T/J$.

Now suppose $S \hookrightarrow S'$ is an injective Hopf superalgebra homomorphism. We can extend the given homo\-geneous basis $s_0,\ldots,s_n$ for $S$ to a homogeneous basis $s_0,\ldots,s_n,s_{n+1},\ldots,s_{n'}$ for $S'$. Then $\bfHom(R,S)$ is the closed subsuperscheme of $\bfHom(R,S')$ defined by the vanishing of the coordinate functions $x_{ij}$ for $n+1 \leq j \leq n'$. Similarly, if $R \twoheadrightarrow R'$ is a quotient map of Hopf superalgebras with kernel $I$, then $\bfHom(R',S)$ is the closed subsuperscheme of $\bfHom(R,S)$ defined by the vanishing of the polynomials $f_1,\ldots,f_n \in T$ that arise from each additional algebra relation $f \in I$. 
\end{proof}

\begin{definition}
Given a finite $k$-supergroup scheme $G$, define $\bsvrg: \csalg_k \rightarrow \sets$ by
	\begin{align*}
	\bsvrg(A) &= \Hom_{Hopf/A}(\Pr \otimes_k A, kG \otimes_k A),
	\end{align*}
the set of Hopf $A$-superalgebra homomorphisms $\rho: \Pr \otimes_k A \rightarrow kG \otimes_k A$. We call $\bsvrg$ the \emph{functor of multiparameter supergroups of height $\leq r$ in $G$} because of Lemma \ref{lemma:bsvrg}\eqref{item:stratification} below.
\end{definition}

Given finite $k$-supergroup schemes $E$ and $E'$, write $E \succ E'$ if $kE'$ is a proper Hopf superalgebra quotient of $kE$, i.e., if there exists a surjective Hopf superalgebra homomorphism $kE \twoheadrightarrow kE'$ and $kE \not\cong kE'$. If $G$ is a finite $k$-supergroup scheme and if $E \succ E'$, then the quotient map $kE \twoheadrightarrow kE'$ defines a closed embedding $\bfHom(E',G) \hookrightarrow \bfHom(E,G)$. Now given $E$ and $G$ as above, set
	\[
	\bfHom(E,G)^+(A) = \bfHom(E,G)(A) - \left( \bigcup_{E \succ E'} \bfHom(E',G)(A) \right).
	\]

\begin{lemma} \label{lemma:bsvrg}
Let $G$ be a finite $k$-supergroup scheme.
	\begin{enumerate}
	\item \label{item:bsvrgaffine} The $k$-superfunctor $\bsvrg$ admits the structure of an affine $k$-superscheme of finite type. Then the assignment $G \mapsto \bsvrg$ is a covariant functor from the category of finite $k$-super\-group schemes to the category of affine $k$-super\-schemes of finite type that takes closed embeddings to closed embeddings.
	
	\item \label{item:stratification} Let $E$ be a multiparameter $k$-supergroup scheme of height $\leq r$. Then the canonical quotient map $\Pr \twoheadrightarrow kE$ induces a closed embedding of affine $k$-super\-schemes $\bfHom(E,G) \hookrightarrow \bsvrg$. Identifying $\bfHom(E,G)$ with its image in $\bsvrg$ via this embedding, one has
		\[
		\bsvrg(k) = \bigcup_{(E \leq G)/\cong} \bfHom(E,G)(k) = \coprod_{(E \leq G)/\cong} \bfHom(E,G)^+(k)
		\]
	where the union (resp.\ disjoint union) is taken over the isomorphism classes of multiparameter $k$-subsupergroups $E$ of height $\leq r$ that occur as closed subsupergroups of $G$.
	
	\item \label{item:union} Similarly to part \eqref{item:stratification},
		\[
		\bsvrg(k) = \bigcup_{E \leq G} \bsvr(E)(k),
		\]
	where now the union is over all multiparameter (closed) $k$-subsupergroups $E$ of height $\leq r$ in $G$, and $\bsvr(E)$ is identified with its image under the closed embedding $\bsvr(E) \hookrightarrow \bsvrg$.
	\end{enumerate}
\end{lemma}

\begin{proof}
Part \eqref{item:bsvrgaffine} is a rephrasing in this context of Lemma \ref{lemma:HomRS}. The decompositions in parts \eqref{item:stratification} and \eqref{item:union} then follow from Proposition \ref{prop:Prquotients} and the fact that $E$ is a closed subsupergroup of $G$ if and only if $kE$ is a Hopf subsuperalgebra of $kG$.
\end{proof}

\begin{notation}
Given a finite $k$-supergroup scheme $G$, let $\Vrg = \bsvrg_{\ev}$ be the underlying purely even subscheme of $\bsvrg$. Then
	\[
	k[\Vrg] = k[\bsvrg]/\subgrp{k[\bsvrg]_{\one}}
	\]
is the largest purely even quotient of $k[\bsvrg]$, and for each commutative $k$-superalgebra $A \in \csalg_k$, one has $\Vrg(A) = \Vrg(\Azero) = \bsvrg(\Azero)$.
\end{notation}

\begin{remark} \label{remark:Vrgpurelyeven}
Let $G$ be a purely even finite $k$-supergroup scheme (i.e., an ordinary finite $k$-group scheme), and let $A = \Azero$ be a purely even commutative $k$-superalgebra (i.e., an ordinary commutative $k$-algebra). Then $kG \otimes_k A$ is a purely even $k$-algebra, and hence any Hopf $A$-superalgebra homomorphism $\rho: \Pr \otimes_k A \rightarrow kG \otimes_k A$ factors through the canonical quotient map
	\[
	\Pr \otimes_k A \twoheadrightarrow \Pr/\subgrp{v} \otimes_k A \cong k\Gar \otimes_k A.
	\]
This implies that $\bsvrg(A) = \Vrg(A) \cong \Hom_{Grp/A}(\Gar \otimes_k A,G \otimes_k A)$. Thus when $G$ is purely even, our use of the notation $\Vrg$ agrees with that of Suslin, Friedlander, and Bendel \cite{Suslin:1997} (although they allow $G$ to be an arbitrary affine algebraic $k$-group scheme, whereas we assume that $G$ is finite in order to obtain an affine superscheme structure on $\bsvrg$).
\end{remark}

\begin{lemma} \label{lemma:kbsvrg=kNrG}
Let $G$ be a unipotent finite $k$-supergroup scheme. Then there exists an integer $s = s(G) \geq 1$ such that for all $s' \geq s$, the canonical quotient maps $\Pr \twoheadrightarrow k\M_{r;s'} \twoheadrightarrow k\Mrs$ induce identifications
	\[
	\bfHom(\Mrs,G) = \bfHom(\M_{r;s'},G) = \bsvrg.
	\]
In particular, if the field $k$ is algebraically closed, then $k[\bsvrg]_{red}$, the largest (purely even) reduced quotient of $k[\bsvrg]$, identifies with the algebra $k[\NrG]$ of \cite[Definition 2.3.6]{Drupieski:2017b}.
\end{lemma}

\begin{proof}
Let $I$ be the augmentation ideal of the group algebra $kG$. By the assumption that $G$ is unipotent, $kG$ is local and $I$ is nilpotent. Let $s = s(G) \geq 0$ be the minimal integer such that $I^{p^s} = 0$. Then for any $A \in \csalg_k$, the augmentation ideal of $kG \otimes_k A$ is $I \otimes_k A$, and $(I \otimes_k A)^{p^s} = 0$. Now if $\rho: \Pr \otimes_k A \rightarrow kG \otimes_k A$ is a homomorphism of Hopf $A$-superalgebras, then $\rho(u_{r-1} \otimes 1)$ is an element of the augmentation ideal of $kG \otimes_k A$, and so $\rho(u_{r-1} \otimes 1)^{p^s} = 0$. Then $\rho$ factors through the canonical quotient map $\Pr \otimes_k A \twoheadrightarrow k\Mrs \otimes_k A$, which means that $\rho \in \bfHom(\Mrs,G)(A)$, and hence $\bsvrg(A) \subseteq \bfHom(\Mrs,G)(A)$. This implies the first assertion of the lemma.

Next, by definition, $k[\NrG] = k[\bfHom(\M_{r;N},G)]_{red}$, where $N \geq 1$ is the minimal integer such that for all $N' \geq N$ and all field extensions $K/k$, the canonical quotient map $\M_{r;N'} \twoheadrightarrow \M_{r;N}$ induces an identification $\bfHom(\M_{r;N},G)(K) = \bfHom(\M_{r;N'},G)(K)$. Evidently $s(G) \geq N$, because we get stabilizations not just of the field-valued points but of the entire $k$-superfunctors at $s(G)$. But for $k$ algebraically closed, we get by \cite[Remark 2.3.7(2)]{Drupieski:2017b} that
	\[
	k[\bfHom(\M_{r;N},G)]_{red} = k[\bfHom(\M_{r;s(G)},G)]_{red} = k[\Vrg]_{red}. \qedhere
	\]
\end{proof}

\begin{lemma} \label{lemma:Zpr2grading}
Let $G$ be a finite $k$-supergroup scheme. Then $k[\bsvrg]$ admits a nonnegative algebra grading by $\Z[\frac{p^r}{2}]$, which in turn induces a corresponding grading on $k[\Vrg]$. If $G \rightarrow G'$ is a homo\-morphism of finite $k$-super\-group schemes, then the induced algebra maps $k[\bsvr(G')] \rightarrow k[\bsvrg]$ and $k[V_r(G')] \rightarrow k[\Vrg]$ respect the gradings.
\end{lemma}

\begin{proof}
Let $A \in \csalg_k$, and let $\mu,a \in \Azero$ with $a^{p^r} = \mu^2$. Then there exists a Hopf $A$-superalgebra homomorphism $\phi = \phi_{\mu,a}: \Pr \otimes_k A \rightarrow \Pr \otimes_k A$ such that $\phi^*(v) = v \cdot \mu$ and $\phi^*(u_i) = u_i \cdot a^{p^i}$ for $0 \leq i \leq r-1$. If $\mu',a' \in \Azero$ also satisfy $(a')^{p^r} = (\mu')^2$, then $\phi_{\mu,a} \circ \phi_{\mu',a'} = \phi_{\mu \mu', aa'}$. Next, let $B = k[x,y]/\subgrp{x^{p^r}-y^2}$ be the purely even bialgebra generated by the grouplike elements $x$ and $y$. Then $\Spec(B)$ admits the structure of a purely even unital associative monoid $k$-scheme. The set of $A$-points of this $k$-scheme is given by
	\[
	\Spec(B)(A) = \Hom_{\salg_k}(B,A) \cong \set{ (\mu,a) \in \Azero : \mu^2 = a^{p^r}},
	\]
and the monoid structure is given by $(\mu,a) \cdot (\mu',a') = (\mu\mu',aa')$. The preceding discussion shows that $\bsvrg$ admits a right monoid action by $\Spec(B)$, with the action of an $A$-point $(\mu,a) \in \Spec(B)(A)$ on $\rho \in \bsvrg(A)$ defined by $\rho \cdot (\mu,a) = \rho \circ \phi_{(\mu,a)}$. Then by \cite[Lemma 3.4.1]{Drupieski:2017a}, the algebra $k[\bsvrg]$ admits a nonnegative $\Z[\frac{p^r}{2}]$-grading (more precisely, a grading by $\Z + \Z \cdot \frac{p^r}{2}$), which induces a $\Z[\frac{p^r}{2}]$-grading on $k[\bsvrg_{\ev}] = k[\Vrg]$. Specifically, if $f^*: k[\bsvrg] \rightarrow k[\bsvrg] \otimes B$ is the comorphism corresponding to the monoid action of $\Spec(B)$ on $\Vrg$, and if we write
	\[
	f^*(a) = \sum_{i \geq 0} \left( \phi_i(a) \otimes x^i + \psi_i(a) \otimes x^i y \right)
	\]
for some unique elements $\phi_i(a),\psi_i(a) \in k[\bsvrg]$, then the functions $\phi_i: k[\bsvrg] \rightarrow k[\bsvrg]$ and $\psi_i: k[\bsvrg] \rightarrow k[\bsvrg]$ define the projection maps onto the $\Z[\frac{p^r}{2}]$-graded components of $k[\bsvrg]$ of degrees $i$ and $i+\frac{p^r}{2}$, respectively. Finally, the last statement of the lemma follows from the naturality of the action of $\Spec(B)$ on $\bsvrg$.
\end{proof}

\begin{remark} \label{remark:Zpr2components}
Let $G$ be a finite $k$-supergroup scheme, and suppose $k$ is algebraically closed. Given $\mu,a \in k$ with $a^{p^r} = \mu^2$, let $\phi_{\mu,a}: \Pr \rightarrow \Pr$ be the corresponding Hopf superalgebra homomorphism as in the proof of the lemma, and let $(\phi_{\mu,a})_*: \Vrg \rightarrow \Vrg$ be the natural transformation defined for $\rho \in \Vrg(A) = \Hom_{Hopf/A}(\Pr \otimes_k A, kG \otimes_k A)$ by $(\phi_{\mu,a})_*(\rho) = \rho \circ (\phi_{\mu,a} \otimes 1)$. Then by Yoneda's Lemma, $(\phi_{\mu,a})_*$ is induced by a $k$-algebra homomorphism $\phi_{\mu,a}^*: k[\Vrg] \rightarrow k[\Vrg]$. Now by the definition of the $\Z[\frac{p^r}{2}]$-grading on $k[\Vrg]$ and by the assumption that $k$ is algebraically closed (and hence contains infinitely many pairs $\mu,a \in k$ such that $a^{p^r} = \mu^2$), it follows that the homogeneous components of $k[\Vrg]$ are given for $i \in \N$ by
	\begin{align*}
	k[\Vrg]_i &= \{ z \in k[\Vrg] : \phi_{\mu,a}^*(z) = a^i \cdot z \text{ for all $\phi_{\mu,a}$} \}, \text{ and} \\
	k[\Vrg]_{i+\frac{p^r}{2}} &= \{ z \in k[\Vrg] : \phi_{\mu,a}^*(z) = a^i\mu \cdot z \text{ for all $\phi_{\mu,a}$} \}.
	\end{align*}
\end{remark}

\subsection{Universal homomorphisms} \label{subsec:universalhom}

\begin{definition}
Let $G$ be a finite $k$-supergroup scheme.
	\begin{enumerate}
	\item Define the \emph{universal Hopf superalgebra homomorphism from $\Pr$ to $kG$},
		\[
		u_G': \Pr \otimes_k k[\bsvrg] \rightarrow kG \otimes_k k[\bsvrg],
		\]
	to be the element of $\bsvrg(k[\bsvrg]) = \Hom_{\salg/k}(k[\bsvrg],k[\bsvrg])$ corresponding to the identity map $1_{k[\bsvrg]} : k[\bsvrg] \rightarrow k[\bsvrg]$. Then $u_G'$ is universal in the sense that if $A \in \csalg_k$ and if $\rho: \Pr \otimes_k A \rightarrow kG \otimes_k A$ is a homomorphism of Hopf $A$-super\-algebras, then there exists a unique $k$-superalgebra homomorphism $\phi: k[\bsvrg] \rightarrow A$ such that $\rho = u_G' \otimes_{\phi} A$, i.e., such that $\rho$ is obtained from $u_G'$ via base change along $\phi$.
	
	\item Let $\pi: k[\bsvrg] \twoheadrightarrow k[\Vrg]$ be the canonical quotient map, and set
		\[
		u_G = u_G' \otimes_\pi k[\Vrg] : \Pr \otimes_k k[\Vrg] \rightarrow kG \otimes k[\Vrg].
		\]
	Then $u_G$ is universal in the sense that if $A \in \calg_k$ is a purely even commutative $k$-algebra and if $\rho: \Pr \otimes_k A \rightarrow kG \otimes_k A$ is a homomorphism of Hopf $A$-super\-algebras, then there exists a unique $k$-algebra homomorphism $\phi: k[\Vrg] \rightarrow A$ such that $\rho = u_G \otimes_{\phi} A$, i.e., such that $\rho$ is obtained from $u_G$ via base change along $\phi$. We call $u_G$ the \emph{universal purely even Hopf superalgebra homomorphism from $\Pr$ to $kG$} (here the phrase \emph{purely even} refers solely to the coefficient ring $A$).
	\end{enumerate}
\end{definition}

Observe that $\Pr$ becomes a $\Z[\frac{p^r}{2}]$-graded algebra if we define the degrees of the generators by $\deg(v) = \frac{p^r}{2}$ and $\deg(u_i) = p^i$ for $0 \leq i \leq r-1$. We consider $kG$ as a $\Z[\frac{p^r}{2}]$-graded algebra concentrated in degree $0$. Then $\Pr \otimes_k k[\bsvrg]$ and $kG \otimes_k k[\bsvrg]$ become $\Z[\frac{p^r}{2}]$-graded algebras with $\deg(a \otimes b) = \deg(a) + \deg(b)$, and similarly for $\Pr \otimes_k k[\Vrg]$ and $kG \otimes_k k[\Vrg]$.

\begin{lemma} \label{lemma:uGgradedmap}
Let $G$ be a finite $k$-supergroup scheme. Then $u_G'$ and $u_G$ are homomorphisms of $\Z[\frac{p^r}{2}]$-graded algebras.
\end{lemma}

\begin{proof}
Extending scalars if necessary, we may assume that the field $k$ is algebraically closed.\footnote{If $K/k$ is a field extension, then there are canonical identifications $kG \otimes_k K = KG_K$, $k[\bsvrg] \otimes_k K = K[\bsvr(G_K)]$, and $\Hbul(G,k) \otimes_k K = \Hbul(G_K,K)$; cf.\ \cite[I.1.10, I.4.13]{Jantzen:2003}.} Let $\mu,a \in k$ such that $\mu^2 = a^{p^r}$, and let $\phi = \phi_{\mu,a}$ as in Remark \ref{remark:Zpr2components}. Then by the universal property of $u_G'$, it follows that
	\[
	u_G' \circ (\phi \otimes 1) = u_G' \otimes_{\phi^*} k[\bsvrg] = (1 \otimes \phi^*) \circ u_G',
	\]
i.e., $u_G' \circ (\phi \otimes 1)$ can be obtained from $u_G'$ via base change along $\phi^*: k[\bsvrg] \rightarrow k[\bsvrg]$. Then
	\begin{align*}
	(1 \otimes \phi^*) \circ u_G'(v \otimes 1) &= \mu \cdot u_G'(v \otimes 1), & \text{and} \\
	(1 \otimes \phi^*) \circ u_G'(u_i \otimes 1) &= a^{p^i} \cdot u_G'(u_i \otimes 1) & \text{for $0 \leq i \leq r-1$.}
	\end{align*}
This implies by Remark \ref{remark:Zpr2components} that $u_G'(v \otimes 1)$ and $u_G'(u_i \otimes 1)$ are homogeneous of degrees $\frac{p^r}{2}$ and $p^i$, respectively. Thus, $u_G'$ is a homomorphism of $\Z[\frac{p^r}{2}]$-graded algebras, and the claim for $u_G$ then follows via base change along the quotient map $k[\bsvrg] \twoheadrightarrow k[\Vrg]$.
\end{proof}

As in \cite{Drupieski:2017a}, we write $H(G,k)$ for the subalgebra
	\[
	\bigoplus_{n \geq 0} \opH^n(G,k)_{\ol{n}} = \opH^{\ev}(G,k)_{\zero} \oplus \opH^{\odd}(G,k)_{\one}
	\]
of the full cohomology ring $\Hbul(G,k)$. Then $H(G,k)$ inherits the structure of a ($\Z$-graded, via the cohomological degree) ordinary commutative $k$-algebra. The subspace $\opH^{\ev}(G,k)_{\one} \oplus \opH^{\odd}(G,k)_{\zero}$ that is complementary to $H(G,k)$ in $\Hbul(G,k)$ consists entirely of nilpotent elements, so from the point of view of considering either prime or maximal ideal spectra, it makes no difference whether we work with the full ring $\Hbul(G,k)$ or with the subring $H(G,k)$; cf.\ \cite[\S2.2]{Drupieski:2016a}.

Recall from \cite[Proposition 3.3.6]{Drupieski:2017b} that if $M$ and $N$ are finite-dimensional rational $\Mr$-super\-modules, then the inclusion of categories ${}_{\Mr}\fsmod \hookrightarrow {}_{\Pr}\fsmod$ induces an isomorphism on cohomology groups $\Ext_{\Mr}^\bullet(M,N) \cong \Ext_{\Pr}^\bullet(M,N)$. Taking $M = N = k$, this implies that the cohomology rings $\Hbul(\Mr,k)$ and $\Hbul(\Pr,k)$ are isomorphic. Then by \cite[Proposition 3.2.1]{Drupieski:2017a},
	\begin{equation} \label{eq:Prcohomology}
	\Hbul(\Pr,k) \cong k[x_1,\ldots,x_r,y]/\subgrp{x_r-y^2} \gotimes \Lambda(\lambda_1,\ldots,\lambda_r),
	\end{equation}
where $x_i \in \opH^2(\Pr,k)_{\zero}$, $\lambda_i \in \opH^1(\Pr,k)_{\zero}$, $y \in \opH^1(\Pr,k)_{\one}$, and $\gotimes$ denotes the graded tensor product of graded superalgebras. Let $\ve: \Hbul(\Pr,k) \rightarrow k$ be the $k$-algebra homomorphism that maps $x_r$ and $y$ each to $1$, but that sends the other generators of $\Hbul(\Pr,k)$ each to $0$. Note that $\ve$ is a $k$-algebra homomorphism but not a $k$-superalgebra homo\-morphism because it maps the odd element $y \in \Hbul(\Pr,k)$ to the even element $1 \in k$.

\begin{definition}
Given a finite $k$-supergroup scheme $G$, define
	\[
	\psi_r: H(G,k) \rightarrow k[\Vrg]
	\]
to be the composite $k$-algebra homomorphism
	\begin{multline*} 
	\psi_r: H(G,k) \stackrel{\iota}{\longrightarrow} \Hbul(G,k) \otimes_k k[\Vrg] = \Hbul(G \otimes_k k[\Vrg],k[\Vrg]) \\
	\stackrel{u_G^*}{\longrightarrow} \Hbul(\Pr \otimes_k k[\Vrg],k[\Vrg]) = \Hbul(\Pr,k) \otimes_k k[\Vrg] \stackrel{\ve \otimes 1}{\longrightarrow} k[\Vrg],
	\end{multline*}
where $\iota$ is the base change map $z \mapsto z \otimes 1$, and the equals signs denote the canonical identifications arising via base change (cf.\ \cite[I.4.13]{Jantzen:2003}).
\end{definition}

\begin{remark} \label{remark:psir} \ 
	\begin{enumerate}
	\item When it is useful to emphasize the base field $k$ or the group $G$, we may denote $\psi_r$ by either $\psi_{r,k}$ or $\psi_{r,G}$. In particular, if $K/k$ is a field extension, then it follows via the identifications $H(G,k) \otimes_k K = H(G_K,K)$ and $k[\Vrg] \otimes_k K = K[V_r(G_K)]$ that $\psi_{r,k} \otimes_k K = \psi_{r,K}$.

	\item \label{item:psiragrees} Suppose $G$ is a purely even finite $k$-group scheme. Then it follows from Remark \ref{remark:Vrgpurelyeven} that the universal Hopf superalgebra homomorphism $u_G: \Pr \otimes_k k[\Vrg] \rightarrow kG \otimes k[\Vrg]$ factors through the (map of group algebras uniquely corresponding to the) universal homomorphism $\Gar \otimes_k k[\Vrg] \rightarrow G \otimes_k k[\Vrg]$ defined by Suslin, Friedlander, and Bendel \cite[1.14]{Suslin:1997}. This implies in turn that the homomorphism $\psi_r: H(G,k) \rightarrow k[\Vrg]$ defined above identifies with the homomorphism of the same name that is also defined in \cite[1.14]{Suslin:1997}.
	\end{enumerate}
\end{remark}

\begin{lemma} \label{lemma:psirnatural}
The homomorphism $\psi_r: H(G,k) \rightarrow k[\Vrg]$ is natural with respect to homomorphisms $\phi: G \rightarrow G'$ of finite $k$-supergroup schemes.
\end{lemma}

\begin{proof}
This follows by the same line of reasoning as in the proof of \cite[Lemma 6.2.1]{Drupieski:2017a}.
\end{proof}

\begin{proposition} \label{prop:psirgradedmap}
Let $G$ be a finite $k$-supergroup scheme. Then
	\[
	\psi_r: H(G,k) \rightarrow k[\Vrg]
	\]
is a homomorphism of graded $k$-algebras that multiplies degrees by $\frac{p^r}{2}$.
\end{proposition}

\begin{proof}
Given $z \in \opH^n(G,k)_{\ol{n}} \subset H(G,k)$, we can write
	\begin{equation} \label{eq:uGiotaz}
	(u_G^* \circ \iota)(z) = \sum_{\bs{i},\bs{j},v} x^{\bs{i}}\lambda^{\bs{j}}y^v \otimes f_{\bs{i},\bs{j},v}(z) \in \opH^n(\Pr,k) \otimes_k k[\Vrg]
	\end{equation}
for some unique elements $f_{\bs{i},\bs{j},v}(z) \in k[\Vrg]$. Here $x^{\bs{i}} = x_1^{i_1} \cdots x_{r-1}^{i_{r-1}}$, $\lambda^{\bs{j}} = \lambda_1^{j_1} \cdots \lambda_r^{j_r}$, and the indices run over all nonnegative integers such that $(\sum_{\ell = 1}^{r-1} 2i_\ell + j_\ell) + j_r + v = n$ and $0 \leq j_1,\ldots,j_r \leq 1$. Note that $x_r$ is omitted from these expressions, but the generator $y$ is allowed to appear with an arbitrary nonnegative exponent; since $x_r = y^2$, the given monomials form a basis for $\opH^n(\Pr,k)$. Then $\psi_r(z)$ is the coefficient of $y^n$.

Next, the cohomology rings $\Hbul(G,k)$ and $\Hbul(\Pr,k)$ each inherit an internal $\Z[\frac{p^r}{2}]$-grading from the $\Z[\frac{p^r}{2}]$-algebra gradings on $kG$ and $\Pr$. Then the cohomology ring $\Hbul(G,k)$ is concentrated in internal degree $0$, while the generator $y \in \opH^1(\Pr,k)_{\one}$ is of internal degree $-\frac{p^r}{2}$, and $x_i \in \opH^2(\Pr,k)_{\zero}$ is of internal degree $-p^i$.\footnote{To see this, observe that the grading on $\Pr$ is induced via duality from a corresponding grading on $k[\Mr]$, defined by $\deg(\tau) = -\frac{p^r}{2}$, $\deg(\theta) = 1$, and $\deg(\sigma_i) = ip^{r-1}$ for $i \geq 0$. Now apply the isomorphism $\Hbul(\Pr,k) \cong \Hbul(\Mr,k)$ and the explicit description of the generators for $\Hbul(\Mr,k)$ given in \cite[Proposition 3.2.1]{Drupieski:2017a}.} Since $u_G$ is a map of $\Z[\frac{p^r}{2}]$-graded algebras by Lemma \ref{lemma:uGgradedmap}, the map $u_G^* \circ \iota: H(G,k) \rightarrow \Hbul(\Pr,k) \otimes_k k[\Vrg]$ must preserve both the cohomological and internal degrees of elements. Now $z \in \opH^n(G,k)_{\ol{n}}$ is of internal degree $0$ and $y^n \in \opH^n(\Pr,k)$ is of internal degree $-n \cdot \frac{p^r}{2}$, so it follows that the coefficient of $y^n$ in $(u_G^* \circ \iota)(z)$ is of internal degree $n \cdot \frac{p^r}{2}$. Thus, $\psi_r$ is a graded map that multiplies degrees by $\frac{p^r}{2}$.
\end{proof}

\begin{lemma} \label{lemma:ontoprthpowers}
Let $G$ be a infinitesimal unipotent $k$-super\-group scheme of height $\leq r$. Then the map $\psi_r: H(G,k) \rightarrow k[\Vrg]$ contains in its image the $p^r$-th power of each element of $k[\Vrg]$.
\end{lemma}

\begin{proof}
Let $s = s(G)$ be the integer of Lemma \ref{lemma:kbsvrg=kNrG}, so that $\bsvrg = \bfHom(\Mrs,G)$. Then considering the coordinate algebras of the underlying purely even subschemes, we get in the notation of \cite[Definition 3.3.7]{Drupieski:2017a} the equality $k[\Vrg] = k[\Vrs(G)]$. Next, the proof of \cite[Lemma 3.5.3]{Drupieski:2017b} shows that the homomorphism $\psi_{r;s}: H(G,k) \rightarrow k[\Vrs(G)]$ defined in \cite[\S6.2]{Drupieski:2017a} contains in its image all $p^r$-th powers. So it suffices to check $\psi_r = \psi_{r;s}$. To see this, first observe by the definition of the integer $s = s(G)$ that $u_G: \Pr \otimes_k k[\Vrg] \rightarrow kG \otimes_k k[\Vrg]$ factors through the canonical quotient map $\Pr \otimes_k k[\Vrg] \twoheadrightarrow k\Mrs \otimes_k k[\Vrg]$, and hence induces a Hopf superalgebra homomorphism $\wt{u}_G: k\Mrs \otimes_k k[\Vrg] \rightarrow kG \otimes_k k[\Vrg]$. By duality, this corresponds to a homomorphism of $k[\Vrg]$-super\-group schemes $\wt{u}_G: \Mrs \otimes_k k[\Vrg] \rightarrow G \otimes_k k[\Vrg]$. By the universality of $u_G$, it follows that $\wt{u}_G$ is equal to the universal purely even supergroup homo\-mor\-phism $\Mrs \otimes_k k[\Vrs(G)] \rightarrow G \otimes_k k[\Vrs(G)]$ defined in \cite[Definition 3.3.8(2)]{Drupieski:2017a}. This implies by the definitions of $\psi_r$ and $\psi_{r;s}$ that $\psi_r = \psi_{r;s}$.
\end{proof}

\subsection{The support scheme of a module} \label{subsec:definesupportscheme}

Let $G$ be a finite $k$-supergroup scheme. At this point we consider the purely even scheme $\Vrg$ as $\Spec(k[\Vrg])$, the prime ideal spectrum of the commutative $k$-algebra $k[\Vrg]$, in addition to thinking of $\Vrg$ as its functor of points,
	\[
	\Vrg: \calg_k \rightarrow \sets, \quad A \mapsto \Vrg(A) \cong \Hom_{\alg/k}(k[\Vrg],A).
	\]
Given a point $\fs \in \Vrg$ (i.e., a prime ideal $\fs \subset k[\Vrg]$), let $k(\fs)$ be the residue field of $\Vrg$ at $\fs$ (i.e., the residue field of the local ring $k[\Vrg]_{\fs}$, or equivalently the field of fractions of the integral domain $k[\Vrg]/\fs$). Then the canonical $k$-algebra homomorphism $\phi_{\fs}: k[\Vrg] \rightarrow k(\fs)$ defines the $k(\fs)$-point $\nu_{\fs} := u_G \otimes_{\phi_{\fs}} k(\fs) \in \Vrg(k(\fs))$. Conversely, if $K/k$ is a field extension and if $\phi \in \Hom_{\alg/k}(k[\Vrg],K)$, then $\fs = \ker(\phi)$ is a prime ideal in $k[\Vrg]$, and the induced $k$-algebra homomorphism $k[\Vrg]/\fs \hookrightarrow K$ extends to a field embedding $k(\fs) \hookrightarrow K$. Then the Hopf superalgebra homomorphism $\nu_{\phi} := u_G \otimes_\phi K : \Pr \otimes_k K \rightarrow kG \otimes_k K$ corresponding to $\phi$ can be obtained from $\nu_{\fs}$ via base change along the field embedding $k(\fs) \hookrightarrow K$.

Write $\Pone = k[u,v]/\subgrp{u^p+v^2}$, and let $\iota: \Pone \rightarrow \Pr$ be the superalgebra map defined by $\iota(u) = u_{r-1}$ and $\iota(v) = v$. Note that $\iota$ is a map of Hopf superalgebras only if $r=1$. Now given a $kG$-super\-module $M$ and a point $\fs \in \Vrg$, we consider $M \otimes_k k(\fs)$ as a $\Pone \otimes_k k(\fs)$-super\-module by pulling back along the composite $k(\fs)$-superalgebra homomorphism
	\[
	\nu_{\fs} \circ (\iota \otimes 1): \Pone \otimes_k k(\fs) \hookrightarrow \Pr \otimes_k k(\fs) \rightarrow kG \otimes_k k(\fs).
	\]

By Lemma \ref{lemma:Zpr2grading}, the ring $k[\Vrg]$ is a $\Z[\frac{p^r}{2}]$-graded $k$-algebra. We say that a Zariski closed subset $X \subset \Vrg$ is \emph{conical} if $X$ is defined by a $\Z[\frac{p^r}{2}]$-homogeneous ideal.

\begin{proposition} \label{prop:VrgMclosed}
Let $G$ be a finite $k$-supergroup scheme, and let $M$ be a $kG$-supermodule. Then
	\[
	\Vrg_M := \set{ \fs \in \Vrg : \pd_{\Pone \otimes_k k(\fs)}(M \otimes_k k(\fs)) = \infty }
	\]
is a Zariski closed conical subset of $\Vrg$. If $f: H \rightarrow G$ is a homomorphism of finite $k$-supergroup schemes, and if $M$ is also considered as a $kH$-supermodule via pullback along $f$, then the morphism of schemes $f_*: V_r(H) \rightarrow \Vrg$ induced by $f$ satisfies
	\begin{equation} \label{eq:invimageVrgM}
	f_*^{-1}(\Vrg_M) = V_r(H)_M.
	\end{equation}
In particular, if $f$ and hence also $f_*$ is a closed embedding, and if we identify $V_r(H)$ with its image in $\Vrg$, then $V_r(H)_M = V_r(H) \cap \Vrg_M$.
\end{proposition}

\begin{proof}
Let $A = k[\Vrg]$. Set $M_A = M \otimes_k A$, and given a point $\fs \in \Vrg$, set $M_{k(\fs)} = M \otimes_k k(\fs)$. We consider $M_A$ as a $\Pone \otimes_k A$-supermodule by pulling back along the composite $A$-superalgebra homomorphism
	\begin{equation} \label{eq:PoneApullback}
	u_G \circ (\iota \otimes_k 1): \Pone \otimes_k A \hookrightarrow \Pr \otimes_k A \rightarrow kG \otimes_k A.
	\end{equation}
Identify the cohomology ring $\Hbul(\Pone,k)$ as in \eqref{eq:Prcohomology} with $r=1$. For each point $\fs \in \Vrg$, the canonical $k$-algebra homomorphism $\phi_{\fs} : A \rightarrow k(\fs)$ induces a ring homomorphism
	\[
	(\phi_{\fs})_*: \Ext_{\Pone \otimes_k A}^\bullet(M_A,M_A) \rightarrow \Ext_{\Pone \otimes_k k(\fs)}^\bullet(M_{k(\fs)},M_{k(\fs)}).
	\]
This homomorphism sends the identity map $1_{M_A} : M_A \rightarrow M_A$ to the identity map $1_{M_{k(\fs)}}: M_{k(\fs)} \rightarrow M_{k(\fs)}$, and commutes with the right cup product action of $\Hbul(\Pone,k)$. (The ring $\Hbul(\Pone,k)$ acts on $\Ext_{\Pone \otimes_k A}^\bullet(M_A,M_A)$ through the base change map $\Hbul(\Pone,k) \rightarrow \Hbul(\Pone,k) \otimes_k A$, $z \mapsto z \otimes 1$, and similarly for its action on $\Ext_{\Pone \otimes_k k(\fs)}^\bullet(M_{k(\fs)},M_{k(\fs)})$.)

Now let $W = W_M = A \cdot (1_{M_A} \cup y^2)$ be the $A$-submodule of $\Ext_{\Pone \otimes_k A}^2(M_A,M_A)$ generated by the cup product $1_{M_A} \cup y^2$, and let $J_M = \ann_A(W)$ be the annihilator in $A$ of $W$. For each $\fs \in \Vrg$, $(\phi_{\fs})_*$ maps $W$ onto the $k(\fs)$-subspace of $\Ext_{\Pone \otimes_k k(\fs)}^2(M_{k(\fs)},M_{k(\fs)})$ spanned by $1_{M_{k(\fs)}} \cup y^2$. Then by Proposition \ref{prop:Poneprojdim}, $\pd_{\Pone \otimes_k k(\fs)}(M_{k(\fs)}) = \infty$ if and only if $(\phi_{\fs})_*(W) \neq 0$. Since $W$ is a finitely-generated $A$-module, Nakayama's Lemma implies that $(\phi_{\fs})_*(W) = W \otimes_A k(\fs)$ is nonzero if and only if the localization $W_{\fs}$ is nonzero. But $W_{\fs} \neq 0$ if and only if $J_M \subseteq \fs$. Thus, $\Vrg_M$ is the Zariski closed subset of $\Vrg$ defined by the ideal $J_M$.

To see that $J_M$ is a $\Z[\frac{p^r}{2}]$-homogeneous ideal, we consider $M$ as a graded $kG$-supermodule concentrated in $\Z[\frac{p^r}{2}]$-degree $0$, and consider $\Pone$ as a $\Z[\frac{p^r}{2}]$-graded algebra with $\deg(u) = p^{r-1}$ and $\deg(v) = \frac{p^r}{2}$, so that \eqref{eq:PoneApullback} is a homomorphism of $\Z[\frac{p^r}{2}]$-graded algebras. Then $\Ext_{\Pone \otimes_k A}^\bullet(M_A,M_A)$ inherits an internal $\Z[\frac{p^r}{2}]$-grading, making it into a graded $A$-module. The cup product $1_{M_A} \cup y^2$ is homogeneous (of degree $-p^r$) with respect to the internal grading, so its  annihilator in $A$ is a homogeneous ideal.

Now let $f: H \rightarrow G$ be a homomorphism of finite $k$-supergroup schemes. Given a field extension $K/k$, the set of $K$-points of the scheme $\Vrg_M$ is given by
	\[
	\Vrg_M(K) = \set{ \nu \in \Vrg(K): \pd_{K\Pone}(\nu^*M_K) = \infty},
	\]
where $K\Pone = \Pone \otimes_k K$, $M_K = M \otimes_k K$, and $\nu^* M_K$ denotes the pullback of the $KG = kG \otimes_k K$-super\-module $M_K$ along the Hopf $K$-superalgebra homomorphism $\nu: K\Pr \rightarrow KG$. One immediately sees that the map on $K$-points $f_*: V_r(H)(K) \rightarrow \Vrg(K)$ satisfies $f_*^{-1}(\Vrg_M(K)) = V_r(H)_M(K)$. Then varying over all $K$, \eqref{eq:invimageVrgM} follows (cf.\ the discussion preceding Proposition \ref{prop:psirinvelementary}).
\end{proof}

\begin{remark} \label{remark:VrgMK}
Retain the notations of the preceding proof, and let $K/k$ be a field extension. Given a $k$-vector space $V$ and an element $v \in V$, set $V_K = V \otimes_k K$ and set $v_K = v \otimes 1 \in V_K$. Then $A_K = K[V_r(G_K)]$, and for $M$ finite-dimensional it follows by Remark \ref{remark:flatbasechange} that
	\[
	\Ext_{A\Pone}^\bullet(M_A,M_A) \otimes_k K = \Ext_{A\Pone}^\bullet(M_A,M_A) \otimes_A A_K \cong \Ext_{A_K\Pone}^\bullet(M_{A_K},M_{A_K})
	\]
Under this identification, $(1_{M_A} \cup y^2)_K$ identifies with the cup product $1_{M_{A_K}} \cup y^2$. Then $V_r(G_K)_{M_K}$ is the Zariski closed subset of $V_r(G_K)$ defined by the ideal
	\[
	J_{M_K} = \ann_{A_K}\left( A_K . (1_{M_A} \cup y^2 )_K \right) \subset A_K.
	\]
From this it follows that $(J_M)_K \subseteq J_{M_K}$, i.e., the image of the ideal $J_M \subset k[\Vrg]$ under the base change map $k[\Vrg] \rightarrow k[\Vrg] \otimes_k K = K[V_r(G_K)]$ is contained in the defining ideal $J_{M_K}$ of the scheme $V_r(G_K)_{M_K}$. In particular, this shows that one of the technical assumptions made to justify \cite[Conjecture 3.5.6]{Drupieski:2017b} is indeed satisfied.
\end{remark}

By Remark \ref{remark:Vrgpurelyeven}, our usage of the notation $\Vrg$ agrees with that of Suslin, Friedlander, and Bendel \cite{Suslin:1997} when $G$ is a purely even finite $k$-group scheme. The next lemma implies that our usage of the notation $\Vrg_M$ is similarly consistent with its usage in \cite{Suslin:1997a}. To simplify notation we work over the base field $k$ rather than over the residue field $k(\fs)$ at a point $\fs \in \Vrg$.

\begin{lemma} \label{lemma:finiteifffree}
Let $M$ be a finite-dimensional $k\Gaone$-supermodule, considered also as a $\Pone$-super\-module via the canonical quotient map $\Pone \twoheadrightarrow \Pone/\subgrp{v} = k[u]/\subgrp{u^p} = k\Gaone$. Then $\pd_{\Pone}(M) < \infty$ if and only if $M$ is projective (equivalently, free) as a $k\Gaone$-module.

Similarly, if $M$ is a finite-dimensional $k\Gam$-super\-module, considered also as a $\Pone$-supermodule via the canonical quotient map $\Pone \twoheadrightarrow \Pone/\subgrp{u} = k[v]/\subgrp{v^2} = k\Gam$, then $\pd_{\Pone}(M) < \infty$ if and only if $M$ is projective (equivalently, free) as a $k\Gam$-supermodule.
\end{lemma}

\begin{proof}
As in the paragraph preceding Lemma \ref{lemma:Extrestrictioniso}, we consider $M$ as a rational $\Mone$-supermodule. Then by Proposition \ref{prop:injprojdimequalities} and the fact that the Hopf algebra $k\Gaone$ is self-injective, the first assertion of the lemma is equivalent to showing that $\id_{\Mone}(M) < \infty$ if and only if $\id_{\Gaone}(M) < \infty$. Recall from \cite[\S3.3.2]{Drupieski:2017b} that there exists a normal subsupergroup $Q \leq \Mone$ such that $\Mone/Q \cong \Gaone$, and such that the corresponding Lyndon--Hochschild--Serre spectral sequence with coefficients in $M$ takes the form
	\[
	E_2^{i,j}(M) = \opH^i(\Gaone,M) \otimes \opH^j(Q,k) \Rightarrow \opH^{i+j}(\Mone,M).
	\]
Furthermore, all differentials in this spectral sequence from the $E_2$-page onward are zero, and the cohomology ring $\Hbul(Q,k)$ is an exterior algebra generated by a single class $[\wt{\tau}] \in \opH^1(Q,k)_{\one}$. Then $\opH^i(\Mone,M) = 0$ for all $i \gg 0$ if and only if $\opH^i(\Gaone,M) = 0$ for all $i \gg 0$. Since $\Mone$ and $\Gaone$ are both unipotent, this proves the first assertion of the lemma. The proof of the second assertion of the lemma is entirely similar, using instead the LHS spectral sequence for the extension $Q^- \rightarrow \Mone \rightarrow \Gam$, as described in \cite[\S3.3.3]{Drupieski:2017b}. 
\end{proof}

\begin{remark} \label{remark:N1GM}
Suppose the field $k$ is algebraically closed, and let $G$ be an algebraic $k$-supergroup scheme. In our previous paper \cite{Drupieski:2017b}, we showed that the set $\NrG := \Hom_{Grp/k}(\Mr,G)$ of $k$-super\-group scheme homomorphisms $\phi: \Mr \rightarrow G$ admits the structure of an affine algebraic variety. We further defined, for $M$ a rational $G$-supermodule, the subset
	\[
	\NoneG_M = \set{ \phi \in \NoneG : \id_{\Mone}(\phi^* M) = \infty}.
	\]
Here $\phi^*M$ denotes the pullback of $M$ along $\phi$. Now for $G$ finite, $\NrG$ identifies by Remark \ref{remark:Vrgpurelyeven} with the set of closed points of the scheme $\Vrg$, and for $M$ finite-dimensional, $\NoneG_M$ identifies by Proposition \ref{prop:injprojdimequalities} with the set of closed points of the scheme $V_1(G)_M$. Then Proposition \ref{prop:VrgMclosed} implies that $\NoneG_M$ is Zariski closed in $\NoneG$, answering Question 3.3.5 of \cite{Drupieski:2017b}.
\end{remark}

\section{Cohomological support schemes} \label{section:supportschemes}

\subsection{The cohomological spectrum}

Our goal in Section \ref{section:supportschemes} is to relate, for $G$ an infinitesimal unipotent $k$-supergroup scheme of height $\leq r$, the cohomological support scheme of a finite-dimen\-sional rational $G$-supermodule $M$ to the scheme $\Vrg_M$ defined in Proposition \ref{prop:VrgMclosed}. First we recall some basic definitions.

\begin{definition}[Cohomological spectrum and cohomological variety of $G$] \label{def:cohomologicalspectrum}
Given a finite $k$-super\-group scheme $G$, the \emph{cohomological spectrum} of $G$ is the affine scheme
	\[
	\abs{G} = \Spec\left( H(G,k) \right).
	\]
Equivalently, $\abs{G}$ is the prime ideal spectrum of the full cohomology ring $\Hbul(G,k)$; cf.\ \cite[\S2.2]{Drupieski:2016a}. In consequence of the main theorem of \cite{Drupieski:2016}, $H(G,k)$ is a finitely-generated commutative $k$-algebra. Thus if $k$ is algebraically closed, the maximal ideal spectrum of $H(G,k)$, $\Max(H(G,k))$, is an affine algebraic variety, which we call the \emph{cohomological variety} of $G$. By abuse of notation, we also denote this affine variety by $\abs{G}$, indicating through the context whether we mean the affine scheme or its affine variety of closed points.
\end{definition}

Now let $M$ be a finite-dimensional rational $G$-supermodule, and set $\Lambda = \End_k(M)$. Then $\Lambda$ is a unital rational $G$-algebra, and the unit map $1_\Lambda: k \rightarrow \Lambda$ is a $G$-supermodule homomorphism. Let $\rho_\Lambda: \Hbul(G,k) \rightarrow \Hbul(G,\Lambda)$ be the map in cohomology induced by $1_\Lambda$, and set $I_M = \ker(\rho_\Lambda)$. By abuse of notation, we also denote $I_M \cap H(G,k)$ by $I_M$. Under the adjoint associativity isomorphism $\Hbul(G,\Hom_k(M,M)) \cong \Ext_G^\bullet(M,M)$, $\rho_\Lambda$ identifies with the ring homomorphism $\Phi_M: \Hbul(G,k) \rightarrow \Ext_G^\bullet(M,M)$ defined by $\Phi_M(z) = 1_M \cup z$. (Recall that multiplication in $\Ext_G^\bullet(M,M)$ is defined via the Yoneda composition of extensions.)

\begin{definition}[Support scheme and support variety of a module] \label{def:supportscheme}
The \emph{cohomological support scheme} of $M$ is the Zariski closed subset $\abs{G}_M$ of the affine scheme $\abs{G}$ defined by the ideal $I_M$:
	\[
	\abs{G}_M = \Spec\left( H(G,k)/I_M \right).
	\]
If the field $k$ is algebraically closed, then we define the \emph{cohomological support variety} of $M$ to be $\Max( H(G,k)/I_M)$, the closed subvariety of $\Max(H(G,k))$ defined by $I_M$. By abuse of notation, we also denote the cohomological support variety of $M$ by $\abs{G}_M$, indicating through the context whether we mean the affine scheme or its affine variety of closed points.
\end{definition}

The $k$-algebra homomorphism $\psi_r: H(G,k) \rightarrow k[\Vrg]$ defines a morphism of $k$-schemes
	\[
	\Psi_r: \Vrg \rightarrow \abs{G}.
	\]
If $G$ is a unipotent infinitesimal $k$-supergroup scheme of height $\leq r$, then we can show that $\Psi_r$ is a universal homeomorphism.

\begin{theorem} \label{thm:psiuniversalhomeo}
Let $G$ be an infinitesimal unipotent $k$-super\-group scheme of height $\leq r$. Then the kernel of the homomorphism
	\[
	\psi_r: H(G,k) \rightarrow k[\Vrg]
	\]
is a locally nilpotent ideal, and the image of $\psi_r$ contains the $p^r$-th power of each element of $k[\Vrg]$. Consequently, the associated morphism of schemes $\Psi_r: \Vrg \rightarrow \abs{G}$ is a universal homeomorphism.
\end{theorem}

\begin{proof}
By Lemma \ref{lemma:ontoprthpowers}, the homomorphism $\psi_r$ contains in its image the $p^r$-th power of each element of $k[\Vrg]$. The argument showing that $\psi_r$ is injective modulo nilpotents is essentially already given in \cite{Drupieski:2017b}, but for the sake of completeness we repeat the argument here.

Since $\psi_r$ is a map of graded $k$-algebras (that multiplies degrees by $\frac{p^r}{2}$) by Proposition \ref{prop:psirgradedmap}, it suffices to consider a homogeneous element $z \in H(G,k)$ such that $\psi_r(z) = 0$. Let $K$ be an algebraically closed extension field of $k$, and let $E$ be an elementary sub\-super\-group scheme of $G_K$, i.e., $E$ is a (infinitesimal, of height $\leq r$) unipotent multiparameter $K$-supergroup scheme that occurs as a (closed) sub\-super\-group scheme of $G_K$. Denote the closed embedding of $E$ into $G_K$ by $\nu: E \hookrightarrow G_K$. Then by Lemma \ref{lemma:psirnatural}, there exists a commutative diagram
\[
\xymatrix{
H(G_K,K) \ar@{->}[r]^{\psi_{r,K}} \ar@{->}[d]^{\nu^*} & K[V_r(G_K)] \ar@{->>}[d]^{\nu^*} \ar@{->>}[r] & K[V_r(G_K)]_{red} \ar@{->>}[d]^{\nu^*} \\
H(E,K) \ar@{->}[r]^{\psi_{r,K}} & K[V_r(E)] \ar@{->>}[r] & K[V_r(E)]_{red}
}
\]
in which the unlabeled horizontal arrows are the canonical quotient maps. The second and third vertical arrows are surjections by Lemma \ref{lemma:bsvrg}. By Lemma \ref{lemma:kbsvrg=kNrG}, the reduced rings $K[V_r(G_K)]_{red}$ and $K[V_r(E)]_{red}$ identify with the algebras $K[\Nr(G_K)]$ and $K[\Nr(E)]$, respectively. Under this identification, the composite homomorphism formed by the bottom row of the diagram identifies with the map $\psi_r: H(E,K) \rightarrow K[\Nr(E)]$ defined in \cite[\S2.4]{Drupieski:2017b}. Then by \cite[Corollary 3.5.5]{Drupieski:2017b}, composition along the bottom row of the diagram is injective modulo nilpotents. Denote the image of $z$ under the base change map $H(G,k) \rightarrow H(G,k) \otimes_k K = H(G_K,K)$ by $z_K$, and similarly write $\psi_r(z)_K$ for the image of $\psi_r(z)$ under the base change map $k[\Vrg] \rightarrow k[\Vrg] \otimes_k K = K[V_r(G_K)]$. Then $\psi_{r,K}(z_K) = \psi_r(z)_K = 0$, so the commutativity of the diagram implies that $\nu^*(z_K)$ is nilpotent in $H(E,K)$. Since $K$ and $E$ were arbitrary, this implies by \cite[Theorem 1.2]{Benson:2018} that $z$ is nilpotent.\footnote{While the statement of \cite[Theorem 1.2]{Benson:2018} does not require the extension field $K$ to be algebraically closed, nothing is harmed by further extending any given field to its algebraic closure.} Now $\Psi_r$ is a universal homeomorphism by \cite[\href{https://stacks.math.columbia.edu/tag/0CNF}{Tag 0CNF}]{stacks-project}.
\end{proof}

\begin{corollary} \label{cor:subgroupspectrum}
Let $\iota: H \hookrightarrow G$ be a closed embedding of infinitesimal unipotent $k$-supergroup schemes of height $\leq r$. Then the restriction map on reduced rings,
	\[
	(\iota^*)_{red}: H(G,k)_{red} \rightarrow H(H,k)_{red},
	\]
contains in its image the $p^r$-th power of each element of $H(H,k)_{red}$. Consequently, the induced morphism of schemes $\iota_*: \abs{H} \rightarrow \abs{G}$ is finite and universally injective.
\end{corollary}

\begin{proof}
By Lemma \ref{lemma:psirnatural}, the embedding $\iota: H \hookrightarrow G$ gives rise to the commutative diagram
\[
\xymatrix@C+1em{
H(G,k)_{red} \ar@{->}[r]^{(\psi_{r,G})_{red}} \ar@{->}[d]^{(\iota^*)_{red}} & k[V_r(G)]_{red} \ar@{->>}[d]^{(\iota^*)_{red}} \\
H(H,k)_{red} \ar@{->}[r]^{(\psi_{r,H})_{red}} & k[V_r(H)]_{red},
}
\]
in which the right-hand vertical arrow is a surjection by Lemma \ref{lemma:bsvrg}. For improved legibility we henceforth omit the subscript `$red$' from the morphisms in the diagram. So let $z \in H(H,k)_{red}$. Then $\psi_{r,H}(z) = \iota^*(x)$ for some $x \in k[\Vrg]_{red}$. By Theorem \ref{thm:psiuniversalhomeo}, the homomorphism $\psi_{r,G}$ contains in its image the $p^r$-th power of each element of $k[\Vrg]_{red}$, so there exists $z' \in H(G,k)_{red}$ such that $\psi_{r,G}(z') = x^{p^r}$. Now by the commutativity of the diagram,
	\begin{align*}
	\psi_{r,H}(\iota^*(z') - z^{p^r}) &= (\psi_{r,H} \circ \iota^*)(z') - \psi_{r,H}(z^{p^r}) \\
	&= (\iota^* \circ \psi_{r,G})(z') - \psi_{r,H}(z)^{p^r} \\
	&= \iota^*(x^{p^r}) - \iota^*(x)^{p^r} = 0.
	\end{align*}
But Theorem \ref{thm:psiuniversalhomeo} implies that the map of reduced rings $\psi_{r,H}: H(H,k)_{red} \rightarrow k[V_r(H)]_{red}$ is an injection, and hence $\iota^*(z') = z^{p^r}$.
\end{proof}

Our goal in the rest of Section \ref{section:supportschemes} is to show, for $G$ infinitesimal unipotent of height $\leq r$ and $M$ a finite-dimensional rational $G$-supermodule, that the morphism of schemes $\Psi_r: \Vrg \rightarrow \abs{G}$ satisfies $\Psi_r^{-1}( \abs{G}_M ) = \Vrg_M$. Equivalently, let $\calJ_M = \sqrt{J_M} \subset k[\Vrg]$ be the ($\Z[\frac{p^r}{2}]$-homogeneous) ideal of functions vanishing on $\Vrg_M$, and set $\calI_M = \sqrt{I_M} \subset H(G,k)$. Then we want to show that the algebra homomorphism $\psi_r: H(G,k) \rightarrow k[\Vrg]$ satisfies $\psi_r^{-1}(\calJ_M) = \calI_M$. Our strategy mimics the approach of Suslin, Friedlander, and Bendel \cite{Suslin:1997a}, arguing first for height-one infinitesimal unipotent supergroup schemes, and then deducing the general case.

\subsection{Height-one infinitesimal unipotent supergroups}

\begin{lemma} \label{lemma:psiIMinJM}
Let $G$ be a height-one infinitesimal $k$-supergroup scheme, and let $M$ be a rational $G$-super\-module. Then the homomorphism $\psi = \psi_1: H(G,k) \rightarrow k[V_1(G)]$ satisfies $\psi(\calI_M) \subseteq \calJ_M$, and hence the associated morphism of schemes $\Psi: V_1(G) \rightarrow \abs{G}$ satisfies $\Psi(V_1(G)_M) \subseteq \abs{G}_M$.
\end{lemma}

\begin{proof}
Set $\psi = \psi_1$, and let $z \in \calI_M$. Since $\psi$ is a map of graded rings, and since $\calI_M$ and $\calJ_M$ are homogeneous ideals (because they are the radicals of homogeneous ideals), we may assume that $z$ is homogeneous, say, $z \in \opH^n(G,k)_{\ol{n}}$. Replacing $z$ with some power of $z$ if necessary, we may further assume that $z \in I_M$ and $n \geq 2$. Next recall that $\Hbul(\Pone,k) \cong k[y] \gotimes \Lambda(\lambda)$, where $y \in \opH^1(\Pone,k)_{\one}$ and $\lambda \in \opH^1(\Pone,k)_{\zero}$. The homomorphism $u_G^* \circ \iota$ in the definition of $\psi$ preserves both the cohomological and super degrees of elements, so it follows by the definition of $\psi$ that $(u_G^* \circ \iota)(z) = y^n \otimes \psi(z) \in \Hbul(\Pone,k) \otimes k[V_1(G)]$. Then given a point $\fs \in V_1(G)$ and the corresponding Hopf superalgebra homomorphism $\nu_\fs: \Pone \otimes_k k(\fs) \rightarrow kG \otimes_k k(\fs)$, the induced map in cohomology $\nu_{\fs}^*: H(G_{k(\fs)},k(\fs)) \rightarrow \Hbul(\Pone \otimes_k k(\fs),k(\fs))$ satisfies $\nu_{\fs}^*(z_{k(\fs)}) = y^n \cdot \psi(z)(\fs)$. Here $\psi(z)(\fs)$ denotes the image of $\psi(z)$ under the $k$-algebra homomorphism $\phi_{\fs}: k[V_1(G)] \rightarrow k(\fs)$. More generally, $\nu_{\fs}$ induces a $k(\fs)$-superalgebra homomorphism
	\[
	\nu_{\fs}^*: \Ext_{G_{k(\fs)}}^\bullet(M_{k(\fs)},M_{k(\fs)}) \rightarrow \Ext_{\Pone \otimes_k k(\fs)}^\bullet(M_{k(\fs)},M_{k(\fs)})
	\]
that is compatible in the evident fashion with the right cup product actions of $\Hbul(G_{k(\fs)},k(\fs))$ and $\Hbul(\Pone \otimes_k k(\fs),k(\fs))$. In particular,
	\[
	\nu_{\fs}^*((1_M \cup z)_{k(\fs)}) = \nu_{\fs}^*(1_{M_{k(\fs)}} \cup z_{k(\fs)}) = 1_{M_{k(\fs)}} \cup \nu_{\fs}^*(z_{k(\fs)}) = (1_{M_{k(\fs)}} \cup y^n) \cdot \psi(z)(\fs).
	\]
By Proposition \ref{prop:Poneprojdim} and the assumption that $n \geq 2$, the cup product $1_{M_{k(\fs)}} \cup y^n$ is nonzero if and only if $\fs \in V_1(G)_M$. But $1_M \cup z = 0$ by the assumption that $z \in I_M$, so we deduce that $\psi(z)(\fs) = 0$ for all $\fs \in V_1(G)_M$, and hence $\psi(z) \in \calJ_M$.
\end{proof}

\begin{lemma} \label{lemma:psiinverseheightone}
Let $G$ be a height-one infinitesimal unipotent $k$-supergroup scheme, and let $M$ be a finite-dimensional rational $G$-supermodule. Then the homomorphism $\psi = \psi_1: H(G,k) \rightarrow k[V_1(G)]$ satisfies $\psi^{-1}(\calJ_M) = \calI_M$, and hence the associated morphism of schemes $\Psi: V_1(G) \rightarrow \abs{G}$ satisfies
	\[
	\Psi^{-1}(\abs{G}_M) = V_1(G)_M.
	\]
\end{lemma}

\begin{proof}
We have already shown in Lemma \ref{lemma:psiIMinJM} that $\psi(\calI_M) \subseteq \calJ_M$. For the reverse inclusion, let $z \in H(G,k)$, and suppose $\psi(z) \in \calJ_M$. We want to show that $z \in \calI_M$, or equivalently, in the notation preceding Definition \ref{def:supportscheme}, that $\rho_{\Lambda}(z)$ is nilpotent in the algebra $\Hbul(G,\Lambda) = \Hbul(G,\End_k(M))$. Replacing $z$ with some power of $z$ if necessary, we may assume that $\psi(z) \in J_M$.

Let $K$ be an algebraically closed field extension of $k$, and let $E$ be a nontrivial elementary sub\-super\-group scheme of $G_K$. Then $E$ is a height-one infinitesimal elementary $K$-supergroup scheme, and hence is isomorphic by Corollary \ref{cor:infinitesimalelementary} to one of $\Gaone$, $\Ga^-$, or $\Mones$ for some $s \geq 1$. Denote the closed embedding of $E$ into $G_K$ by $\nu: E \hookrightarrow G_K$, and consider $M_K$ as a rational $E$-supermodule by pulling back along $\nu$. Let $\calJ_{M_K}^E \subset K[V_1(E)]$ be the radical ideal that defines $V_1(E)_{M_K}$ as a Zariski closed subset of $V_1(E)$. Then by Lemma \ref{lemma:psirnatural} and Proposition \ref{prop:VrgMclosed}, the homomorphism $\nu$ gives rise to the commutative diagram
\[
\xymatrix{
H(G_K,K) \ar@{->}[r]^{\psi} \ar@{->}[d]^{\nu^*} & K[V_1(G_K)] \ar@{->>}[d]^{\nu^*} \ar@{->>}[r] & K[V_1(G_K)]/\calJ_{M_K} \ar@{->>}[d]^{\nu^*} \\
H(E,K) \ar@{->}[r]^{\psi_E} & K[V_1(E)] \ar@{->>}[r] & K[V_1(E)]/\calJ_{M_K}^E,
}
\]
in which the unlabeled arrows are the canonical quotient maps. By Remark \ref{remark:VrgMK}, $\psi(z_K) = \psi(z)_K \in (J_M)_K \subseteq J_{M_K} \subseteq \calJ_{M_K}$. Then by the commutativity of the diagram, $\psi_E(\nu^*(z_K)) \in \calJ_{M_K}^E$.

Next, let $\calI_{M_K}^E \subset H(E,K)$ be the radical ideal that defines $\abs{E}_{M_K}$ as a Zariski closed subset of $\abs{E}$. Since $K$ is algebraically closed, we can consider $\abs{E}$ and $\abs{E}_{M_K}$ as affine algebraic varieties. Then applying Remark \ref{remark:N1GM}, we can interpret \cite[Theorem 3.4.1]{Drupieski:2017b} as saying that the morphism of affine varieties $\Psi: V_1(E) \rightarrow \abs{E}$ induced by $\psi_E: H(E,K) \rightarrow K[V_1(E)]$ satisfies $\Psi^{-1}(\abs{E}_{M_K}) = V_1(E)_{M_K}$. Theorem \ref{thm:psiuniversalhomeo} implies that $\Psi: V_1(E) \rightarrow \abs{E}$ is a homeomorphism of affine varieties, so this means that the map $\fm \mapsto \psi_E^{-1}(\fm)$ defines a bijection of maximal ideal spectra
	\[
	\Max(K[V_1(E)]/\calJ_{M_K}^E) \rightarrow \Max(H(E,K)/\calI_{M_K}^E).
	\]
Now because $\psi_E(\nu^*(z_K)) \in \calJ_{M_K}^E$, it follows that $\nu^*(z_K) \in \fn$ for every maximal ideal $\fn \subset H(E,k)$ such that $\fn \supseteq \calI_{M_K}^E$. But $\calI_{M_K}^E$ is a radical ideal and $H(E,K)$ is a finitely-generated commutative (in the ordinary non-super sense) ring over the field $K$, so this implies that $\nu^*(z_K) \in \calI_{M_K}^E$. Equivalently, $\rho_{\Lambda_K}(\nu^*(z_K)) = \nu^*(\rho_{\Lambda_K}(z_K)) = \nu^*(\rho_{\Lambda}(z)_K)$ is nilpotent in the algebra $\Hbul(E,\Lambda_K)$. Finally, since $K$ and $E$ were arbitrary, this implies by \cite[Theorem 8.4]{Benson:2018} that $\rho_{\Lambda}(z)$ is nilpotent in the algebra $\Hbul(G,\Lambda)$, which is what we wanted to show.
\end{proof}

\subsection{Support calculations for \texorpdfstring{$\Gar$, $\Mrs$, and $\Mrseta$}{Gar, Mrs, and Mrseta}}

Throughout this subsection assume that $r \geq 2$. Then by Corollary \ref{cor:infinitesimalelementary}, each height-$r$ infinitesimal elementary $k$-supergroup scheme $G$ is isomorphic to either $\Gar$ or $\Mrseta$ for some $s \geq 1$ and some $\eta \in k$. Our goal in this subsection is to show for each such $G$ and each finite-dimensional rational $G$-supermodule $M$ that the morphism of schemes $\Psi = \Psi_r: \Vrg \rightarrow \abs{G}$ satisfies $\Psi^{-1}(\abs{G}_M) = \Vrg_M$. Our argument parallels the reasoning given by Suslin, Friedlander, and Bendel \cite[\S6]{Suslin:1997a}.

\begin{lemma} \label{lemma:unipotentannihilator}
Let $G$ be a finite unipotent $k$-supergroup scheme, and let $M$ be a finite-dimensional rational $G$-supermodule. Let $L_M \subset H(G,k)$ be the annihilator in $H(G,k)$ for the left (equivalently, right) cup product action of $H(G,k)$ on $\Hbul(G,M)$. Then $\sqrt{L_M} = \calI_M$.
\end{lemma}

\begin{proof}
Rational cohomology for $G$ identifies with cohomology for the Hopf superalgebra $kG$, so the comment about the left and right annihilators in $H(G,k)$ of $\Hbul(G,M)$ coinciding follows from \cite[Proposition 2.3.5]{Drupieski:2016a}. Next, the left cup product action of $H(G,k)$ on $\Hbul(G,M)$ factors through the cup product action of $H(G,k)$ on $\Ext_G^\bullet(M,M)$ (see again \cite[Proposition 2.3.5]{Drupieski:2016a}), so $\calI_M \subseteq \sqrt{L_M}$. Finally, let $z \in L_M$. Taking $A = kG$, the odd isomorphism of \eqref{eq:oddhom} extends for each pair of rational $G$-supermodules $V$ and $W$ to an odd isomorphism $\Ext_G^\bullet(\bsPi(V),W) \simeq \Ext_G^\bullet(V,W)$, which the reader can check is compatible with left Yoneda multiplication by $\Ext_G^\bullet(W,W)$. Taking $V = k$ and $W = M$, this implies that $L_M$ is also the annihilator in $H(G,k)$ of $\Ext_G^\bullet(\bsPi(k),M)$. Now since $G$ is unipotent, it follows that $M$ admits a $G$-supermodule filtration of length $d = \dim_k(M)$ such that each section of the filtration is (even) isomorphic to either the trivial $G$-supermodule $k$ or its parity shift $\bsPi(k)$. Then applying the long exact sequence in cohomology in the first variable, it follows by induction on $d$ that $z^d \in I_M \subseteq \calI_M$, and hence $\sqrt{L_M} \subseteq \calI_M$.
\end{proof}

Let $A \in \calg_k$ be a purely even commutative $k$-algebra, and let $X = \Spec(A)$ be the corresponding affine $k$-scheme. Recall that points of $X$ correspond bijectively to equivalence classes of morphisms $\Spec(K) \rightarrow X$ for $K$ a field extension of $k$ \cite[\href{http://stacks.math.columbia.edu/tag/01J9}{Tag 01J9}]{stacks-project}. Equivalently, points of $X$ correspond to equivalence classes of $k$-algebra homomorphisms $A \rightarrow K$. Here two $k$-algebra homomorphisms $A \rightarrow K$ and $A \rightarrow L$ are equivalent if there exists a common field extension $\Omega$ of $K$ and $L$ such that the composites $A \rightarrow K \rightarrow \Omega$ and $A \rightarrow L \rightarrow \Omega$ are equal.\footnote{If $K$ and $L$ are both field extensions of $k$, then a common overfield $\Omega$ can always be found as a quotient of the commutative $k$-algebra $K \otimes_k L$.} Each equivalence class is a partially ordered set with $(A \rightarrow K) \leq (A \rightarrow L)$ if and only if $L$ is an extension field of $K$. Then the equivalence class corresponding to a point $x \in X$ contains a unique (up to unique isomorphism) minimal representative given by the canonical $k$-algebra homomorphism $A \rightarrow k(x)$. Thus,
	\[
	X = \varinjlim_K X(K) = \varinjlim_K \Hom_{\alg/k}(A,K),
	\]
where the limit is taken over the category whose objects are the field extensions of $K$ and whose morphisms are the injective $k$-algebra homomorphisms.

\begin{proposition} \label{prop:psirinvelementary}
Suppose $r \geq 2$. Let $G$ be a height-$r$ infinitesimal elementary $k$-supergroup scheme, and let $M$ be a finite-dimensional rational $G$-supermodule. Then the morphism of schemes $\Psi_r: \Vrg \rightarrow \abs{G}$ induced by the homomorphism $\psi_r : H(G,k) \rightarrow k[\Vrg]$ satisfies
	\[
	\Psi_r^{-1}(\abs{G}_M) = \Vrg_M.
	\]
\end{proposition}

We consider separately the cases $G = \Gar$, $G = \Mrs$, and $G = \Mrseta$ with $0 \neq \eta \in k$.

\begin{proof}[Proof of Proposition \ref{prop:psirinvelementary} for $G = \Gar$]
This case follows from Remark \ref{remark:psir}\eqref{item:psiragrees} and the classical calculation of \cite[Proposition 6.5]{Suslin:1997a}.
\end{proof}

\begin{proof}[Proof of Proposition \ref{prop:psirinvelementary} for $G = \Mrs$]
Suppose $G = \Mrs$ for some integer $s \geq 1$. Set $\Grs = (\Gaone)^{\times (r-1)} \times \Mones$. Then as $k$-superalgebras (but not as Hopf superalgebras),
	\[
	k\Mrs = k[u_0,\ldots,u_{r-1},v]/\subgrp{u_0^p,\ldots,u_{r-2}^p,u_{r-1}^p + v^2,u_{r-1}^{p^s}} \cong (k\Gaone)^{\otimes (r-1)} \otimes k\Mones = k\Grs.
	\]
Specifically, for $1 \leq i \leq r-1$, the $i$-th tensor factor of $k\Gaone$ corresponds to the $k$-subalgebra of $k\Mrs$ generated by $u_{i-1}$, and the factor of $k\Mones$ corresponds to the $k$-subalgebra of $k\Mrs$ generated by $u_{r-1}$ and $v$. Then the category of $k\Mrs$-supermodules is equivalent (though not tensor equivalent) to the category of $k\Grs$-supermodules. Given a $k\Mrs$-supermodule $M$, let $\wt{M}$ denote the same module considered as a $k\Grs$-supermodule via the isomorphism $k\Mrs \cong k\Grs$. Then $\Hbul(\Mrs,M) \cong \Hbul(\Grs,\wt{M})$. In particular, $H(\Mrs,k) \cong H(\Grs,k)$ as $k$-algebras because multi\-pli\-cation in the cohomology ring is given by Yoneda composition of extensions (cf.\ \cite[\S2.3]{Drupieski:2016a}), which does not depend on the Hopf structures of $k\Mrs$ or $k\Grs$. Similarly, the isomorphism $\Hbul(\Mrs,M) \cong \Hbul(\Grs,\wt{M})$ is compatible with the right cup product actions of $H(\Mrs,k)$ and $H(\Grs,k)$, since by \cite[Proposition 2.3.5]{Drupieski:2016a} the cup products are given by Yoneda compositions. Then it follows from Lemma \ref{lemma:unipotentannihilator} that the algebra isomorphism $H(\Mrs,k) \cong H(\Grs,k)$ induces an isomorphism of schemes $\abs{\Mrs} \cong \abs{\Grs}$, which maps $\abs{\Mrs}_M$ onto $\abs{\Grs}_{\wt{M}}$. Since $\Grs$ is a height-one infinitesimal unipotent $k$-supergroup scheme, we get by Lemma \ref{lemma:psiinverseheightone} that the morphism $\Psi_1: V_1(\Grs) \rightarrow \abs{\Grs}$ satisfies $\Psi_1^{-1}(\abs{\Grs}_{\wt{M}}) = V_1(\Grs)_{\wt{M}}$. Our strategy will be to use the result for $\Grs$ and the isomorphism $\abs{\Mrs}_M \cong \abs{\Grs}_{\wt{M}}$ to deduce that $\Psi_r^{-1}(\abs{\Mrs}_M) = V_r(\Mrs)_M$.

To show that $\Psi_r^{-1}(\abs{\Mrs}_M) = V_r(\Mrs)_M$, it suffices by the discussion preceding the proposition to show that the statement holds at the level of $K$-points for each field extension $K/k$. That is, it suffices to show for each field extension $K/k$ that the morphism on $K$-points,
	\[
	\Psi_r(K) : V_r(\Mrs)(K) = \Hom_{\alg/k}(k[V_r(\Mrs)],K) \stackrel{\psi_r^*}{\longrightarrow} \Hom_{\alg/k}(H(\Mrs,k),K) = \abs{\Mrs}(K),
	\]
satisfies $\Psi_r(K)^{-1}(\abs{\Mrs}_M(K)) = V_r(\Mrs)_M(K)$. So for the remainder of the proof fix an extension field $K$ of $k$. By abuse of notation, we may simply write $V_r(\Mrs)$ for the set of $K$-points of $V_r(\Mrs)$, and similarly for the sets of $K$-points of $\abs{\Mrs}$, $V_1(\Grs)$, and $\abs{\Grs}$. Then
	\[
	\abs{\Grs} \cong \abs{\Mrs} = \set{ (d,c_1,\ldots,c_r,e) \in K^{r+2} : c_r = d^2 \text{ if $s \geq 2$, and } e = 0 \text{ if $s = 1$}}.
	\]
Here we make the identification
	\[
	\Hbul(\Mrs,k) = \begin{cases}
	k[x_1,\ldots,x_r,y] \gotimes \Lambda(\lambda_1,\ldots,\lambda_r) & \text{if $s=1$,} \\
	k[x_1,\ldots,x_r,y,w_s]/\subgrp{x_r-y^2} \gotimes \Lambda(\lambda_1,\ldots,\lambda_r) & \text{if $s \geq 2$,}
	\end{cases}
	\]
as in \cite[Proposition 3.2.1]{Drupieski:2017a}. Then a point $(d,\ul{c},e) = (d,c_1,\ldots,c_r,e) \in \abs{\Mrs}$ corresponds to the unique $k$-algebra map $\phi = \phi_{(d,\ul{c},e)}: H(\Mrs,k) \rightarrow K$ such that $\phi(x_i) = c_i$ for $1 \leq i \leq r$, $\phi(y) = d$, and $\phi(w_s) = e$ if $s \geq 2$. Next, since $\Mrs$ is unipotent, we get by Lemma \ref{lemma:kbsvrg=kNrG} and \cite[Proposition 2.2.2]{Drupieski:2017b} that $V_r(\Mrs)(K) \cong \Hom_{Grp/K}(\Mr \otimes_k K,\Mrs \otimes_k K)$, the set of $K$-super\-group scheme homomorphisms $\Mr \otimes_k K \rightarrow \Mrs \otimes_k K$, and hence
	\[
	V_r(\Mrs) \cong \set{ (\mu,a_0,\ldots,a_{r-1},b) \in K^{r+2}: \mu^2 = a_0^{p^r} \text{ if $s \geq 2$, and } b = 0 \text{ if $s=1$}}.
	\]
Given a point $(\mu,\ul{a},b) = (\mu,a_0,\ldots,a_{r-1},b) \in V_r(\Mrs)$, the corresponding comorphism $\phi = \phi_{(\mu,\ul{a},b)}: K[\Mrs] \rightarrow K[\Mr]$ is the unique map of Hopf $K$-superalgebras such that $\phi(\tau) = \mu \cdot \tau$, $\phi(\theta) = \sum_{i=0}^{r-1} a_i \theta^{p^i}$, $\phi(\sigma_i) = a_0^{ip^{r-1}} \cdot \sigma_i$ for $0 \leq i < p^{s-1}$, and $\phi(\sigma_{p^{s-1}}) = a_0^{p^{r+s-2}} \cdot \sigma_{p^{s-1}} + b \cdot \sigma_1$. Similarly, $V_1(\Grs)(K) \cong \Hom_{Grp/K}(\Mone \otimes_k K,\Grs \otimes_k K)$. Then it follows from \cite[Proposition 2.2.2]{Drupieski:2017b} that
	\[
	V_1(\Grs) \cong \set{ (\mu,a_0,\ldots,a_{r-1},b) \in K^{r+2}: \mu^2 = a_{r-1}^p \text{ if $s \geq 2$, and } b = 0 \text{ if $s=1$}}.
	\]
More specifically, write
	\begin{align*}
	K[\Grs] &= K[\Gaone]^{\otimes (r-1)} \otimes K[\Mones] = K[\theta_0]/\subgrp{\theta_0^p} \otimes \cdots \otimes K[\theta_{r-2}]/\subgrp{\theta_{r-2}^p} \otimes K[\Mones], \text{ and} \\
	K[\Mones] &= \textstyle K[\tau,\sigma_1,\ldots,\sigma_{p^s-1}]/\subgrp{\tau^2 = 0 \text{ and } \sigma_i \sigma_j = \binom{i+j}{i} \sigma_{i+j} \text{ for $1 \leq i,j < p^s$}}.
	\end{align*}
Then the comorphism $\phi = \phi_{(\mu,\ul{a},b)}: K[\Grs] \rightarrow K[\Mone]$ corresponding to the point $(\mu,\ul{a},b) \in V_1(\Grs)$ is the unique map of Hopf $K$-superalgebras such that $\phi(\theta_i) = a_i \cdot \theta$ for $0 \leq i \leq r-2$, $\phi(\tau) = \mu \cdot \tau$, $\phi(\sigma_i) = a_{r-1}^i$ for $0 \leq i < p^{s-1}$, and $\phi(\sigma_{p^{s-1}}) = (a_{r-1})^{p^{s-1}} \cdot \sigma_{p^{s-1}} + b \cdot \sigma_1$.

Given $(\mu,\ul{a},b) \in V_r(\Mrs)$, let $\whphi_{(\mu,\ul{a},b)}: k[V_r(\Mrs)] \rightarrow K$ be the corresponding map of $k$-algebras, and let $\nu_{(\mu,\ul{a},b)} = \phi_{(\mu,\ul{a},b)}^*: \Mr \otimes_k K \rightarrow \Mrs \otimes_k K$ be the corresponding map of supergroups. Let $z \in \opH^n(\Mrs,k)_{\ol{n}}$. Then by the definition of $\psi_r$, the scalar $\whphi_{(\mu,\ul{a},b)} \circ \psi_r(z) \in K$ is equal to the coefficient of $y^n$ in $\nu_{(\mu,\ul{a},b)}^*(z_K) \in \opH^n(\Mr,k) \otimes_k K = \opH^n(\Pr,k) \otimes_k K$. Applying \cite[Lemma 3.1.1]{Drupieski:2017b} and the description of the inflation map $\Hbul(\Mrs,k) \rightarrow \Hbul(\Mr,k) = \Hbul(\Pr,k)$ in \cite[Proposition 3.2.1]{Drupieski:2017a}, it then follows that the map on $K$-points $\Psi_r: V_r(\Mrs) \rightarrow \abs{\Mrs}$ is given by
	\[
	\Psi_r(\mu,a_0,\ldots,a_{r-1},b) = (\mu,a_{r-1}^p,a_{r-2}^{p^2},\ldots,a_0^{p^r},b^p).
	\]
By similar reasoning, it follows that the map on $K$-points $\Psi_1: V_1(\Grs) \rightarrow \abs{\Grs}$ is given by
	\[
	\Psi_1(\mu,a_0,\ldots,a_{r-1},b) = (\mu,a_0^p,\ldots,a_{r-1}^p,b^p).
	\]
Define $h: V_r(\Mrs) \rightarrow V_1(\Grs)$ on $K$-points by $h(\mu,a_0,\ldots,a_{r-1},b) = (\mu,a_{r-1},a_{r-2}^p,\ldots,a_0^{p^{r-1}},b)$. Then the following diagram on sets of $K$-points commutes:
	\[
	\xymatrix{
	V_r(\Mrs) \ar@{->}[r]^{\Psi_r} \ar@{->}[d]^{h} & \abs{\Mrs} \ar@{->}[d]^{=} \\
	V_1(\Grs) \ar@{->}[r]^{\Psi_1} & \abs{\Grs}.
	}
	\]
Since $\Psi_1^{-1}(\abs{\Grs}_{\wt{M}}) = V_1(\Grs)_{\wt{M}}$, this implies that $\Psi_r^{-1}(\abs{\Mrs}_M) = h^{-1}(V_1(\Grs)_{\wt{M}})$.

Let $(\mu,\ul{a},b) \in V_r(\Mrs)$. Then the map of Hopf $K$-superalgebras $\rho_{(\mu,\ul{a},b)}: K\Pr \rightarrow K\Mrs$ labeled by $(\mu,\ul{a},b)$ is the induced via duality by the map $\phi_{(\mu,\ul{a},b)}: K[\Mrs] \rightarrow K[\Mr]$. Then in terms of the $K$-basis \eqref{eq:gammaell} for $K\Pr$, one can check 
	\begin{equation} \label{eq:rhoMrs}
	\begin{split}
	\rho_{(\mu,\ul{a},b)}(v) &= \mu \cdot v, \quad \text{and} \\
	\rho_{(\mu,\ul{a},b)}(u_{r-1}) &= b \cdot u_{r-1}^{p^{s-1}} + \sum \binom{i}{i_0,\ldots,i_{r-1}} a_0^{i_0} \cdots a_{r-1}^{i_{r-1}} \cdot \gamma_i,
	\end{split}
	\end{equation} 
where the sum is over all integers $0 < i \leq p^{r-1}$ and $i_0,\ldots,i_{r-1} \in \N$ such that $i_0+\cdots+i_{r-1} = i$ and $i_0+i_1p+\cdots+i_{r-1}p^{r-1}=p^{r-1}$, and $\binom{i}{i_0,\ldots,i_{r-1}}$ denotes the usual multinomial coefficient. In particular, $\rho_{(\mu,\ul{a},b)}(u_{r-1})$ is congruent modulo the square of the ideal $\subgrp{u_0,\ldots,u_{r-2}} \subset K\Mrs$ to
	\begin{equation} \label{eq:rhoMrsmodrad2}
	b \cdot u_{r-1}^{p^{s-1}} + a_0^{p^{r-1}} \cdot u_{r-1} + a_1^{p^{r-2}} \cdot u_{r-2} + \cdots + a_{r-2}^p \cdot u_1 + a_{r-1} \cdot u_0.
	\end{equation}
Similarly, one can check the map of Hopf $K$-superalgebras $K\Pone \rightarrow K\Grs \cong K\Mrs$ labeled by $h(\mu,\ul{a},b) = (\mu,a_{r-1},a_{r-2}^p,\ldots,a_0^{p^{r-1}},b) \in V_1(\Grs)$ is given by
	\begin{equation} \label{eq:rhoGrs}
	\begin{split}
	v &\mapsto \mu \cdot v, \quad \text{and} \\
	u &\mapsto b \cdot u_{r-1}^{p^{s-1}} + a_0^{p^{r-1}} \cdot u_{r-1} + a_1^{p^{r-2}} \cdot u_{r-2} + \cdots + a_{r-2}^p \cdot u_1 + a_{r-1} \cdot u_0.
	\end{split}
	\end{equation}
From this we deduce that
	\begin{align*}
	\Psi_r^{-1}(\abs{\Mrs}_M) &= h^{-1}(V_1(\Grs)_{\wt{M}}) \\
	&= \{ (\mu,a_0,\ldots,a_{r-1},b) \in V_r(\Mrs) : \pd_{K\Pone}(M_K) = \infty \\
	&\phantom{= \{ \ } \text{when the action of $\Pone$ on $M_K$ is given by \eqref{eq:rhoGrs}} \}
	\end{align*}
Since the ideal $\subgrp{u_0,\ldots,u_{r-2}} \subset K\Mrs$ is generated by $p$-nilpotent elements, Proposition \ref{prop:secondradical} and the observation \eqref{eq:rhoMrsmodrad2} imply that $\pd_{K\Pone}(M_K) = \infty$ when the action of $K\Pone$ is specified by \eqref{eq:rhoGrs} if and only if $\pd_{K\Pone}(M_K) = \infty$ when the action of $K\Pone$ is specified by the inclusion $K\Pone \hookrightarrow K\Pr$ and the formulas \eqref{eq:rhoMrs}. So $\Psi_r^{-1}(\abs{\Mrs}_M) = V_r(\Mrs)_M$.
\end{proof}

\begin{proof}[Proof of Proposition \ref{prop:psirinvelementary} for $G = \Mrseta$ with $0 \neq \eta \in k$]
This reasoning for this case exactly parallels (down to the results cited for justification at each step) the reasoning for $G = \Mrs$, so we will just describe the details that are different. By \cite[Remark 3.1.8(4)]{Drupieski:2017a}, the $K$-superalgebra map $\pi: K\Mrseta \rightarrow K\M_{r-1;s+1}$ defined on generators by $\pi(v) = v$, $\pi(u_i) = u_{i-1}$ for $1 \leq i \leq r-1$, and $\pi(u_0) = (-\eta^{-1}) \cdot u_{r-2}^{p^s}$, is an isomorphism. Then the categories of $k\Mrseta$-supermodules and $k\Grmsp$-super\-modules are equivalent. Given a $k\Mrseta$-supermodule $M$, write $\wt{M}$ for $M$ considered as a $k\Grmsp$-super\-module via the superalgebra  isomorphism $k\Mrseta \cong k\M_{r-1;s+1} \cong k\Grmsp$. Then $\abs{\Mrseta}_M \cong \abs{\Grmsp}_{\wt{M}}$.  Next,
	\[
	\Hbul(\Mrseta,k) \cong k[x_1,\ldots,x_{r-1},y,w]/\subgrp{x_{r-1}-y^2} \gotimes \Lambda(\lambda_1,\ldots,\lambda_{r-1}),
	\]
so the set of $K$-points of $\abs{\Mrseta} \cong \abs{\Grmsp}$ is given by
	\[
	\abs{\Grmsp} \cong \abs{\Mrseta} = \set{ (d,c_1,\ldots,c_{r-1},e) \in K^{r+1}: c_{r-1} = d^2},
	\]
with a point $(d,\ul{c},e)$ corresponding to the unique $k$-algebra map $\phi = \phi_{(d,\ul{c},e)}: H(\Mrseta,k) \rightarrow K$ such that $\phi(x_i) = c_i$ for $1 \leq i \leq r-1$, $\phi(y) = d$, and $\phi(w) = e$.

As in the case $G = \Mrs$, we get $V_r(\Mrseta)(K) \cong \Hom_{Grp/K}(\Mr \otimes_k K,\Mrseta \otimes_k K)$, so
	\[
	V_r(\Mrseta) \cong \set{ (\mu,a_0,\ldots,a_{r-1}) \in K^{r+1} : \mu^2 = a_0^{p^r}}.
	\]
The Hopf superalgebra structure of $k[\Mrseta]$ is described for $r \geq 2$ in \cite[Lemma 2.2.1]{Drupieski:2017b}. In particular, $k[\Mrseta] \cong k[\Mrs]$ as $k$-superalgebras, and then via this identification, the comorphism $\phi = \phi_{(\mu,\ul{a})}: K[\Mrseta] \rightarrow K[\Mr]$ corresponding to the point $(\mu,\ul{a}) \in V_r(\Mrseta)$ is specified by the formulas $\phi(\tau) = \mu \cdot \tau$, $\phi(\theta) = (\sum_{i=0}^{r-1} a_i \cdot \theta^{p^i}) - \eta (a_0^{p^{r+s-1}} \cdot \sigma_{p^s})$, and $\phi(\sigma_i) = a_0^{ip^{r-1}} \cdot \sigma_i$ for $0 \leq i < p^s$. The morphism on $K$-points $\Psi_r: V_r(\Mrseta) \rightarrow \abs{\Mrseta}$ is given by
	\[
	\Psi_r(\mu,a_0,\ldots,a_{r-1}) = (\mu,a_{r-2}^{p^2},a_{r-3}^{p^3},\ldots,a_0^{p^r},(-\eta^{-1})^p \cdot a_{r-1}^p).
	\]
On the other hand, the set of $K$-points of $V_1(\Grmsp)$ is given by
	\[
	V_1(\Grmsp) = \set{ (\mu,a_0,\ldots,a_{r-2},b) \in K^{r+1}: \mu^2 = a_{r-2}^p},
	\]
and the morphism $\Psi_1: V_1(\Grmsp) \rightarrow \abs{\Grmsp}$ on $K$-points is given by
	\[
	\Psi_1(\mu,a_0,\ldots,a_{r-2},b) = (\mu,a_0^p,\ldots,a_{r-2}^p,b^p).
	\]
Then the map $h: V_r(\Mrseta) \rightarrow V_1(\Grmsp)$ that makes the square
	\begin{equation} \label{eq:Mrsetasquare}
	\vcenter{\xymatrix{
	V_r(\Mrseta) \ar@{->}[r]^{\Psi_r} \ar@{->}[d]^{h} & \abs{\Mrseta} \ar@{->}[d]^{=} \\
	V_1(\Grmsp) \ar@{->}[r]^{\Psi_1} & \abs{\Grmsp}
	}}
	\end{equation}
commute is defined by $h(\mu,a_0,\ldots,a_{r-1}) = (\mu,a_{r-2}^p,a_{r-3}^{p^2},\ldots,a_0^{p^{r-1}},(-\eta^{-1}) \cdot a_{r-1})$.

Let $(\mu,\ul{a}) \in V_r(\Mrseta)$. Then the map $\rho_{(\mu,\ul{a})}: K\Pr \rightarrow K\Mrseta$ labeled by $(\mu,\ul{a})$ satisfies
	\begin{equation} \label{eq:rhoMrseta}
	\begin{split}
	\rho_{(\mu,\ul{a})}(v) &= \mu \cdot v, \quad \text{and} \\
	\rho_{(\mu,\ul{a})}(u_{r-1}) &= \sum \binom{i}{i_0,\ldots,i_{r-1}} a_0^{i_0} \cdots a_{r-1}^{i_{r-1}} \cdot \gamma_i,
	\end{split}
	\end{equation}
where the sum is over all integers $0 < i \leq p^{r-1}$ and $i_0,\ldots,i_{r-1} \in \N$ such that $i_0+\cdots+i_{r-1}=i$ and $i_0 + i_1p+\cdots+i_{r-1}p^{r-1} = p^{r-1}$. In particular, $\rho_{(\mu,\ul{a})}(u_{r-1})$ is congruent modulo the square of the ideal $\subgrp{u_0,\ldots,u_{r-2}} \subset K\Mrseta$ to
	\begin{equation} \label{eq:rhoMrsetamodrad}
	a_0^{p^{r-1}} \cdot u_{r-1} + a_1^{p^{r-2}} \cdot u_{r-2} + \cdots + a_{r-1} \cdot u_0.
	\end{equation}
The Hopf superalgebra map $\rho: K\Pone \rightarrow K\Grmsp$ labeled by $h(\mu,\ul{a})$ is given by
	\begin{equation} \label{eq:rhoGrmsp}
	\begin{split}
	\rho(v) &= \mu \cdot v, \quad \text{and} \\
	\rho(u) &= (-\eta^{-1}) a_{r-1} \cdot u_{r-2}^{p^s} + a_0^{p^{r-1}} \cdot u_{r-2} + \cdots + a_{r-3}^{p^2} \cdot u_1 + a_{r-2}^p \cdot u_0.
	\end{split}
	\end{equation}
Composing with the inverse isomorphism $\pi^{-1}: K\Grmsp \cong K\M_{r-1;s+1} \rightarrow K\Mrseta$, we get
	\begin{equation}
	\begin{split}
	(\pi^{-1} \circ \rho)(v) &= \mu \cdot v, \quad \text{and} \\
	(\pi^{-1} \circ \rho)(u) &= a_{r-1} \cdot u_0 + a_0^{p^{r-1}} \cdot u_{r-1} + \cdots + a_{r-3}^{p^2} \cdot u_2 + a_{r-2}^p \cdot u_1.
	\end{split}
	\end{equation}
Since the ideal $\subgrp{u_0,\ldots,u_{r-2}} \subset K\Mrseta$ is generated by $p$-nilpotent elements, we deduce by the observation \eqref{eq:rhoMrsetamodrad} and Proposition \ref{prop:secondradical} that $\pd_{K\Pone}(M_K) = \infty$ when the action of $K\Pone$ is specified by \eqref{eq:rhoGrmsp} if and only if $\pd_{K\Pone}(M_K) = \infty$ when the action of $K\Pone$ is specified by the inclusion $K\Pone \hookrightarrow K\Pr$ and the formulas \eqref{eq:rhoMrseta}. Then by the commutativity of the square \eqref{eq:Mrsetasquare} and the equality $\Psi_1^{-1}(\abs{\Grmsp}_{\wt{M}}) = V_1(\Grmsp)_{\wt{M}}$, we deduce that $\Psi_r^{-1}(\abs{\Mrseta}_M) = V_r(\Mrseta)_M$.
\end{proof}

\begin{corollary} \label{cor:psiinverse}
Let $r \geq 1$ be arbitrary, let $G$ be an infinitesimal elementary $k$-supergroup scheme of height $\leq r$, and let $M$ be a finite-dimensional rational $G$-supermodule. Then the homomorphism $\psi_r : H(G,k) \rightarrow k[\Vrg]$ satisfies $\psi_r^{-1}(\calJ_M) = \calI_M$.
\end{corollary}

\begin{proof}
If the height of $G$ is exactly $r$, then the result immediately follows from either Lemma \ref{lemma:psiinverseheightone} or Proposition \ref{prop:psirinvelementary}. Otherwise, suppose $G$ is infinitesimal of height $r'$, with $r' < r$. The canonical projection map $\Pr \twoheadrightarrow \P_{r'}$ defines by Lemma \ref{lemma:HomRS} a closed embedding $V_{r'}(G) \hookrightarrow \Vrg$. If $K/k$ is a field extension and if $\nu \in \Vrg(K) = V_r(G_K)(K)$, then we see by Lemma \ref{lemma:bsvrg} that $\nu: K\Pr \rightarrow KG$ factors through the canonical quotient map $K\Pr \twoheadrightarrow KE$ for some closed multi\-parameter $K$-subsupergroup scheme $E$ of $G_K$. Then $E$ must also be infinitesimal of height $\leq r'$, and hence the canonical quotient map $K\Pr \twoheadrightarrow KE$ factors through the canonical quotient map $K\Pr \twoheadrightarrow K\P_{r'}$. Thus $\nu$ is in the image of the map on $K$-points $V_{r'}(G)(K) \hookrightarrow V_r(G)(K)$. Since $K$ was arbitrary, this implies by the discussion preceding Proposition \ref{prop:psirinvelementary} that the closed embedding $V_{r'}(G) \hookrightarrow V_r(G)$ is an equality. Now the claim for $\psi_r$ follows from the corresponding result for $\psi_{r'}: H(G,k) \rightarrow k[V_{r'}(G)]$; cf.\ the proof of \cite[Corollary 3.5.5]{Drupieski:2017b}.
\end{proof}

\subsection{Arbitrary infinitesimal unipotent supergroups}

We now apply the results of the previous subsections to deduce the analogue of Lemma \ref{lemma:psiinverseheightone} and Proposition \ref{prop:psirinvelementary} for arbitrary infinitesimal unipotent $k$-super\-group schemes.

\begin{theorem} \label{thm:Psirinv}
Let $G$ be an infinitesimal unipotent $k$-supergroup scheme of height $\leq r$, and let $M$ be a finite-dimensional rational $G$-supermodule. Then the homomorphism $\psi = \psi_r: H(G,k) \rightarrow k[\Vrg]$ satisfies $\psi^{-1}(\calJ_M) = \calI_M$, and hence the associated morphism of schemes $\Psi_r: \Vrg \rightarrow \abs{G}$ satisfies $\Psi_r^{-1}(\abs{G}_M) = \Vrg_M$. Thus, $\Psi_r$ restricts to a homeomorphism
	\[
	\Psi_r: \Vrg_M \stackrel{\sim}{\rightarrow} \abs{G}_M.
	\]
\end{theorem}

\begin{proof}
First suppose $z \in \calI_M$; we want to show for each $\fs \in \Vrg_M$ that $\phi_{\fs}(\psi(z)) = 0$. So fix $\fs \in \Vrg_M$, and set $K = k(\fs)$. Let $\nu_\fs: K\Pr \rightarrow KG_K$ and $\iota_K: K\Pone \hookrightarrow K\Pr$ be the maps as defined in the beginning of Section \ref{subsec:definesupportscheme}. Then $\pd_{K\Pone}(\iota_K^* \nu_\fs^* M_K) = \infty$ by the definition of $\Vrg_M$. Here we have written $\iota_K^* \nu_{\fs}^* M_K$ to emphasize that the $K\Pone$-super\-module structure on $M_K$ is obtained by pulling back first along $\nu_{\fs}$ and then along $\iota_K$. By Lemma \ref{lemma:kbsvrg=kNrG}, $\nu_\fs$ factors for some integer $s \geq 1$ through the canonical quotient map $K\Pr \twoheadrightarrow K\Mrs$, and hence factors as $K\Pr \twoheadrightarrow KE \hookrightarrow KG_K$ for some elementary $K$-super\-group scheme $E$. Here $\pi: K\Pr \twoheadrightarrow KE$ is the canonical quotient map, and $\nu: KE \hookrightarrow KG_K$ is an injective Hopf super\-algebra map. Then $E$ identifies with a closed $K$-sub\-super\-group scheme of $G_K$, and so is also infinitesimal of height $\leq r$. Now $\nu^*(z_K) \in \calI_{\nu^*M_K}^E$, so $\nu^*(\psi(z)_K) = \psi_E(\nu^*(z_K)) \in \calJ_{\nu^*M_K}^E$ by Corollary \ref{cor:psiinverse}. Then $\phi_{\fp}(\nu^*(\psi(z)_K)) = 0$ for all points $\fp \in V_r(E)_{\nu^*M_K}$. By a slight abuse of notation, let $\pi \in V_r(E)$ denote the point corresponding as in the discussion preceding Proposition \ref{prop:psirinvelementary} to the equivalence class of $\pi: K\Pr \twoheadrightarrow KE$, and let $\phi_{\pi}: K[V_r(E)] \rightarrow K$ be the corresponding $K$-algebra homomorphism. Since $\nu_{\fs} = \nu \circ \pi$, then $\pd_{K\Pone}(\iota_K^* \pi^*\nu^* M_K) = \pd_{K\Pone}(\iota_K^* \nu_\fs^* M_K) = \infty$, and hence $\pi \in V_r(E)_{\nu^* M_K}$. Then $\phi_{\pi}(\nu^*(\psi(z)_K)) = 0$. So to show that $\phi_\fs(\psi(z)) = 0$, it suffices to show that $\phi_{\fs} = \phi_\pi \circ \nu^*: K[V_r(G_K)] \rightarrow K$, where by abuse of notation we have also written $\phi_{\fs}$ for the unique extension of $\phi_\fs$ to a $K$-algebra map $K[V_r(G_K)] = k[\Vrg] \otimes_k K \rightarrow K$. On the one hand, $\phi_{\fs}: K[V_r(G_K)] \rightarrow K$ is the unique $K$-algebra map such that $u_{G_K} \otimes_{\phi_{\fs}} K = \nu_{\fs} : K\Pr \rightarrow KG_K$. On the other hand, consider the universal Hopf superalgebra map $u_E: K\Pr \otimes_K K[V_r(E)] \rightarrow KE \otimes_K K[V_r(E)]$. One can check 
	\[
	(\nu \otimes 1) \circ u_E = u_{G_K} \otimes_{\nu^*} K[V_r(E)]: K\Pr \otimes_K K[V_r(E)] \rightarrow KG_K \otimes K[V_r(E)].
	\]
Then
	\begin{align*}
	u_{G_K} \otimes_{\phi_{\pi} \circ \nu^*} K &= (u_{G_K} \otimes_{\nu^*} K[V_r(E)]) \otimes_{\phi_{\pi}} K \\
	&= [(\nu \otimes 1) \circ u_E] \otimes_{\phi_{\pi}} K \\
	&= \nu \circ (u_E \otimes_{\phi_{\pi}} K) \\
	&= \nu \circ \pi = \nu_{\fs},
	\end{align*}
and hence $\phi_{\fs} = \phi_{\pi} \circ \nu^*$ by the uniqueness of $\phi_{\fs}$.

Now suppose that $\psi(z) \in \calJ_M$. Let $K$ be an extension field of $k$, and let $E$ be an (infinitesimal) elementary $K$-sub\-super\-group scheme of $G_K$. Denote the closed embedding of $E$ into $G_K$ by $\nu: E \hookrightarrow G_K$. As in the proof of Lemma \ref{lemma:psiinverseheightone}, one gets $\psi(z)_K \in \calJ_{M_K}$, and hence $\nu^*(\psi(z)_K) = \psi_E(\nu^*(z_K)) \in \calJ_{M_K}^E$. This implies by Lemma \ref{lemma:psiinverseheightone} and Proposition \ref{prop:psirinvelementary} that $\nu^*(z_K) \in \calI_{M_K}^E$, or equivalently that $\rho_{\Lambda_K}(\nu^*(z_K)) = \nu^*(\rho_{\Lambda}(z)_K)$ is nilpotent in the algebra $\Hbul(E,\Lambda_K)$. Here $\Lambda = \End_k(M)$. Now since $K$ and $E$ were arbitrary, this implies by \cite[Theorem 8.4]{Benson:2018} that $\rho_{\Lambda}(z)$ is nilpotent in the algebra $\Hbul(G,\Lambda)$, and hence $z \in \calI_M$.
\end{proof}

\section{Applications and Examples}\label{S:Applications}

\subsection{Consequences} \label{SS:applications}

In this section we collect some consequence of Theorems \ref{thm:psiuniversalhomeo} and \ref{thm:Psirinv}. Having established the requisite precursors, most of the results follow more or less formally from arguments already in the literature, so we omit many of the details. The first two results below follow by exactly the same lines of reasoning as Proposition~7.1 and Theorem~7.2 of \cite{Suslin:1997a}.

\begin{theorem}[Naturality of supports] \label{thm:naturality}
Let $f: H \rightarrow G$ be a homomorphism of infinitesimal uni\-potent $k$-supergroup schemes, and let $M$ be a finite-dimensional rational $G$-supermodule, considered also as a rational $H$-super\-module via pullback along $f$. Then the induced morphism of schemes $f_{*}: \abs{H} \to \abs{G}$ satisfies
\[
f^{-1}_{*}\left(\abs{G}_{M} \right) = \abs{H}_{M}.
\]
In particular, if $f$ is a closed embedding, and if we identify $\abs{H}$ with a subset of $\abs{G}$ via the injection $f_*: \abs{H} \rightarrow \abs{G}$ of Corollary \ref{cor:subgroupspectrum}, then $\abs{H}_M = \abs{H} \cap \abs{G}_M$.
\end{theorem}

\begin{theorem}[Tensor product property] \label{thm:tensorproductproperty}
Let $G$ be an infinitesimal unipotent $k$-supergroup scheme, and let $M$ and $N$ be finite-dimensional rational $G$-supermodules. Then 
\begin{equation} \label{eq:tensorproductproperty}
\abs{G}_{M\otimes N} = \abs{G}_M \cap \abs{G}_N.
\end{equation}
\end{theorem}

Applying Theorem \ref{thm:Psirinv} to \eqref{eq:tensorproductproperty}, one obtains:

\begin{corollary} \label{cor:Vrgtensorproduct}
Let $G$ be an infinitesimal unipotent $k$-supergroup scheme of height $\leq r$, and let $M$ and $N$ be finite-dimensional rational $G$-supermodules. Then
\begin{equation} \label{eq:Vrgtensorproduct}
\Vrg_{M\otimes N} = \Vrg_M \cap \Vrg_N.
\end{equation}
\end{corollary}

Next, let $G$ be a finite $k$-supergroup scheme, and let $(P_\bullet,d)$ be a minimal projective resolution of the trivial module $k$ in the category $(\fsmod_{kG})_{\ev}$. Set $\Omega^n(k) = \ker(d: P_n \rightarrow P_{n-1})$. For each $n \in \N$ one has $\opH^n(G,k) = \opH^n(\Hom_{kG}(P_\bullet,k)) \cong \Hom_{kG}(\Omega^n(k),k)$, so given a homogeneous cohomology class $\zeta \in \opH^n(G,k)$, there exists a representative linear map $\wh{\zeta}: \Omega^n(k) \rightarrow k$ of the same parity as $\zeta$. Set $L_\zeta = \ker(\wh{\zeta})$. Then $L_\zeta$ is a finite-dimensional $kG$-supermodule. Given a homogeneous ideal $I \subset H(G,k)$, let $Z(I) \subseteq \abs{G}$ be the Zariski closed, conical subset of $\abs{G}$ defined by $I$.

The next theorem follows from essentially a word-for-word repetition of the proof of \cite[Theorem 2.5]{Feldvoss:2010} (see also \cite{Feldvoss:2015}), making the following substitutions: the Hopf algebra $A$ is replaced with the Hopf superalgebra $kG$, $\opH^{\ev}(A,k)$ is replaced with $H(G,k)$, and instead of considering maximal ideals $\fm \in \Max(\opH^{\ev}(A,k))$ one instead considers prime ideals $\fp \in \Spec(H(G,k))$. In this context, the finite-generation assumption of \cite{Feldvoss:2010} holds by the main theorem of \cite{Drupieski:2016}, and the additional quasi\-tri\-ang\-ularity hypothesis of \cite{Feldvoss:2015} is obviated by the fact $kG$ is (super)cocommutative. The proof of \cite[Theorem 2.5]{Feldvoss:2010} also uses some general properties of relative support varieties \cite[Proposition 2.4(4)--(5)]{Feldvoss:2010}; super analogues of these results are stated in \cite[\S2.3]{Drupieski:2016a}, and these properties hold more generally at the level of schemes via precisely the same proofs.

\begin{theorem}[Realization]
Let $G$ be an infinitesimal unipotent $k$-supergroup scheme, and let $\zeta \in \opH^n(G,k)$ be a homogeneous element. Then $\abs{G}_{L_\zeta} = Z(\subgrp{\zeta})$, where $\subgrp{\zeta}$ is the homogeneous ideal of $\opH^n(G,k)$ generated by $\zeta$. More generally, let $W \subseteq \abs{G}$ be a Zariski closed, conical subset of $\abs{G}$ defined by the homogeneous ideal $I = \subgrp{\zeta_1,\ldots,\zeta_t}$. Then
	\[
	W = Z(I) = Z(\subgrp{\zeta_1}) \cap \cdots \cap Z(\subgrp{\zeta_t}) = \abs{G}_{L_{\zeta_1}} \cap \cdots \cap \abs{G}_{L_{\zeta_t}} = \abs{G}_{L_{\zeta_1} \otimes \cdots \otimes L_{\zeta_t}}.
	\]
Thus, a subset $W \subseteq \abs{G}$ is of the form $\abs{G}_M$ for some finite-dimensional rational $G$-supermodule $M$ if and only if $W$ is a Zariski closed, conical subset of $\abs{G}$.
\end{theorem}

Lemma \ref{lemma:bsvrg} describes how $\Vrg(k)$ is stratified by pieces coming from the multiparameter $k$-subsupergroup schemes of $G$. The next theorem translates this to the support variety $\abs{G}_M$.

\begin{theorem} \label{theorem:supportvarietyunion}
Suppose $k$ is algebraically closed. Let $G$ be an infinitesimal unipotent $k$-supergroup scheme of height $\leq r$, and let $M$ be a finite-dimensional rational $G$-supermodule. Then the support variety $\abs{G}_M$ (i.e., the set of $k$-points of the scheme of the same name) can be written as
	\[
	\abs{G}_M = \bigcup_{E \leq G} \res_{G,E}^*(\abs{E}_M),
	\]
where the union is taken over all multiparameter (closed) $k$-subsupergroup schemes $E$ of $G$, and $\res_{G,E}: H(G,k) \rightarrow H(E,k)$ is the restriction map induced by the embedding $E \hookrightarrow G$.
\end{theorem}

\begin{proof}
First, since $k$ is algebraically closed, the universal homeomorphism of Theorem \ref{thm:psiuniversalhomeo} induces a homeomorphism of affine algebraic varieties $\Vrg(k) \simeq \abs{G}$, which is natural with respect to subgroup inclusions by Lemma \ref{lemma:psirnatural}. Next, $\Vrg(k) = \bigcup_{E \leq G} V_r(E)(k)$ by Lemma \ref{lemma:bsvrg}. Transporting this decomposition across the homeomorphism, we get $\abs{G} = \bigcup_{E \leq G} \res_{G,E}^*(\abs{E})$. Finally, by naturality of supports (Theorem \ref{thm:naturality}), it follows that this decomposition of $\abs{G}$ restricts to the decomposition $\abs{G}_M = \bigcup_{E \leq G} \res_{G,E}^*(\abs{E}_M)$.
\end{proof}

The following corollary should be contrasted with the projectivity detection theorem of Benson, Iyengar, Krause, and Pevtsova \cite{Benson:2018}, which in general requires the consideration of field extensions (but is also applicable to infinite-dimensional modules).

\begin{corollary}
Suppose $k$ is algebraically closed. Let $G$ be an infinitesimal unipotent $k$-super\-group scheme, and let $M$ be a finite-dimensional rational $G$-supermodule. Then $M$ is projective as a $G$-super\-module if and only if $M$ is projective as an $E$-supermodule for each elementary subsupergroup scheme $E$ of $G$.
\end{corollary}

\begin{proof}
By \cite[Proposition 2.3.13]{Drupieski:2016a}, $M$ is projective as a $G$-supermodule if and only if the support variety $\abs{G}_M$ is trivial. Since $\abs{G}_M = \bigcup_{E \leq G} \res_{G,E}^*(\abs{E}_M)$, and since $\res_{G,E}^*: \abs{E} \rightarrow \abs{G}$ is injective by Corollary \ref{cor:subgroupspectrum}, this implies that $\abs{G}_M$ is trivial if and only if each $\abs{E}_M$ is trivial, i.e., if and only if $M$ is projective as an $E$-supermodule for each elementary subsupergroup $E$ of $G$.
\end{proof}

As mentioned in the introduction, a stratification result like Theorem \ref{theorem:supportvarietyunion} already appears in the literature in the context of finite-dimensional graded connected cocommutative Hopf algebras; see 
\cite[Theorem 3.2]{Nakano:1998} and \cite[Theorem 1.4]{Palmieri:1997}. These earlier stratification theorems rely on a general $F$-surjectivity theorem stated in \cite[Theorem 4.1]{Palmieri:1997}. However, the proof of the $F$-surjectivity theorem implicitly requires that the Hopf algebra in question has only finitely many graded Hopf subalgebras, and as we show in Example \ref{example:Fsurjectivity} below, $F$-surjectivity need not hold when this assumption is not satisfied. We expect results like \cite[Theorem 3.2]{Nakano:1998} and \cite[Theorem 1.4]{Palmieri:1997} should hold in general. As evidence, we note the group algebras of the supergroups appearing in Theorem \ref{theorem:supportvarietyunion} are quasi-elementary in the sense of \cite[Definition 2.4]{Nakano:1998} (provided one weakens the definition to also allow Hopf sub-superalgebras, and not just $\Z$-graded Hopf subalgebras), and hence are of the type considered in \cite{Nakano:1998,Palmieri:1997}. However, the results of the present paper do not definitively settle the issue: the inclusions appearing in \cite{Nakano:1998,Palmieri:1997} are maps of $\Z$-graded Hopf algebras (i.e., maps which preserve the $\Z$-gradings), whereas the inclusions $kE \leq kG$ appearing in Theorem \ref{theorem:supportvarietyunion} are only maps of Hopf superalgebras. Thus Theorem \ref{theorem:supportvarietyunion} is not directly comparable with the stratification theorems described in \cite{Nakano:1998,Palmieri:1997}.

\begin{example} \label{example:Fsurjectivity}
Let $k$ be an infinite field of characteristic $2$, and let $A = k[u,v]/\subgrp{u^2,v^2}$. We consider $A$ as a finite-dimensional graded connected cocommutative Hopf algebra with $u$ and $v$ each primitive of degree $1$. Given scalars $\lambda,\mu \in k$ not both $0$, let $A_{\lambda,\mu}$ be the (graded) Hopf subalgebra of $A$ generated by $w_{\lambda,\mu} := \lambda \cdot s + \mu \cdot t$. Then $A_{\lambda,\mu} \cong k[w_{\lambda,\mu}]/\subgrp{w_{\lambda,\mu}^2}$, and one can check that every nontrivial proper graded Hopf subalgebra of $A$ is of the form $A_{\lambda,\mu}$ for some $\lambda,\mu \in k$, with $A_{\lambda,\mu} = A_{\lambda',\mu'}$ if and only if $\lambda \mu' = \lambda' \mu$. In particular, since the field $k$ is infinite, $A$ contains infinitely many distinct proper graded Hopf subalgebras.

The cohomology ring $\Hbul(A_{\lambda,\mu},k)$ can be computed using the free resolution
	\[
	\cdots \rightarrow A_{\lambda,\mu} \rightarrow A_{\lambda,\mu} \rightarrow A_{\lambda,\mu} \rightarrow k,
	\]
in which the rightmost arrow is the augmentation map and the remaining arrows are multiplication by $w_{\lambda,\mu}$. Then $\Hbul(A_{\lambda,\mu},k) \cong k[z_{\lambda,\mu}]$ with $z_{\lambda,\mu}$ of cohomological degree $1$. Specifically, $z_{\lambda,\mu}$ corresponds to the functional $A_{\lambda,\mu} \rightarrow k$ that is linearly dual to $w_{\lambda,\mu}$. Since $A = A_{1,0} \otimes A_{0,1}$, the K\"{u}nneth theorem then gives $\Hbul(A,k) \cong k[x,y]$ with $x$ and $y$ each of cohomological degree $1$, and $x$ and $y$ can be interpreted as the functionals $A \rightarrow k$ that are linearly dual to $u$ and $v$, respectively. Then the restriction map in cohomology $\res_{\lambda,\mu}: \Hbul(A,k) \rightarrow \Hbul(A_{\lambda,\mu},k)$ is given by $\res_{\lambda,\mu}(x) = \lambda \cdot z_{\lambda,\mu}$ and $\res_{\lambda,\mu}(y) = \mu \cdot z_{\lambda,\mu}$.

Now let $\mathcal{C} = \set{k} \cup \set{A_{0,1}} \cup \set{A_{1,\mu}: \mu \in k}$ be the set of all proper graded Hopf subalgebras of $A$, considered as a partially ordered set by subalgebra inclusion. In fact, the only subalgebra inclusions in $\mathcal{C}$ are of the form $k \hookrightarrow B$, so for $i \geq 1$ it follows that
	\[
	\varprojlim_{B \in \mathcal{C}} \opH^i(B,k) = \prod_{B \in \mathcal{C}} \opH^i(B,k).
	\]
Define $z = (z_B)_{B \in \mathcal{C}} \in \prod_{B \in \mathcal{C}} \opH^1(B,k)$ by $z_k = 0$, $z_{A_{0,1}} = z_{0,1}$, and $z_{A_{1,\mu}} = z_{1,\mu}$ for $\mu \in k$. Now it is straightforward to check that no power of $z$ lies in the image of the natural algebra map
	\[
	q: \Hbul(A,k) \rightarrow \varprojlim_{B \in \mathcal{C}} \Hbul(B,k),
	\]
and hence $q$ is not an $F$-surjection.  While we did not look for them, we expect similar examples exist in odd characteristic as well.
\end{example}

\subsection{Example: \texorpdfstring{$\Gaone \times \Gam$}{Ga1 and Ga-}} \label{subsec:examples}

By Lemma \ref{lemma:finiteifffree}, the support theories for $\Gaone$ and $\Gam$ reduce to the type of freeness conditions that one sees already in the classical theory for restricted Lie algebras. The supergroup $\Moneone = \Gaone \times \Gam$ is thus the first example where new purely `super' phenomena emerge and we can no longer expect a freeness condition to suffice.

Throughout this section we write $\Pone = k[u,v]/\subgrp{u^p+v^2}$ as usual, and write $k\Moneone = k[s,t]/\subgrp{s^p,t^2}$ with $\ol{s} = \zero$ and $\ol{t} = \one$. To  simplify matters, \emph{we assume throughout this section that the field $k$ is algebraically closed}, and we work with support varieties rather than with support schemes. Thus when we write $\abs{G}$ or $\Vrg$, we are referring just to the $k$-points in the ambient schemes.

The varieties $\abs{\Moneone}$ and $V_1(\Moneone)$ both identify with the affine space $\{ (\mu,a) \in k^2 \}$, and the homeomorphism $\Psi_1: V_1(\Moneone) \rightarrow \abs{\Moneone}$ is then given by $\Psi_1(\mu,a) = (\mu,a^p)$; cf.\ \cite[Corollary 3.1.2]{Drupieski:2017b}. For $\lambda \in k$, the dilation homomorphism $m_\lambda: H(\Moneone,k) \rightarrow H(\Moneone,k)$ is defined on homogeneous elements by $m_\lambda(z) = \lambda^{\deg(z)} \cdot z$, and then the dilation action of $\lambda$ on $\abs{\Moneone}$ is given by $(\mu,a) \mapsto m_\lambda^*(\mu,a) = (\lambda \mu, \lambda^2 a)$. Thus if $\mu,a \in k$ are not both zero, the affine line in $\abs{\Moneone}$ through the point $(\mu,a)$ is given by $\{(\lambda\mu,\lambda^2 a): \lambda \in k \}$, and two points $(\mu_0,a_0)$ and $(\mu_1,a_1)$ lie on the same affine line in $\abs{\Moneone}$ if and only if $a_0 \mu_1^2 = a_1 \mu_0^2$.

For each nonzero point $(\mu,a) \in \abs{\Moneone}$, define the $k\Moneone$-supermodule $L_{(\mu,a)}$ as follows: A basis for $L_{(\mu,a)}$ is given by the vectors $x_0,x_1,\ldots,x_{p-1}$ of even superdegree and the vectors $y_0,y_1,\ldots,y_{p-1}$ of odd superdegree. Set $x_i = y_i = 0$ for $i \geq p$. Then the action of $s \in k\Moneone$ on $L_{(\mu,a)}$ is defined by
	\[
	s.x_0 = -\mu^2 \cdot x_1, \quad s.x_i = x_{i+1} \text{ for $i \geq 1$,} \quad \text{and} \quad s.y_i = y_{i+1} \text{ for $i \geq 0$,}
	\]
and the action of $t \in \Moneone$ on $L_{(\mu,a)}$ is defined by
	\[
	t.x_0 = a^p \cdot y_{p-1}, \quad t.x_i = 0 \text{ for $i \geq 1$,} \quad \text{and} \quad t.y_i = x_{i+1} \text{ for $i \geq 0$.}
	\]
It is then straightforward to check:
	\begin{itemize}
	\item $L_{(\mu,a)}$ is projective (equivalently, free) over $k\Gaone = k[u]/\subgrp{u^p}$ if and only if $\mu \neq 0$, and
	\item $L_{(\mu,a)}$ is projective (equivalently, free) over $k\Gam = k[v]/\subgrp{v^2}$ if and only if $a \neq 0$.
	\end{itemize}
In particular, if $\mu$ and $a$ are both nonzero, then $L_{(\mu,a)}$ is projective over both $k\Gaone$ and $k\Gam$, but is not projective over the full (local) algebra $k\Moneone = k\Gaone \otimes k\Gam$, because it is not free.

\begin{proposition} \label{prop:M11support}
Fix scalars $\mu,a \in k$ not both zero. Then $\abs{\Moneone}_{L_{(\mu,a^p)}}$ is equal to the affine line in $\abs{\Moneone}$ through the point $(\mu,a^p)$.
\end{proposition}

\begin{proof}
Set $L = L_{(\mu,a^p)}$. By Lemma \ref{lemma:psiinverseheightone}, it is equivalent to show that
	\begin{equation}
	\begin{split}
	V_1(\Moneone)_L &= \set{ (d,c) \in V_1(\Moneone): \text{$(d,c^p)$ lies on the affine line through $(\mu,a^p)$}} \\
	&= \set{ (d,c) \in V_1(\Moneone): a^p d^2 = c^p \mu^2}.
	\end{split}
	\end{equation}
Fix a point $(d,c) \in V_1(\Moneone)$, corresponding to the Hopf superalgebra homomorphism $\phi: \Pone \rightarrow k\Moneone$ such that $\phi(u) = c \cdot s$ and $\phi(v) = d \cdot t$. If $(c,d) = (0,0)$, so that $\phi$ is the trivial homomorphism, then $\pd_{\Pone}(\phi^* L) = \pd_{\Pone}(k) = \infty$, and hence $(0,0) \in V_1(\Moneone)_L$. But $(0,0)$ is trivially on the affine line in $\abs{\Moneone}$ through $(\mu,a^p)$, so assume that $(c,d) \neq (0,0)$. If $d=0$ (but $c \neq 0$), then $\phi$ factors through the canonical quotient map $\Pone \twoheadrightarrow \Pone/\subgrp{v} = k\Gaone$, and it follows from Lemma \ref{lemma:finiteifffree} and the observations preceding the proposition that $(d,c) = (0,c) \in V_1(\Moneone)_L$ if and only if $\mu = 0$, i.e., if and only if the points $(d,c^p) = (0,c^p)$ and $(\mu,a^p)$ lie on the same affine line in $\abs{\Moneone}$. Similarly, if $c = 0$ (but $d \neq 0$), it follows that $(d,c) = (d,0) \in V_1(\Moneone)_L$ if and only if $a = 0$, i.e., if and only if the points $(d,c^p) = (d,0)$ and $(\mu,a^p)$ lie on the same affine line in $\abs{\Moneone}$. Thus we may assume for the rest of the proof that $c$ and $d$ are both nonzero.

Under the assumption that $c$ and $d$ are both nonzero, we will show that $\pd_{\Pone}(\phi^*L) = \infty$ if and only if $a^p d^2 = c^p \mu^2$. Define $\partial: \Pone \oplus \Pi(\Pone) \rightarrow L$ by $\partial(\alpha,\beta^\pi) = \phi(\alpha).x_0 + \phi(\beta).y_0$. Since $c$ and $d$ are both nonzero, $\partial$ is then a surjective even $\Pone$-supermodule homomorphism. As in Section \ref{subsec:projectiveresolutions}, we will write elements of $\Pone \oplus \Pi(\Pone)$ as column vectors, with the top entry an element of $\Pone$ and the bottom entry an element $\Pi(\Pone)$. Also, for improved readability we will omit the superscript $\pi$ from elements of $\Pi(\Pone)$. Then by direct calculation, one can check that $\ker(\partial)$ is generated as a $\Pone$-super\-module by the vectors
	\[
	w_1 = \binom{d \cdot u}{\mu^2 c \cdot v},\, w_2 = \binom{c^{p-1} \cdot v}{-a^p d \cdot u^{p-1}},\, w_3 = \binom{u^p}{0},\, w_4 = \binom{0}{u^p},\, w_5 = \binom{uv}{0},\, w_6 = \binom{0}{u^{p-1}v}.
	\]

Suppose first that $a^p d^2 \neq c^p\mu^2$. We want to show that $\pd_{\Pone}(\phi^*L) < \infty$. Since the matrix of coefficients $\sm{d & c^{p-1} \\ -\mu^2 c & -a^pd}$ has nonzero determinant, it follows that the vectors 
	\[
	v \cdot w_1 = \binom{d \cdot uv}{-\mu^2c \cdot u^p} \quad \text{and} \quad u \cdot w_2 = \binom{c^{p-1} \cdot uv}{-a^pd \cdot u^p}
	\]
are $k$-linearly independent, and hence $w_4$ and $w_5$ are contained in the $\Pone$-supermodule generated by $w_1$ and $w_2$. Similarly, one sees that the vectors
	\[
	u^{p-1} \cdot w_1 = \binom{d \cdot u^p}{\mu^2c \cdot u^{p-1}v} \quad \text{and} \quad v \cdot w_2 = \binom{-c^{p-1} \cdot u^p}{-a^p d \cdot u^{p-1}v}
	\]
are $k$-linearly independent, and deduces that $w_3$ and $w_6$ are also contained in the $\Pone$-supermodule generated by $w_1$ and $w_2$. Then $\ker(\partial)$ is generated as a $\Pone$-supermodule by the even vector $w_1$ and the odd vector $w_2$. We claim that $w_1$ and $w_2$ are $\Pone$-linearly independent. If so, the complex
	\[
	\ker(\partial) \hookrightarrow \Pone \oplus \Pi(\Pone) \stackrel{\partial}{\rightarrow} L \rightarrow 0
	\]
is a finite $\Pone$-free resolution of $L$, showing that $\pd_{\Pone}(\phi^*L) < \infty$, as desired. So suppose
	\[ \textstyle
	\left( \sum_{i \geq 0} \alpha_i u^i + \sum_{i \geq 0} \beta_i u^i v \right) \cdot w_1 + \left( \sum_{i \geq 0} \gamma_i u^i + \sum_{i \geq 0} \delta_i u^i v \right) \cdot w_2 = 0
	\]
for some scalars $\alpha_i,\beta_i,\gamma_i,\delta_i \in k$ (almost all equal to $0$). This dependence relation translates into
	\[
	\binom{\sum (\alpha_{i-1}d-\delta_{i-p}c^{p-1})u^i + \sum (\gamma_i c^{p-1}+\beta_{i-1}d)u^iv}{\sum (\alpha_i \mu^2c-\delta_{i-p+1}a^pd)u^iv -\sum (\beta_{i-p}\mu^2c+\gamma_{i-p+1}a^pd)u^i} = 0,
	\]
where the summations are taken over all $i \in \Z$, and by convention $\alpha_i = \beta_i = \gamma_i = \delta_i = 0$ for $i < 0$. Then by the $k$-linear independence of the set $\{u^i,u^i v: i \in \N \}$, it follows for all $i \in \Z$ that
	\[
	\begin{pmatrix} d & c^{p-1} \\ \mu^2 c & a^pd \end{pmatrix} \begin{pmatrix}\alpha_i \\ -\delta_{i-p+1} \end{pmatrix} = 0 \quad \text{and} \quad \begin{pmatrix} d & c^{p-1} \\ \mu^2 c & a^pd \end{pmatrix} \begin{pmatrix}\beta_{i-1} \\ \gamma_i \end{pmatrix} = 0,
	\]
and hence $\alpha_i = \beta_i = \gamma_i = \delta_i = 0$ for all $i \in \Z$. Thus, $w_1$ and $w_2$ are $\Pone$-linearly independent.

Now suppose that $a^p d^2 = c^p \mu^2$. We want to show that $\pd_{\Pone}(\phi^*L) = \infty$. It follows from Proposition \ref{prop:Poneprojdim} that $\pd_{\Pone}(\phi^*L) \in \set{0,1,\infty}$, so we just have to show that $\pd_{\Pone}(\phi^*L) > 1$. For this, it suffices to show that $\ker(\partial)$ is not projective as a $\Pone$-supermodule. Suppose to the contrary that $\ker(\partial)$ is projective. Set $P_2 = \Pone \oplus \Pi(\Pone) \oplus \Pone \oplus \Pi(\Pone) \oplus \Pone \oplus \Pi(\Pone)$, and define $\partial_1: P_2 \rightarrow \ker(\partial)$ by $\partial_1(\alpha_1,\alpha_2,\alpha_3,\alpha_4,\alpha_5,\alpha_6) = \sum_{i=1}^6 \alpha_i \cdot w_i$. Once again, we have omitted the superscript $\pi$ from the elements of $\Pi(\Pone)$. Since $\ker(\partial)$ is projective by assumption, there exists a $\Pone$-supermodule splitting $\sigma: \ker(\partial) \rightarrow P_2$. Given $w \in \ker(\partial)$, write
	\[
	\sigma(w) = (\sigma_1(w),\sigma_2(w),\sigma_3(w),\sigma_4(w),\sigma_5(w),\sigma_6(w)) \in P_2.
	\]
We now use the fact that the $\Z_2$-grading on $\Pone$ lifts to a nonnegative $\Z$-grading such that $\deg(u) = 2$ and $\deg(v) = p$. Since $\sigma \circ \partial_1(w_1) = w_1$, it follows from considering the $\Z$-degree of elements that $\sigma_1(w)$ must be equal to the scalar $d^{-1} \in k$ plus a sum (perhaps zero) of terms of greater $\Z$-degree in $\Pone$. Now since $\sigma$ is a $\Pone$-supermodule homomorphism, we get
\[
0 = \sigma(0) = \sigma\left( (c^{p-1} \cdot v).w_1 - (d \cdot u).w_2 \right) = (c^{p-1} \cdot v).\sigma(w_1) - (d \cdot u).\sigma(w_2),
\]
and inspecting first coordinates we get $0 = (c^{p-1} \cdot v).\sigma_1(w_1) - (d \cdot u).\sigma_1(w_2)$. The summand of least $\Z$-degree in $(c^{p-1} \cdot v).\sigma_1(w_1)$ is now $c^{p-1} d^{-1} \cdot v$, which is of $\Z$-degree $p$. On the other hand, the least possible odd $\Z$-degree that a nonzero summand of $(d \cdot u).\sigma_1(w_2)$ could have is $p+2$. Thus, $(d \cdot u).\sigma_1(w_2)$ has no summand of $\Z$-degree $p$ to cancel out the summand of $\Z$-degree $p$ in $(c^{p-1} \cdot v).\sigma_1(w_1)$, a contradiction. Therefore $\ker(\partial)$ is not a projective $\Pone$-supermodule, and so $\pd_{\Pone}(\phi^*L) = \infty$, as desired.
\end{proof}

\subsection{Subalgebras of the Steenrod algebra} \label{subsec:steenrod}

For the rest of this section let $A$ be the mod-$p$ Steenrod algebra with scalars extended to $k$. Then $A$ is a graded connected cocommutative Hopf algebra. In \cite{Nakano:1998}, Nakano and Palmieri investigated support varieties for finite-dimensional graded Hopf subalgebras $B$ of $A$. Specifically, they investigated how support varieties for $B$ are related to support varieties for the so-called \emph{quasi-elementary} Hopf subalgebras of $B$. In the rest of this section we describe how the group algebra $k\Moneone$ occurs as a (graded, quasi-elementary) Hopf subalgebra of $A$, and we describe how the calculations of Section \ref{subsec:examples} show that, in general, a set defined in \cite[Theorem 1.1(b)]{Nakano:1998} via restriction to cyclic subalgebras need not give the cohomological support variety for modules over quasi-elementary subalgebras of $A$.

Let $B$ be a finite-dimensional graded Hopf subalgebra of $A$. Then $B$ and its dual algebra $B^\#$ are both graded connected Hopf algebras. Their augmentation ideals are  generated by their elements of nonzero degree, and hence are nilpotent because $B$ and $B^\#$ are both finite-dimensional. By reducing its $\Z$-grading modulo two, $B$ admits the structure of a cocommutative Hopf super\-algebra, so we can write $B = kG$ for some finite $k$-supergroup scheme $G$. Then the preceding discussion implies that $G$ is both infinitesimal and unipotent. Thus the support theory developed in the present paper can be applied to $B$.

The quasi-elementary condition \cite[Definition 2.4]{Nakano:1998} concerns the non-vanishing of certain products in cohomology, and this condition can be verified for $kE$ for each unipotent multiparameter $k$-super\-group scheme $E$.  In particular, $k\Moneone$ occurs as a quasi-elementary Hopf subalgebra of $A$. In the notation of \cite{Adams:1974,Nakano:1998}, let
	\[
	A^{\#gr} \cong k[\xi_1,\xi_2,\xi_3,\ldots,] \otimes E(\tau_0,\tau_1,\tau_2,\ldots)
	\]
be the graded dual of $A$. Then $k\Moneone \cong B$, where
	\[
	B = \left( A^{\#gr} / \subgrp{\xi_1^p,\xi_2,\xi_3,\ldots;\tau_0,\tau_2,\tau_3,\tau_4,\ldots} \right)^{\#gr}.
	\]
As an algebra, $B$ is generated by the commuting elements $P_1^0$ (the functional linearly dual to $\xi_1$) and $Q_1$ (the functional linearly dual to $\tau_1$) subject only to the relations $(P_1^0)^p = 0$ and $(Q_1)^2 = 0$. The $\Z$-grading on $B$ is given by $\deg(P_1^0) = 2p-2$ and $\deg(Q_1) = 2p-1$. Reducing modulo two, $B$ identifies as a Hopf superalgebra with $k\Moneone = k[s,t]/\subgrp{s^p,t^2}$ via $P_1^0 \mapsto s$ and $Q_1 \mapsto t$.  Under this isomorphism, for a $B$-module $M$, \cite[Theorem 1.1(b)]{Nakano:1998} introduces the set
	\begin{equation} \label{eq:NPrankset}
	\set{ x = (a,b) \in k\cdot s \oplus k \cdot t : M|_{\subgrp{x}} \text{ is not free}}.
	\end{equation} 

Our explicit calculations for $\Moneone$ from the previous section allows us to compare \eqref{eq:NPrankset} to the support variety $\abs{\Moneone}_M$. Indeed, take $M$ to be the $\Moneone$-supermodule $L_{(0,1)}$ as defined in Section \ref{subsec:examples}. The $\Z_2$-grading on $L_{(0,1)}$ lifts to a $\Z$-grading such that
	\begin{align*}
	\deg(y_i) &= 1+(2p-2)i & \text{for $i \geq 0$,} \\
	\deg(x_i) &= 1+(2p-2)(i-1)+(2p-1) & \text{for $i \geq 1$, and} \\
	\deg(x_0) &= 1+(2p-2)(p-1)-(2p-1).
	\end{align*}
Then considering $k\Moneone = k[s,t]/\subgrp{s^p,t^2}$ as a $\Z$-graded algebra with $\deg(s) = 2p-2$ and $\deg(t) = 2p-1$, this makes $L_{(0,1)}$ into a graded $k\Moneone$-module. Now for each nonzero point $x = (a,b) \in k\cdot s \oplus k \cdot t \subset k\Moneone$, it is straightforward to check the restriction of $L_{(0,1)}$ to the $k$-subalgebra of $k\Moneone$ generated by $x$ is free if and only if $a=0$. That is, for $M = L_{(0,1)}$ the set \eqref{eq:NPrankset} is equal to the affine space $k \cdot s \oplus k \cdot t$ with the affine line through $s$ removed.  On the other hand, $\abs{\Moneone}_{L_{(0,1)}}$ is an affine line by Proposition \ref{prop:M11support}.

As discussed in the paragraph preceding Example \ref{example:Fsurjectivity}, Theorems \ref{thm:Psirinv} and \ref{theorem:supportvarietyunion} combined with Lemma \ref{lemma:bsvrg}\eqref{item:union} can be interpreted as replacements for the stratification and rank variety results described in \cite[Theorem 1.1]{Nakano:1998}, provided one accepts a slightly broader definition of a \emph{quasi-elementary} Hopf subalgebra.

\makeatletter
\renewcommand*{\@biblabel}[1]{\hfill#1.}
\makeatother

\bibliographystyle{eprintamsplain}
\bibliography{support-schemes-unipotents}

\end{document}